\documentclass[10pt]{amsart}
\allowdisplaybreaks
\usepackage[utf8]{inputenc}
\usepackage[unicode]{hyperref}
\usepackage[nomain,section,sort=def]{glossaries}
\usepackage{glossary-mcols}
\usepackage{mathtools}
\usepackage{graphicx}
\usepackage{url}
\usepackage{enumerate}
\usepackage{tikz}
\usetikzlibrary{intersections, calc}
\usetikzlibrary {decorations.pathmorphing, decorations.pathreplacing, decorations.shapes}
\usepackage[all]{xy}
\usepackage{bbm}  
\usepackage{amssymb}
\usepackage{xcolor}
\makeatletter
\let\save@mathaccent\mathaccent
\newcommand*\if@single[3]{%
  \setbox0\hbox{${\mathaccent"0362{#1}}^H$}%
  \setbox2\hbox{${\mathaccent"0362{\kern0pt#1}}^H$}%
  \ifdim\ht0=\ht2 #3\else #2\fi
  }
\newcommand*\rel@kern[1]{\kern#1\dimexpr\macc@kerna}
\newcommand*\wideaccent[2]{\@ifnextchar^{{\wide@accent{#1}{#2}{0}}}{\wide@accent{#1}{#2}{1}}}
\newcommand*\wide@accent[3]{\if@single{#2}{\wide@accent@{#1}{#2}{#3}{1}}{\wide@accent@{#1}{#2}{#3}{2}}}
\newcommand*\wide@accent@[4]{%
  \begingroup
  \def\mathaccent##1##2{%
    \let\mathaccent\save@mathaccent
    \if#42 \let\macc@nucleus\first@char \fi
    \setbox\z@\hbox{$\macc@style{\macc@nucleus}_{}$}%
    \setbox\tw@\hbox{$\macc@style{\macc@nucleus}{}_{}$}%
    \dimen@\wd\tw@
    \advance\dimen@-\wd\z@
    \divide\dimen@ 3
    \@tempdima\wd\tw@
    \advance\@tempdima-\scriptspace
    \divide\@tempdima 10
    \advance\dimen@-\@tempdima
    \ifdim\dimen@>\z@ \dimen@0pt\fi
    \rel@kern{0.6}\kern-\dimen@
    \if#41
      #1{\rel@kern{-0.6}\kern\dimen@\macc@nucleus\rel@kern{0.4}\kern\dimen@}%
      \advance\dimen@0.4\dimexpr\macc@kerna
      \let\final@kern#3%
      \ifdim\dimen@<\z@ \let\final@kern1\fi
      \if\final@kern1 \kern-\dimen@\fi
    \else
      #1{\rel@kern{-0.6}\kern\dimen@#2}%
    \fi
  }%
  \macc@depth\@ne
  \let\math@bgroup\@empty \let\math@egroup\macc@set@skewchar
  \mathsurround\z@ \frozen@everymath{\mathgroup\macc@group\relax}%
  \macc@set@skewchar\relax
  \let\mathaccentV\macc@nested@a
  \if#41
    \macc@nested@a\relax111{#2}%
  \else
    \def\gobble@till@marker##1\endmarker{}%
    \futurelet\first@char\gobble@till@marker#2\endmarker
    \ifcat\noexpand\first@char A\else
      \def\first@char{}%
    \fi
    \macc@nested@a\relax111{\first@char}%
  \fi
  \endgroup
}
\makeatother

\newcommand\widebar{\wideaccent\overline}

\newtheorem{Thm}{Theorem}[section]
\newtheorem{Prop}[Thm]{Proposition}

\newtheorem{Lem}[Thm]{Lemma}
\theoremstyle{definition}
\newtheorem{case}{Case}
\newtheorem{subcase}{Subcase}[case]
\newtheorem{Rem}[Thm]{Remark}
\newtheorem{Def}[Thm]{Definition} 
\newtheorem{Expl}[Thm]{Example}
\numberwithin{equation}{section}  

\newcommand{\todo}[1]{{\textcolor{red}{#1}}}

\newcommand{\magx}[1]{\textcolor{magenta}{#1}}

\newcommand{\viox}[1]{\textcolor{violet}{#1}}

\newcommand{\ebrace}[1]{\langle #1\rangle}
\newcommand{\eebrace}[1]{\langle\!\langle #1\rangle\!\rangle}

\newcommand{\abs}[1]{\lvert #1\rvert}

\newcommand{\bbe}{\mathbf{b}}
\newcommand{\bd}{\mathbf{d}}
\newcommand{\be}{\mathbf{e}}  
\newcommand{\bg}{\mathbf{g}}
\newcommand{\bh}{\mathbf{h}}
\newcommand{\bl}{\mathbf{l}}
\newcommand{\bo}{\mathbf{o}}
\newcommand{\bp}{\mathbf{p}}
\newcommand{\bq}{\mathbf{q}}
\newcommand{\br}{\mathbf{r}}
\newcommand{\bs}{\mathbf{s}}

\newcommand{\bv}{\mathbf{v}}
\newcommand{\bw}{\mathbf{w}}
\newcommand{\bx}{\mathbf{x}}
\newcommand{\by}{\mathbf{y}}
\newcommand{\bz}{\mathbf{z}}

\newcommand{\tbq}{\widetilde{\bq}}
\newcommand{\tbs}{\widetilde{\bs}}

\newcommand{\hbp}{\hat{\bp}}
\newcommand{\hbq}{\hat{\bq}}
\newcommand{\hbr}{\hat{\br}}

\newcommand{\cA}{\mathcal{A}} 

\newcommand{\cH}{\mathcal{H}}
\newcommand{\cI}{\mathcal{I}}
\newcommand{\cK}{\mathcal{K}}
\newcommand{\cL}{\mathcal{L}}
\newcommand{\cC}{\mathcal{C}}

\newcommand{\cP}{\mathcal{P}}
\newcommand{\cQ}{\mathcal{Q}}
\newcommand{\cR}{\mathcal{R}}

\newcommand{\cX}{\mathcal{X}}

\newcommand{\CC}{\mathbb{C}}     

\newcommand{\NN}{\mathbb{N}}
\newcommand{\QQ}{\mathbb{Q}}
\newcommand{\ZZ}{\mathbb{Z}}
\newcommand{\II}{\mathbbm{1}}

\newcommand{\alp}{\alpha}      
\newcommand{\bet}{\beta}
\newcommand{\gam}{\gamma}      
\newcommand{\del}{\delta}
\newcommand{\eps}{\varepsilon} 
\newcommand{\tht}{\theta}
\newcommand{\zet}{\zeta}

\newcommand{\lam}{\lambda}      
\newcommand{\sig}{\sigma}
\newcommand{\Del}{\Delta}      
\newcommand{\Gam}{\Gamma}      

\newcommand{\Sig}{\Sigma}      
\newcommand{\bSig}{\mathbf{\Sig}}      

\newcommand{\df}{\colon}       
\newcommand{\ra}{\rightarrow}

\newcommand{\sfA}{\mathsf{A}}
\newcommand{\sfD}{\mathsf{D}}
\newcommand{\sftA}{\widetilde{\sfA}}
\newcommand{\sftD}{\widetilde{\sfD}}
\newcommand{\sfP}{\mathsf{P}}

\newcommand{\us}{\underline{s}}
\newcommand{\ut}{\underline{t}}

\newcommand{\tbx}{{\widetilde{\bx}}}
\newcommand{\tby}{{\widetilde{\by}}}

\newcommand{\tQ}{\widetilde{Q}}
\newcommand{\tq}{\widetilde{q}}

\newcommand{\uQ}{{\underline{Q}}}
\newcommand{\ord}{{\mathrm{ord}}}
\newcommand{\spe}{{\mathrm{sp}}}
\newcommand{\sym}{{\mathrm{sym}}}
\newcommand{\Qsp}{{Q_1^\spe}}
\newcommand{\Qord}{{Q_1^\ord}}
\newcommand{\Qspv}{{Q_0^\spe}}
\newcommand{\Qordv}{{Q_0^\ord}}
\newcommand{\huQ}{{\widehat{\uQ}}}
\newcommand{\hQ}{\widehat{Q}}
\newcommand{\hQord}{\hQ_1^{\mathrm{ord}}}
\newcommand{\hQsp}{\hQ_1^{\mathrm{sp}}}
\newcommand{\tuQ}{\widetilde{\uQ}}
\newcommand{\uH}{\underline{H}}

\newcommand{\Hsp}{{H_1^{\mathrm{sp}}}}

\newcommand{\St}{\operatorname{St}}
\newcommand{\DSt}{\operatorname{DSt}}
\newcommand{\Ba}{\operatorname{Ba}}
\newcommand{\pBa}{\operatorname{pBa}}
\newcommand{\Adm}{\operatorname{Adm}}
\newcommand{\DAdm}{\operatorname{DAdm}}
\newcommand{\tAdm}{\widetilde{\Adm}}
\newcommand{\AdmSt}{\operatorname{AdmSt}}
\newcommand{\AdmBa}{\operatorname{AdmBa}}
\newcommand{\NeSi}{\operatorname{NeSi}}
\newcommand{\Diag}{\operatorname{Diag}}
\newcommand{\type}{\operatorname{type}}
\newcommand{\inj}{{\operatorname{inj}}}
\newcommand{\inc}{{\operatorname{in}}}
\newcommand{\out}{{\operatorname{out}}}
\newcommand{\Ker}{\operatorname{ker}}
\newcommand{\CoKer}{\operatorname{coker}}

\newcommand{\wt}{\operatorname{wt}}
\newcommand{\KRS}{\operatorname{KRS}}
\newcommand{\MSW}{\operatorname{Lam}}

\newcommand{\Ka}{\Bbbk}

\newcommand{\Hom}{\operatorname{Hom}}
\newcommand{\Ext}{\operatorname{Ext}}
\newcommand{\Ima}{\operatorname{Im}}
\newcommand{\GL}{\operatorname{GL}}

\newcommand{\DecIrr}{\operatorname{DecIrr}}
\newcommand{\DecRep}{\operatorname{DecRep}}
\newcommand{\Irr}{\operatorname{Irr}}

\newcommand{\Supp}{\operatorname{Supp}}

\newcommand{\codim}{\operatorname{codim}}
\newcommand{\kar}{\operatorname{char}}
\newcommand{\op}{{\operatorname{op}}}      

\newcommand{\id}{\operatorname{id}}

\newcommand{\ind}{{\mathrm{ind}}}

\newcommand{\dimv}{\underline{\dim}}         

\newcommand{\rep}{\operatorname{Rep}}
\newcommand{\cyc}{{\operatorname{cyc}}}
\newcommand{\sgn}{\operatorname{sgn}}

\newcommand{\Su}{\Sigma}       
\newcommand{\bSu}{\mathbf{\Sigma}}       
\newcommand{\Pu}{\mathbb{P}}   
\newcommand{\Ma}{\mathbb{M}}   
\newcommand{\dSu}{\partial\Su} 
\newcommand{\Sh}{\operatorname{Sh}}
\newcommand{\cLC}{\mathcal{LC}}
\newcommand{\cAC}{\mathcal{AC}}
\newcommand{\cAL}{\mathcal{AL}}
\newcommand{\tcAC}{\widetilde{\cAC}}
\newcommand{\tcAL}{\widetilde{\cAL}}
\newcommand{\tcLC}{\widetilde{\cLC}}

\newcommand{\piorb}{\pi_{1,\Pu}^{\mathrm{orb}}}
\newcommand{\piorbfree}{\pi_{1,\Pu}^{\mathrm{orb, free}}}
\newcommand{\piorbfreepr}{\pi_{1,\Pu,\mathrm{prim.}}^{\mathrm{orb, free}}}

\newcommand{\intn}{\operatorname{int}}
\newcommand{\Intn}{\operatorname{Int}}
\newcommand{\fra}{\mathfrak{a}}

\setglossarystyle{mcolindex}

\newglossary[sy-glg]{symbols}{sy-gls}{sy-glo}{Symbols}
\newglossary[kw-glg]{keywords}{kw-gls}{kw-glo}{Keywords}

\makeglossaries

\newcommand*{\newsymb}[3][]{%
  \newglossaryentry{#2}{type=symbols,%
    name={$#3$},text={#3},description={},#1}%
}

\newcommand*{\newkywd}[3][]{%
  \newglossaryentry{#2}{type=keywords,%
    name={$#3$},text={#3},description={},#1}%
}

\newsymb{gvecA}{\mathbf{g}_A}
\newsymb{EinvA}{E_A}
\newsymb{eAX}{e_A(X)}
\newsymb{eAXY}{e_A(X,Y)}
\newsymb{gAX}{\bg_A(X)}
\newsymb{KRSCe}{\KRS(C,e)}
\newsymb{DecIrrt}{\DecIrr^\tau}
\newsymb{tmuk}{\tilde{\mu}_k}
\newsymb{ust}{\us, \ut}
\newsymb{Qspvo}{\Qspv, \Qordv}
\newsymb{KauQ}{\Ka\uQ}
\newsymb{tQo}{\tQ_0}
\newsymb{S(i)}{S(i)}
\newsymb{bSu}{\bSu}
\newsymb{cL(uQ)}{\cL(\uQ)}
\newsymb{IIirho}{\II_{(i,\rho)}, \II_{(i,\rho)}^{-1}}
\newsymb{<}{<}
\newsymb{wop}{\bw_{[k]}, \bw^{[k]}, \bw^{-1}, \us(\bw), \ut(\bw), l(\bw)}
\newsymb{StrQ}{\St(\uQ), \St_{(i,\rho)}(\uQ)}
\newsymb{tbsi}{\tbs_i}
\newsymb{tbpq}{\tilde{\bp}_{(i,\rho)}, \tilde{\bq}_{(i,\rho)}}
\newsymb{tcPI}{\widetilde{\cP}_\uQ, \widetilde{\cI}_\uQ }  
\newsymb{sim}{\sim}
\newsymb{BaQ}{\Ba(\uQ), \pBa(\uQ), \pBa'(\uQ)}
\newsymb{huQ}{\huQ}
\newsymb{FQ}{F}
\newsymb{prec}{\prec}
\newsymb{bsi}{\bs_i}
\newsymb{A(bw)}{A(\bw)}
\newsymb{cpl}{\overline{?}}
\newsymb{AdmSBQ}{\AdmSt(\uQ), \AdmBa(\uQ), [\AdmBa(\uQ)]}
\newsymb{AdmQ}{\Adm(\uQ), \Adm_s(\uQ), \Adm_b(\uQ), [\Adm_b(\uQ)]}
\newsymb{GHx}{G_\bx, \uH(\bx)}
\newsymb{cHQxy}{\cH_\uQ(\bx,\by)}
\newsymb{cKQxy}{\cK_\uQ(\bx,\by), \cK_\uQ(\bx,\by)_\sfA}
\newsymb{APQfxy}{\sfA_{\uQ^f}(\bx,\by), \sfP_\uQ(\bx,\by), \Diag_b(\bx,\by)}
\newsymb{ADx}{A_i^+(\bx), A_i^-(\bx), D_i^+(\bx), D_i^-(\bx)}
\newsymb{S(x)}{S(\bx)}
\newsymb{wtix}{\wt(s), s^\iota, s^\chi}
\newsymb{-bsipm}{-\bs^\pm_i}
\newsymb{DAdm*Q}{\Adm^*(\uQ), \DAdm^*(\uQ), [\DAdm^*(\uQ)]}
\newsymb{bd(xs)}{\bd(\bx,s)}
\newsymb{cegQ}{\be_\uQ, \bg_\uQ}
\newsymb{d_i}{d_2, d_3}
\newsymb{gam-1}{\gam^{-1}}
\newsymb{fgpoid}{\pi_1(\Su\setminus\Pu, E)}
\newsymb{forgpoid}{\piorb(\Su, E), \piorb(\Su, E)^*, \piorb(\Su,E)}
\newsymb{bargam}{\bar{\gamma}}
\newsymb{tcALC}{\tcAC(\bSu),\tcAL(\bSu),\tcLC(\bSu),\tcAC(\bSu)_b,\tcLC(\bSu)_s}
\newsymb{cALC}{\cAC(\bSu),\cLC(\bSu),\cAC_b(\bSu),\cLC_s(\bSu),\cLC_b(\bSu)}
\newsymb{tcAL}{\tcAL(\bSu)_\sym, \cAL(\bSu)}
\newsymb{S(gam)}{S(\gam)}
\newsymb{mtALC}{\cLC^*(\bSu), \cAC^{\bowtie}(\bSu)}
\newsymb{preInt}{\sfP_{\bSu}(\gam_1,\gam_2), \Diag_b(\gam_1,\gam_2)}
\newsymb{Intn*}{\Intn^*((\gam,s), (\lam,t))}
\newsymb{LamSu}{\MSW(\bSu)}
\newsymb{tArc}{\cAC^{\bowtie}_\tau(\bSu), \cAC^{\bowtie}_\tau(\bSu)/_{\simeq}}
\newsymb{Q(T)f}{\uQ(T), \uQ^f(T)}
\newsymb{Tro}{T^{\circ}, \Gam(T^{\circ})}
\newsymb{barr}{\overline{r}\df\piorb(\Su, E_T)\xrightarrow{\sim} \piorb(\Gam(T^{\circ}), E_T)}
\newsymb{frasb}{\fra_s, \fra_b, \fra_f}
\newsymb{rho12}{\rho^{1/2}}

\newkywd{pamo}{partial\ monoid}
\newkywd{paKRS}{partial\ {KRS-monoid}}
\newkywd{taure}{generically\ \mbox{$\tau$}-regular}
\newkywd{skgpQ}{skewed-gentle\ polarized\ quiver}
\newkywd{heredpolQ}{hereditary\ polarized\ quiver}
\newkywd{fringing}{fringing}
\newkywd{fringevertices}{fringe\ vertices}
\newkywd{fringearrow}{fringe\ arrow}
\newkywd{word}{word}
\newkywd{inextensible}{inextensible}
\newkywd{string}{string}
\newkywd{symmetric}{symmetric}
\newkywd{simplestring}{simple\ string}
\newkywd{band}{band}
\newkywd{primitive}{primitive}
\newkywd{rotations}{rotations}
\newkywd{equivalent}{equivalent}
\newkywd{puncturedletters}{punctured\ letters}
\newkywd{type}{type}
\newkywd{completion}{completion}
\newkywd{admissible}{admissible}
\newkywd{strict}{strict}
\newkywd{morphism}{morphism}
\newkywd{winding}{winding}
\newkywd{bvertex}{boundary\ vertex, polarized}
\newkywd{H-triple}{H-triple}
\newkywd{K-triple}{K-triple}
\newkywd{negativesimple}{negative\ simple}
\newkywd{curve}{curve}
\newkywd{loop}{loop}
\newkywd{fundamentalgroupoid}{fundamental\ groupoid}
\newkywd{fundorbgroupoid}{fundamental\ orbifold\  groupoid}
\newkywd{freeorbifoldfundamentalgroup}{free\ orbifold\ fundamental\ group}
\newkywd{mtcurve}{marked\ curve, tagged\ curve}
\newkywd{internalintersectionnumber}{internal\ intersection\ number}
\newkywd{markedintersectionnumber}{marked\ intersection\ number}
\newkywd{lamination}{lamination}
\newkywd{taggedarc}{tagged\ arc}
\newkywd{taggedtriangulation}{tagged\ triangulation}
\newkywd{sign0}{signature\ zero\ tagged\ triangulation}
\newkywd{associatedidealtriangulation}{associated\ ideal\ triangulation}
\newkywd{dualgraph}{dual\ graph}
\title[Laminations as $\tau$-regular components]{Laminations of punctured surfaces \\ as $\tau$-regular irreducible components}
\author{Christof Geiß}
\author{Daniel Labardini-Fragoso}
\address{Instituto de Matemáticas, UNAM, Ciudad Universitaria,
  04510 CDMX, México}
\email{christof.geiss@im.unam.mx, labardini@im.unam.mx}
\author{Jon Wilson}
\address{Jeremiah Horrocks Institute, University of Central Lancashire,
  Preston PR1 2HE, UK}
\email{jwilson30@uclan.ac.uk}

\date{September 10, 2025}

\begin{document}

\begin{center}
\emph{Dedicated to Prof. Claus Michael Ringel on the occasion of his 80th birthday}
\end{center}

\begin{abstract}
Let $\bSu:=(\Su,\Ma,\Pu)$ be a surface with marked points 
$\Ma\subset\dSu\neq\emptyset$ on the boundary, and punctures 
$\Pu\subset\Su\setminus\dSu$,  and $T$ an arbitrary tagged triangulation of 
$\bSu$ in the sense of Fomin--Shapiro--Thurston.  
The Jacobian algebra $A(T):=\cP(Q(T), W(T))$ corresponding to the non-degenerate potential $W(T)$ of $T$ defined by Cerulli Irelli and the second author is tame, as shown by Schröer and the first two authors. 
In this paper, we show that there is a natural isomorphism
$\pi_T\df\MSW(\bSu)\ra\DecIrr^\tau(A(T))$ of tame partial KRS-monoids that intertwines dual shear coordinates with respect to $T$, and generic 
$g$-vectors of irreducible components.
Here, $\MSW(\bSu)$ is the set of laminations of $\bSu$, considered by Musiker--Schiffler--Williams, 
with the disjoint union of non-intersecting laminations as partial monoid operation. 
On the other hand, $\DecIrr^\tau(A(T))$ denotes the set of generically 
$\tau$-regular irreducible components of the decorated representation varieties of $A(T )$, with the direct sum of generically $E$-orthogonal irreducible components as partial monoid operation, where $E$ is the symmetrized $E$-invariant of Derksen--Weyman--Zelevinsky, 
$E(-,\bullet)=\dim\Hom_{A(T)}(-,\tau(\bullet))+\dim\Hom_{A(T)}(\bullet,\tau(-))$.
\end{abstract}
\maketitle

{\tiny \tableofcontents}

\section{Introduction}
Let $\bSu:=(\Su, \Ma, \Pu)$ be a surface with marked points and
non-empty boundary. Thus, $\Su$ is a compact, connected oriented
Riemann surface of genus $g$, with boundary 
$\dSu=S_1\cup S_2\cup\cdots\cup S_b$
consisting of $b\geq 1$ connected components $S_1,\ldots,S_b$, a set $\Ma\subset\dSu$ of marked points on the boundary such that $\Ma \cap S_i \neq \emptyset$ for every $i \in \{1,\ldots, b\}$, and a set of punctures
$\Pu\subset\Su\setminus\dSu$. 
We  assume  moreover that the rank  
$n(\bSu):=6(g-1) +3(b+\abs{\Pu})+\abs{\Ma}$
is  positive.   

\subsection{Main Result}
Let $T$ be a tagged triangulation of $\bSu$ with quiver $Q(T)$. If $\bSu$ is not the unpunctured torus with exactly one marked point, then the potential 
$W(T)\in\Ka\ebrace{(Q(T)}_\cyc$ defined in \cite{CILF12,LF16} is, up to weak right equivalence, the unique non-degenerate potential on $Q(T)$ \cite{GLFS16}. 
Furthermore, the Jacobian algebra $A(T):=\cP_\Ka(Q(T), W(T))$
is tame \cite{GLFS16}.  In fact, since 
$\Ma\neq\emptyset$, the surface $\bSig$ admits a tagged triangulation $T'$ of signature 0, and thus $A(T')$ is skewed-gentle. As a consequence, we have a parametrization of the indecomposable objects of the corresponding cluster category.
Note, however, that for  $\Pu\neq\emptyset$, the algebra $A(T)$ fails to be gentle (except only for the once-punctured triangle). 
It even fails to be skewed-gentle for most triangulations.  
So, in general, an explicit  description of the indecomposable representations of $A(T)$ is not known,
 even though these representations may be in principle be parametrized via the cluster category.
See~\cite{Dom17} or~\cite{LF10} for partial results in this direction, which illustrate the complexity of the situation.  
In this paper, we show that the set 
 $\DecIrr^\tau(A(T))$ of generically $\tau$-regular irreducible components of the representation spaces of $A(T)$ can be naturally parameterized by the 
 set $\MSW(\bSig)$ of laminations on $\bSu$. 
Our main result is  a more precise version of this statement:

\begin{Thm} \label{thm1}
Let $\bSu$ be a surface with $\dSu\neq\emptyset$. For each  
tagged triangulation $T$ with Jacobian algebra $A(T)$ corresponding to the above-mentioned potential $W(T)$, there is a unique isomorphism of tame partial KRS-monoids
\[
  \pi_T\df (\MSW(\bSu),\Intn^*,+) \ra
  (\DecIrr^\tau(A(T)), e_{A(T)},\oplus), 
\]
which intertwines dual shear coordinates with respect to $T$ and generic projective $g$-vectors with respect to $A(T)$.
\end{Thm}

This  extends considerably a similar result for
the special case of surfaces without punctures, which was settled recently by two of the authors together with J. Schröer in~\cite[Thm.~1.7]{GLFS22}. 
In that case the algebras involved are gentle and the corresponding well-known classifications of their indecomposable representations and morphisms of representations were heavily exploited by the  authors.

 Theorem~\ref{thm1} is crucial for our follow-up work~\cite{GLFW23}, where we identify for surface cluster algebras with $\abs{\Ma}\geq 2$
the generic basis with the bangle functions  from~\cite{MSW13}. 

\subsection{Partial KRS-monoids}
The quite simple notion of (tame) \emph{partial KRS-monoid} will be explained in Section~\ref{ssec:KRS-Def}.  As mentioned in the abstract,  $\MSW(\bSu)$ is the set of laminations  of $\bSu$ considered by Musiker--Schiffler--Williams in~\cite{MSW13}, who denoted it by $\mathcal{C}^o$.  In \emph{loc.~cit.} the authors introduced bangle functions, and used the set $\MSW(\bSu)$ to parameterize those functions 
in a natural way. The disjoint union of non-intersecting (up to homotopy) laminations
as a partial monoid operation provides $\MSW(\bSu)$ with a structure of
a partial KRS-monoid, see Section~\ref{ssec:taggedC} for details.  
On the other hand, $\DecIrr^\tau(A)$ is the set of decorated, generically
$\tau$-regular irreducible components of the representation varieties 
of a basic $\Ka$-algebra $A$.  Previously those components were
called strongly reduced, later on the name generically $\tau$-reduced.
In a very recent preprint~\cite{BS25}, Bobiński and Schröer make a very good
point why the name $\tau$-regular is better suited. The concept
 was introduced in this form
in~\cite{CLS15}, inspired by~\cite{GLS12} and~\cite{Pl13}, in order to parameterize in a natural way the set of generic Caldero--Chapoton functions.  
The set~$\DecIrr^\tau(A)$ carries also a natural structure of partial KRS-monoid, which basically reflects the canonical decomposition of each element of~$\DecIrr^\tau(A)$ into \emph{indecomposable} generically $\tau$-regular irreducible components. 
This involves the results of several authors, see Section~\ref{ssec:FD-KRS} for details.  

\subsection{Proof strategy for the main result}
Our approach to prove Theorem~\ref{thm1} consists of two steps. We first prove it for the
special case when $T$ is a signature zero triangulation of $\bSu$, in which case the Jacobian
algebra $A(T)$ is a skewed-gentle algebra in the sense of~\cite{GP99}, see Theorem~\ref{thm:sign0}.
This step occupies the bulk of the present paper, and is based on ideas from~\cite{QZ17}, recent progress on the description of homomorphisms between finite dimensional representations of skewed-gentle algebras~\cite{Ge23}, and a careful analysis of the resolution of kinks for curves on surfaces~\cite{GLF23}.

In the second step, we extend Theorem~\ref{thm:sign0} by a mutation argument to the general case. More precisely, we study the precise behaviour of decorated, generically $\tau$-regular irreducible components under QP-mutations. See Proposition~\ref{prp:JaMut}, which is a slight sharpening of~\cite[Thm.~1.10]{GLFS20}.  

\subsection{Skewed-gentle algebras}
Inspired by the notion of gentle algebras in~\cite{AS87}, skewed-gentle algebras were introduced in~\cite{GP99} as a special case of  Crawley-Boevey's clannish algebras~\cite{CB89b}. 
In particular, the description of their indecomposable representations is governed by a complicated combinatorics of strings and bands. 
Skewed-gentle algebras are derived tame, since (by construction) their repetitive algebra is essentially clannish. 
However, it is easy to see that the class of skewed-gentle algebras is not closed under derived equivalence.
Inspired by the works of Crawley-Boevey~\cite{CB89a} and Krause~\cite{Kr91} on the representations of special biserial algebras, in~\cite{Ge99} the study of homomorphisms between finite dimensional, indecomposable representations of clans and clannish algebras  was started. 
Unfortunately, in~\cite{Ge99}  the homomorphisms between representations which involved  bands could not be described satisfactorily.  
For this reason, for example in~\cite{QZ17}, and more recently in~\cite{HZZ22} the exchange graph of support $\tau$-tilting
modules of (Jacobian) skewed-gentle algebras could be described combinatorially, but not the full set of decorated generically 
$\tau$-regular irreducible components of the corresponding representation varieties. 
Recently, in~\cite{Ge23} the description of homomorphisms, initiated in~\cite{Ge99} improved. 
Namely, for skewed-gentle algebras $A$  over a field 
 $\Ka$ with $\kar( \Ka)\neq 2$, a complete description of the homomorphisms between indecomposable representations was achieved.
As an application, in~\cite{Ge23} a description was given of \emph{all} indecomposable, generically $\tau$-regular irreducible components of the representation varieties of
$A$, together with the generic value of the symmetrized 
$E$-invariant in terms of strings and bands.  
We review this result in Theorem~\ref{thm:SkewedGIso}, using the terminology of (tame) partial KRS-monoids. 

\subsection{Signature zero tagged triangulations} \label{ssec:intro-sign0}
It is easy to see that in the situation of Theorem~\ref{thm1}, the surface $\bSu$ admits a tagged triangulation $T$ of signature zero in the sense of~\cite{FST08}, 
and that for these triangulations, the Jacobian algebra $A(T)$
 is skewed-gentle in the sense of~\cite{GP99}. 
 Such triangulations were called ``admissible'' in~\cite{QZ17}, see also Section~\ref{ssec:tagged3ang} below for details.
As observed by Qiu and Zhou in~\cite{QZ17}, 
there is a close connection between Crawley-Boevey's  description of the indecomposable representations of the Jacobian skewed-gentle (hence clannish) algebra $A(T)$ in terms of strings and bands on the one hand, 
and certain curves ``without kinks'' on $\bSu$ on the other  side (see \cite{GLF23}, where the uniqueness of the resolutions of kinks of a curve is established).  
Thus, each  curve has, up to 2-orbifold homotopy, a unique representative without kinks.
With the help of~\cite{Ge99}, Qiu and Zhou explained in \emph{loc.~cit.}~that for ``tagged curves'' %
\emph{with endpoints in $\Ma\cup\Pu$}, a suitably defined
intersection number can be interpreted as the symmetrized $E$-invariant
$E_{A(T)}(M, N):=\dim\Hom_{A(T)}(M,\tau N)+\dim\Hom_{A(T)}(N,\tau M)$  between the corresponding (decorated) representations.  
From this, they established a natural isomorphism between the exchange graph of support $\tau$-tilting modules and the exchange graph of tagged triangulations.
Roughly speaking, establishing Theorem~\ref{thm:sign0} means to extend Qiu--Zhou's result from the set of $\tau$-rigid components, i.e., the set of $\tau$-regular components with generic $E$-invariant zero, to the whole set of $\tau$-regular components. 

More precisely, we establish here the expected natural bijection by extending the correspondence 
\emph{intersection number $\longleftrightarrow$ combinatorial $E$-invariant}  to a large enough class of indecomposable representations, resp.~ marked, possibly closed,  curves.
Moreover, we show that this bijection intertwines the generic $g$-vector of a representation with the shear coordinates of the corresponding laminate, with respect to the triangulation $T$. 
We take the opportunity to review carefully quite delicate arguments from~\cite{QZ17}. 
For this, we found the ideas of Amiot and Plamondon~\cite{AP17} concerning orbifold homotopy, and the notion of kisses between strings from~\cite{BY20} very inspiring. 

\subsection{Behaviour of \texorpdfstring{$\DecIrr^\tau(A(T))$}{DecIrrτ(A(T))} under mutation}
\label{ssec:intro-transf}
Having established Theorem~\ref{thm1} for tagged triangulations of signature zero, we show that the change undergone by the dual shear coordinates under an arbitrary flip of the reference triangulation $T$, and the change undergone by the generic (projective) $g$-vectors under the mutation of $\tau$-regular components defined in \cite{GLFS20} when the reference quiver with potential is mutated, obey the same recursive formulae. This allows us to extend Theorem \ref{thm1} to arbitrary tagged triangulations, since any two tagged triangulations can be obtained from one another by a finite sequence of flips for surfaces with non-empty boundary. To be more precise, consider an arbitrary triangulation $T$ of $\bSu$ and a flip of the
arc $k$.  By Proposition~\ref{prp:JaMut} the corresponding QP-mutation $\mu_k(Q(T),W(T))$
induces an isomorphism of partial KRS-monoids
\[
\tilde{\mu}_k\df\DecIrr^\tau(\cP_\Ka(Q(T),W(T)))\ra
\DecIrr^\tau(\cP_\Ka(\mu_k(Q(T),W(T)))),
\]
under which the generic $g$-vector of each component $Z\in\DecIrr^\tau(\cP_\Ka(Q(T),W(T)))$
transforms according to the rule~\eqref{eq:gproj-transf} for $g$-vectors under change of the reference cluster, see~\cite[(7.18)]{CA4}.

\section{Tame partial KRS-monoids and preliminary results}
\label{sec:KRS}
\subsection{Basic definitions and examples} \label{ssec:KRS-Def}
We introduce the quite simple notion of partial 
Krull--Remak--Schmidt monoids as a useful framework for our findings.

\begin{Def} 
A \emph{\gls{pamo}} is a triple $(X, e, \oplus)$, where  $X$ is a set, 
$e\df X\times X\ra\NN$ is a symmetric function, and
  $\oplus: \{(x,y)\in X\times X\mid e(x,y)=0\}\ra X$ is a partial sum with the following properties:
  \begin{itemize}
  \item[(s)] If $e(x,y)=0$ we have $x\oplus y =y\oplus x$,  
  \item[($0$)]  There exists a unique element $0\in X$ with $e(0,x)=0$
    and  $0\oplus x=x$ for all $x\in X$,
  \item[(d)] If $e(y,z)=0$ we have 
  $e(x, y\oplus z)=e(x,y)+e(x,z)$ for all $x\in  X$,
  \item[(a)]  $(x\oplus y)\oplus z= x\oplus (y\oplus z)$ whenever  one side of the equation is defined.
  \end{itemize}  
A \emph{morphism} between partial monoids $X=(X, e,\oplus)$
and $X'=(X',e',\oplus')$ is a map $f\df X\ra  X'$ of sets, such that
$e'(f(x), f(y))=e(x,y)$ for all $x,y\in X$ and
$f(x\oplus y)= f(x)\oplus' f(y)$ when $e(x,y)=0$.
\end{Def}

\begin{Rem} 
(1) Suppose $(x_1\oplus x_2)\oplus x_3$ is defined.  
Then we have in view of $(d)$ already $e(x_i, x_j)$ for all $i<j$ and thus
$x_1\oplus (x_2\oplus x_3)$ is defined.

(2)  Suppose that we have $x_1, x_2, \ldots, x_n\in X$ with $e(x_i, x_j)=0$ for all $i<j$,  
then $x_1\oplus x_2\oplus\cdots\oplus x_n\in X$ is well-defined,
and  for each permutation $\sig\in\mathfrak{S}_n$ we have
  \[
    x_1\oplus x_2\oplus\cdots\oplus x_n=
    x_{\sig(1)}\oplus x_{\sig(2)}\oplus\cdots\oplus x_{\sig(n)}.
  \]
(3) We are mainly interested in isomorphisms of partial monoids. 
\end{Rem}  

\begin{Def} \label{def:KRS2}
Let   $X=(X, e, \oplus)$ be a partial monoid.
\begin{itemize}
\item
We denote by
  \[
    X_\ind:=\{x\in X\setminus\{0\}
    \mid x= y\oplus z \text{ implies } y=0 \text{ or } z=0\}
  \]
the set of \emph{indecomposable elements} of $X$.
\item
We say that $X$ is a \emph{\gls{paKRS}} 
if each $x\in X$ is a (finite) direct sum of indecomposable elements, and whenever
\[
  x_1\oplus x_2\oplus\cdots\oplus x_m= y_1\oplus y_2\oplus\cdots\oplus y_n
\]
for $x_1, \ldots, x_m, y_1, \ldots, y_n\in X_\ind$, 
then $m=n$ and there exists a permutation 
$\sig\in\mathfrak{S}_n$ with 
$y_i=x_{\sig(i)}$ for all $i=1,2,\ldots, n$.
\item
We say that $X$ is \emph{tame}, if $e(x,x)=0$ for all $x\in X$.
\item
 A (injective) map $\bg\df X\ra\ZZ^n$ for some $n$ 
  is a \emph{(faithful) framing}, if
$\bg(x\oplus y)=\bg(x)+\bg(y)$ for all
$x, y\in  X$ with $e(x,y)=0$.   
\item
A framed partial monoid $X=(X, e, \oplus,\bg)$ is 
\emph{free of rank } $n$ if $\bg\df X\ra\ZZ^n$ is bijective.  
\end{itemize}  
\end{Def}

\begin{Expl} \label{expl:KRS1}
Let $C$ be a set equipped with a  symmetric function $e\df C\times C\ra\NN$ such that $e(c,c)=0$ for all $c\in C$. Then
\[
  \gls{KRSCe}:=\{f\df C\ra\NN\mid c_1, c_2\in\Supp(f)\Rightarrow e(c_1, c_2)=0
  \text{ and } \abs{\Supp(f)}<\infty \}
\]
is a tame partial KRS-monoid with
\[
e(f,g):= \sum_{c,d\in C} f(c)\cdot g(d)\cdot e(c,d),
\]
and
\[
  (f\oplus g)(c):= f(c)+g(c) \quad\text{for all } c\in C.
\]
Note, that
\[
  \KRS(C,e)_\ind=\{\del_c\mid c\in C\}\quad\text{where}\quad
  \del_c(d):=\begin{cases} 1 &\text{if } d=c,\\ 0 &\text{else.} 
  \end{cases}
\]
With this notation we have 
$f=\oplus_{c\in C} \del_c^{\oplus f(c)}$ for all
$f\in\KRS(C,e)$.  
Note, that for a tame partial KRS-monoid 
$X=(X, e, \oplus)$ we have $X\cong\KRS(X_\ind, e')$, 
where $e'$ is the restriction of $e$ to $X_\ind\times X_\ind$. 
In this situation, any map $\bg\df C\ra\ZZ^n$ can be extended to a framing
$\bg\df\KRS(C, e)\ra\ZZ^n$ by defining simply
$\bg(f):=\sum_{c\in C} f(c)\bg(c)$. 
\end{Expl}

\subsection{Generically \texorpdfstring{$\tau$}{τ}-regular irreducible components} \label{ssec:gen-taured}
Let $A$ be a (basic) finite dimensional $\Ka$-algebra with
$\Ka$ an algebraically closed field,  and denote by $Q$ the Gabriel quiver of $A$. Thus, we may write $A=\Ka Q/I$ for 
an admissible ideal $I\subset \Ka Q$, and we may identify the Grothendieck group  $K_0(A)$ of $A$ with  $\ZZ^{Q_0}$. Recall that in this
situation $K_0(A)$ has a natural $\ZZ$-basis given by the simple representations $S_i$ for $i\in Q_0$. 

Each finite dimensional representation of $A$ is of  the form
$M=((M_i)_{i\in Q_0}, (M(\alp)_{\alp\in Q_1})$ with 
$M_i\cong\Ka^{m_i}$ and 
$M(\alp)\in \Hom_\Ka(M_{s(\alp)}, M_{t(\alp)})$. In this case,
$\dimv(M)=(m_i)_{i\in Q_0}$ is the \emph{dimension vector} of $M$.
Inspired by~\cite{DWZ2}, we consider 
\emph{decorated representations} $(M,\bv)$ of $A$, 
where $M$ is a representation of $A$, and 
$\bv\in\ZZ_{\geq 0}^{Q_0}$. Here, $\bv$ stands for the
\emph{negative semisimple representation} 
$\oplus_{i\in Q_0} (-S_i)^{v_i}$. 
The \emph{$g$-vector} 
\[ 
\gls{gvecA}(M,\bv):= (-\dim\Hom_A(M, S_i)+\dim\Ext^1_A(M, S_i)+v_i)_{i\in Q_0}
\]
stores relevant information about the minimal projective presentation
of $(M,\bv)$. 
Given another  decorated representation $(N,\bw)$ of $A$, the (symmetrized) $E$-invariant is 
\[
\gls{EinvA}((M,\bv),(N,\bw)):=\dim\Hom_A(M,\tau N)+\dim\Hom_A(N,\tau M)+ 
\dimv(M)\cdot\bw + \bv\cdot\dimv(N). 
\]
Following~\cite[Sec.~2 \& Sec.~5]{CLS15} and~\cite[Sec.~2 \& Sec.~9]{GLFS22}, we denote by $\rep_A^\bd$ the affine $\Ka$-scheme of representations,
with dimension vector $\bd\in\ZZ_{\geq 0}^{Q_0}$, of $A$.
The algebraic group 
$G_\bd(\Ka)=\bigtimes_{i\in Q_0}\GL_{d_i}(\Ka)$
acts on the $\Ka$-rational points of $\rep_A^\bd$ by conjugation,
and its orbits are naturally in bijection with the isomorphism classes of representations, with dimension vector $\bd$, of $A$.  
See for example~\cite[Sec.~2.1]{GLFS22}, where this scheme is
denoted by $\operatorname{rep}(A,\bd)$, for more details.  
Given $(\bd,\bv)\in\ZZ_{\geq 0}^{Q_0}\times\ZZ_{\geq 0}^{Q_0}$,
we consider also the affine $\Ka$-scheme
$\DecRep_A^{(\bd,\bv)}$, which consists of pairs $(M,\bv)$ with
$M\in\rep(A,\bd)$.  Obviously, as a $G_\bd$-scheme, 
$\DecRep_A^{(\bd,\bv)}$ is isomorphic to $\rep_A^\bd$.
For  irreducible components $X\subset\DecRep(A,(\bd,\bv))$ and
$Y\subset\DecRep(A,(\be,\bw))$ we define 
\begin{align*}
c_A(X)&:=\min\{\codim_Z(G_\bd\cdot(M,\bv))\mid (M,\bv)\in X\},\\
\gls{eAX}&:=\min\{1/2 E_A((M,\bv), (M,\bv))\mid (M,\bv)\in X\},\\
\gls{eAXY} &:=
\min\{E_A((M,\bv), (N,\bw))\mid (M,\bv)\in X, (N,\bw)\in Y\},\\
\gls{gAX}_i &:= \min\{\bg_A(M,\bv)_i \mid (M,\bv)\in X\}.\\
\end{align*}
It is well-known, that the codimension of orbits is an upper semicontinuous function. By the main result of~\cite{GLFS23},
the same is true for the $E$-invariant, and by~\cite[Cor.~1.3]{GLFS23}
it is also true for the components of the $g$-vector.
Thus, $c_A(X)$ is the generic codimension of $G_\bd$-orbits in $X$,
and $\bg_A(X)$ is the generic value of the function 
$\bg_A\df X\ra\ZZ^{Q_0}, (M,\bv)\mapsto \bg_A(M,\bv)$, etc. Note, that
possibly $1/2E_A(X,X)<E_A(X)$.

By definition, $X$ is called a \emph{\gls{taure} (irreducible) component}, if  $c_A(X)=e_A(X)$.  By Voigt's Lemma
and the Auslander--Reiten formula, $X$ is in this case generically
reduced as a scheme.  However, being generically $\tau$-regular
is a stronger (algebraic) condition than being generically reduced.
In earlier works, e.g.~\cite{CLS15}, those components were called
\emph{strongly reduced}. 
Recently, in~\cite{BS25} the name \emph{generically $\tau$-regular} is promoted. 
We denote by $\gls{DecIrrt}(A, (\bd,\bv))$ the set of generically
$\tau$-regular irreducible components of $\DecRep_A^{(\bd,\bv)}$,
and set
\[
{\gls{DecIrrt}(A)}:= \!\!\bigcup_{(\bd,\bv)\in\ZZ_{\geq 0}^{Q_0}\times \ZZ_{\geq 0}^{Q_0}} \!\!\DecIrr^\tau(A, (\bd,\bv)).
\]
It will be convenient to introduce a pair of functions 
\[
(\bd, \bv)\df\DecIrr^\tau(A)\ra \NN^{Q_0}\times \NN^{Q_0}
\]
such that $Z\in\DecIrr^\tau(A, (\bd(Z), \bv(Z))$ for all $Z\in\DecIrr^\tau(A)$.
This construction extends the fundamental notion of support $\tau$-tilting modules from~\cite{AIR14}. 

In case  $Q$ is 2-acyclic and $A$ is a non-degenerate Jacobian algebra
over $\CC$, the set $\DecIrr^\tau(A)$ parametrizes naturally the
generic Caldero--Chapoton functions in the upper cluster algebra
associated to $Q$, see for example~\cite{CLS15} for more details.

\subsection{KRS-monoids of \texorpdfstring{$\tau$}{τ}-regular components} 
We keep the notation from Section~\ref{ssec:gen-taured} and consider
\label{ssec:FD-KRS}
\[
  \DecIrr^\tau(A)=(\DecIrr^\tau(A), e_A, \oplus, \bg_A),
\]
where  $ X\oplus Y:=\overline{X\oplus Y}$ denotes
the direct sum of $e_A$-orthogonal irreducible components.
See for example~\cite[Sec.~9.5]{GLFS22} for more details. 
Our definitions allow us now, to state a combination of   results by  Cerulli Irelli, Crawley-Boevey,   Plamondon, Schröer and two of the present authors (see~\cite{CBS02}, \cite{Pl13}, \cite{CLS15}, \cite{GLFS22}, \cite{GLFS23}) in a compact way:  

\begin{Thm} \label{cor:fd}
  Let $A$ be a finite dimensional algebra over an algebraically
  closed field with Grothen\-dieck group of rank $n$.
\begin{itemize}
\item[(a)]     
$\DecIrr^\tau(A)=(\DecIrr^\tau(A), e_A, \oplus)$ is a KRS-monoid.  
  The subset $\DecIrr^\tau_\ind(A)$ of components,
  which contain a dense set of indecomposable representations, is precisely the set of indecomposable elements in the sense of Definition~\ref{def:KRS2}.
\item[(b)]
  With the framing from the generic $g$-vector, 
  $\DecIrr^\tau(A)$ is a free partial KRS-monoid of rank $n$.  
\item[(c)] 
  If $A$ is tame, $\DecIrr^\tau(A)$ is tame in the sense of the same  Definition~\ref{def:KRS2}.  
  In particular,  we have an isomorphism of partial KRS-monoids
  \[
  \DecIrr^\tau(A)\cong \KRS(\DecIrr^\tau_\ind(A), e_A).  
\]
Moreover in this case each $Z\in\DecIrr^\tau_\ind(A)$ contains either a dense orbit, or a one-parameter family of bricks.  
\end{itemize}
\end{Thm}  

In fact, Part~(a) is the well-known combination of~\cite[Thm.~1.2]{CBS02}
and~\cite[Thms.~1.3 \& 1.5]{CLS15}.
Part~(b) is Plamondon's theorem~\cite{Pl13}.  
Part~(c) is~\cite[Thm.~3.2]{GLFS22} and of~\cite[Cor.¬1.7]{GLFS23}.
Note, that in particular for $A$ tame we have $e_A(Z, Z)=0$ for \emph{all} irreducible components $Z\in\DecIrr^\tau(A)$. 

\begin{Rem}
In the context of cluster algebras it is convenient to study also the dual version with
\begin{multline*}
E_A^\inj((M,\bv), (N,\bw)):=\dim\Hom_A(\tau^{-1}M, N) + \dim\Hom_A(\tau^{-1}N,M)+ \\
\bv\cdot\dimv(N)+ \dimv(M)\cdot\bw
\end{multline*}
and
\[
\bg_A^\inj(M,\bv)= (\dim\Hom_A(\tau^{-1}M, S_i)-\dim\Hom_A(S_i,M) + v_i)_{i\in Q_0}
\]
Then, duality induces a natural isomorphism of framed partial KRS-monoids 
\[
(\DecIrr^\tau(A), e_A, \oplus, \bg_A)\cong 
(\DecIrr^{\tau^-}(A^\op), e_{A^\op}^\inj, \oplus, \bg_{A^\op}^\inj)
\]
\end{Rem}

\subsection{KRS-monoid isomorphisms from mutations of quivers with potential}\label{subsec:Jac-algs-and-muts}
Let $Q$ be a 2-acyclic quiver and 
$B=(b_{ij})\in\ZZ^{Q_0\times Q_0}$ with
\[
b_{ij}=
\abs{\{\alp\in Q_1\mid \alp:i\rightarrow j \ \text{in} \ Q\}} -\abs{\{ \bet\in Q_0\mid \bet:j\rightarrow i \ \text{in} \ Q\}}.
\]
Thus, $B$ is the transpose of the skew-symmetric matrix $B_{\operatorname{DWZ}}(Q)$ associated by Derksen--Weyman--Zelevinsky to $Q$ in \cite[Equation (1.4)]{DWZ2} and \cite[Section 2 and Equation (7.1)]{DWZ1}.  Suppose that
$W\in  \Ka\eebrace{Q}_{\cyc}$ is a potential such that the Jacobian algebra $A=\cP_{\Ka}(Q,W)$ is finite-dimensional.  Since $B=-B_{\operatorname{DWZ}}(Q)$, for cluster algebra purposes, we have to consider 
\emph{dual Caldero--Chapoton functions} as in~\cite[Sec.~11.3]{GLFS22}.  

Roughly following~\cite[Sec.~10]{DWZ1}, we introduce for each representation $M$ of 
$A=\cP_\Ka(Q, W)$ and each vertex $k\in Q_0$ the spaces
\[
M_\out(k):= \bigoplus_{\substack{\bet\in Q_1\\ s(\alp)=k}} M(t(\bet))\quad\text{and}\quad
M_\inc(k):= \bigoplus_{\substack{\alp\in Q_1\\ t(\alp)=k}} M(s(\alp)),
\]
together with the corresponding maps
\[
M(\alp_k)=\coprod_{\substack{\alp\in Q_1\\ t(\alp)=k}} M(\alp),\quad
M(\bet_k)=\prod_{\substack{\bet\in Q_1\\ s(\bet)=k}} M(\bet)\quad\text{and}\quad
M(\gam_k)= \prod_{\substack{\alp\in Q_1\\ t(\alp)=k}}\  \coprod_{\substack{\bet\in Q_1\\ s(\alp)=k}}
M(\partial_{\bet\alp} W), 
\]
which can be summarized in the following diagram
\[\xymatrix{
& M(k) \ar[rd]^{M(\bet_k)}\\
M_\inc(k)\ar[ru]^{M(\alp_k)}&&\ar[ll]^{M(\gam_k)} M_\out(k)
}.\]
\begin{Rem}
  In view of~\cite[Prop.~10.4 \& Rem.~10.8]{DWZ2} we have for each finite dimensional representation $M$ of $A$ the following chain of equations
\begin{align} \label{eq:gvec}
\begin{split}
    &\bg_A(M) = (\dim\CoKer(M(\gam_k))-\dim M(k))_{k\in Q_0}\\
= & \bg_{A^\op}^\inj(DM) =(\dim\Ker(DM(\gam_k))-\dim DM(k))_{k\in Q_0}=:\bg_{A^\op}^{\mathrm{DWZ}}(DM).
\end{split}
\end{align}
Here, $DM$ denotes the $\Ka$-linear dual of $M$, which we interpret as a representation of
$A^\op=\cP_\Ka(Q^\op, W^\op)$, and 
$\bg_{A^\op}^{\mathrm{DWZ}}(DM)$ denotes the $g$-vector of a QP-representation in the sense of~\cite[(1.13)]{DWZ2}.  
This can be easily extended to decorated QP-representations.

Now, consider for $k\in Q_0$ the QP-mutation $\mu_k(Q,W)$, which was introduced  in~\cite{DWZ1}, and let us abbreviate 
$A':=\cP_\Ka(\mu_k(Q,W))$.  
We also need the piecewise linear transformation of integer vectors 
\begin{equation} \label{eq:gproj-transf}
\gam_k^B\df \ZZ^{Q_0}\ra \ZZ^{Q_0}\quad\text{with}\quad 
\gam_k^B(\bg)_i = \begin{cases}
        -g_i &\text{if } i=k,\\
         g_i + \sgn(g_k)[b_{ik}\cdot g_k]_+ &\text{else}.
    \end{cases}
\end{equation}
Note that this is just  another way of writing the conjectural transformation rule for $g$-vectors of cluster monomials from~\cite[(7.18)]{CA4}.  This formula was proved for the skew-symmetric case in~\cite[Sec.~9]{DWZ2}. 
\end{Rem}

We have the following easy consequence  of~\cite{GLFS20}.  
\begin{Prop} \label{prp:JaMut}
    For each $Z\in\DecIrr^\tau(A)$  there exists a dense open subset $U_Z\subset Z$,
    a unique irreducible component 
    $\gls{tmuk}(Z)\in\DecIrr^\tau(A')$ and a regular
    map $\nu_Z\df U_Z\ra\tilde{\mu}_k(Z)$ with the following properties:
    \begin{itemize}
    \item[(a)]
        For each $X\in U_Z$ we have $\nu_Z(X)\cong\mu_k(X)$, where $\mu_k$ denotes the mutation of the decorated QP-representation $X$ in direction $k$, 
        as defined in~\cite{DWZ1}.
    \item[(b)]  The morphism of affine varieties
    $G_{\bd(\tilde{\mu}_k(Z))}\times U_Z \ra \tilde{\mu}_k(Z), (g, X)\mapsto g.\nu_Z(X)$
    is dominant.
    \item[(c)]
    For each $Z\in\DecIrr^\tau(A)$ we have 
    \[
    \bg_{A'}(\tilde{\mu}_k(Z))=\gam_k^B(\bg_A(Z)).
    \]
   \item[(d)] 
    With a slight abuse of notation, we have 
    $\tilde{\mu}_k(\tilde{\mu_k}(Z))=Z$ for all 
    $Z\in\DecIrr^\tau(A)$.
    \item[(e)] The map 
    \[
    \tilde{\mu}_k\df\DecIrr^\tau(A)\ra\DecIrr^\tau(A')
    \]
    is an isomorphism of partial KRS-monoids.  
    In particular, $Z\in\DecIrr^\tau(A)$ is indecomposable if and only if 
    $\tilde{\mu}_k(Z)\in\DecIrr^\tau(A')$ is indecomposable.
 \end{itemize}    
\end{Prop}

\begin{proof}
Part~(a) and~(b) follow by taking $\Ka$-duals from~\cite[Thm.~1.10]{GLFS20}, 
if we take into account that the mutation of QP-representations commutes with 
$\Ka$-duality.  
This allows us to work with generically $\tau$-regular components in place of 
generically $\tau^-$-regular components. 

(c) Recall that mutation of (decorated) QP-representations commutes with taking direct sums.
Thus, in view of the sign coherence for the canonical decomposition of decorated, generically $\tau$-regular irreducible components~\cite[Thm.~8.1]{CLS15} we may assume that
$Z$ is indecomposable.  Then we may assume  by~\cite[Cor.~2.8]{Pl13} 
\begin{multline*}
h'_k(M):=-\dim\CoKer(M(\alp_k)):=\min(0, \bg_A(M)_k)\\
=\min(0, (\bg^\inj_{A^\op}(DM))_k)=-\dim\Ker(DM(\alp_k))=:h_k(DM),
\end{multline*}
for all $M$ in the dense open subset $U_Z$ of $Z$.  As discussed in~\cite[Sec.~9]{DWZ2} (beginning of the Proof of Theorem~1.7) this implies already by Equation~\eqref{eq:gvec} that
$\bg_{A'}(\mu_k(M))=\gam_k(\bg_A(M))$ for all $M\in U_Z$.  
By Part~(b) this implies the desired formula for any  $Z\in\DecIrr^\tau(A)$.  
See also the comment after Formula~\eqref{eq:gproj-transf}.

(d)   We recall from~\cite[Thm.~5.7]{DWZ1} that $\mu_k(\mu_k(Q,W))$ is right equivalent to $(Q,W)$.  
Such a right equivalence induces an isomorphism of faithful framed partial KRS-monoids 
 \[
\iota\df\DecIrr^\tau(\cP(\mu_k(\mu_k(Q,W))))\ra\DecIrr^\tau(A), 
\]
where we use on both sides the g-vector framing.   With
$A'':=\cP_\Ka(\mu_k(\mu_k(Q,W)))$
we have in view of Part~(c) of the Proposition
 \[
 \bg_{A''}(\tilde{\mu}_k(\tilde{\mu}_k(Z)))=\bg_A(Z) \text{ for all } Z\in\DecIrr^\tau(A).  
 \]
 Thus, we have $\iota\circ\tilde{\mu}_k\circ\tilde{\mu}_k=\id_{\DecIrr^\tau(A)}$. 

(e)   By~\cite[Thm.~7.1]{DWZ2} and the discussion in~\cite[Sec.~10]{DWZ2} we conclude that
 in our situation we have $e_{A'}(\tilde{\mu}_k(X), \tilde{\mu}_k(Y))=e_A(X,Y)$
 for all $X, Y\in\DecIrr^\tau(A)$.  Thus,  $X\oplus Y$ is defined if and only if 
 $\tilde{\mu}_k(X)\oplus\tilde{\mu}_k(Y)$ is defined, and it remains to show that in this
 case we have $\tilde{\mu}_k(X\oplus Y)=\tilde{\mu}(X)\oplus\tilde{\mu}_k(Y)$.
 Now, by part~(c) of the Proposition, we have
 \[
 \bg_{A'}(\tilde{\mu}_k(X\oplus Y))=\gam_k^B(\bg_A(X\oplus Y))=\gam_k^B(\bg_A(X)+\bg_A(Y))
 =\gam_k^B(\bg_A(X))+\gam_k^B(\bg_A(Y)),
 \]
 where the last  equality holds because of the sign coherence of g-vectors  for $e_A$-orthogonal generically $\tau$-regular components~\cite[Thm.~8.1]{CLS15}.  
 On the other hand, we have trivially
\[
\bg_{A'}(\tilde{\mu}_k(X)\oplus\tilde{\mu}_k(Y))=
\bg_{A'}(\tilde{\mu}_k(X))+\bg_{A'}(\tilde{\mu}_k(Y))=\gam_k^B(\bg_A(X))+\gam_k^B(\bg_A(Y))
\] 
by Part~(c).
Thus, the g-vectors of the generically $\tau$-regular irreducible components
$\tilde{\mu}_k(X\oplus Y)$ and $\tilde{\mu}_k(X)\oplus\tilde{\mu}_k(X)$ coincide.  Since $\bg_{A'}$ is a faithful framing of  $\DecIrr^\tau(A')$ by Plamondon's theorem~\cite[Thm.~1.2]{Pl13} 
(see also Thm.~\ref{cor:fd}~(b)), we have indeed 
$\tilde{\mu}_k(X\oplus Y)=\tilde{\mu}_k(X)\oplus\tilde{\mu}_k(X)$. 
\end{proof}

\begin{Rem}
As in the opening of the current subsection, let \todo{$B_{\operatorname{DWZ}}(Q)$} be the skew-symmetric matrix associated by Derksen--Weyman--Zelevinsky to $Q$ in \cite[Equation (1.4)]{DWZ2} and \cite[Section 2 and Equation (7.1)]{DWZ1}.
Thus, $B_{\operatorname{DWZ}}(Q)$ has entries  
$\beta_{ij}:=b_{ji}$ for all  $i,j\in Q_0$.  Then the transformation rule for $g$-vectors in Part~(c) of the above Proposition can be rewritten as follows.
Let $\bg_A(Z)=(g_i)_{i\in Q_0}$
and $\bg_A'(\tilde{\mu}_k(Z))=(g'_i)_{i\in Q_0}$, then 
\begin{equation}
    g'_i = \begin{cases}
        -g_i &\text{if } i=k,\\
         g_i + \sgn(g_k)[g_k\cdot \beta_{ki}]_+ &\text{else}.
    \end{cases}
\end{equation}
Note that this transformation rule is in line with the  \todo{row} transformation for an extended matrix under mutation in direction $k$. See Theorem \ref{thm:behavior-of-dual-shear-coords-under-matrix-mutation} and \eqref{eq:unraveled-recursion-satisfied-by-dual-shear-coords} below.
\end{Rem}

\subsection{KRS-monoids of admissible strings and bands}
Following~\cite{Ge23}, a
\emph{\gls{skgpQ}} 
$\uQ=(Q_0, \Qsp\coprod \Qord,\us,\ut)$ is  a tuple,
where $Q_0$ is the (finite) set of vertices $Q_1:=\Qsp\cup \Qord$ is
the set of arrows, which contains a (possibly empty) subset of ``special'' loops $\Qsp$,  and 
  \[
    \gls{ust}\df Q_1\ra Q_0\times\{-1,1\}  
  \]
are two injective maps, which we write as
$\us=(s,s_1)$ and $\ut=(t,t_1)$, such that
$\us(\eps)=\ut(\eps)\in  Q_0
\times\{-1\}$ for all  $\eps\in\Qsp$.
We request moreover that there exists an upper bound
$L(\uQ)$ for the length $l$ of \emph{admissible paths} of the form
  $\alp_1\alp_2\cdots\alp_l$ with
  $\us(\alp_i)=-\ut(\alp_{i+1}):=(t(\alp_{i+1}),-t_1(\alp_{i+1}))$
for $i=1,2,\ldots, l-1$.
If, moreover $\Qsp=\emptyset$, we say that $\uQ$ is a \emph{gentle polarized quiver}.
   
For any  skewed-gentle polarized quiver $\uQ$ we define 
\[
  \glslink{Qspvo}{\Qspv}:=\{s(\eps)\mid \eps\in\Qsp\}\subset Q_0
  \quad\text{and}\quad \glslink{Qspvo}{\Qordv}:=Q_0\setminus\Qspv.
\]
It is clear that for each $i\in\Qspv$ there exists a \emph{unique}
$\eps_i\in\Qsp$ with $s(\eps_i)=i$.
See Section~\ref{sssec:skgep} for a concrete example of a skewed-gentle polarized quiver.

Let $\Ka$ be a field and $\uQ$ a skewed-gentle polarized quiver.  We write
\begin{align*}
  \gls{KauQ}   &:= \Ka Q/\ebrace{\cR(\uQ)}\quad\text{where}\\
  \cR(\uQ) &:= \{\alp\bet\mid \alp,\bet\in Q_1^\ord\text{ and }
             \us(\alp)=\ut(\bet)\} \cup \{\eps^2-e_{s(\eps)}\mid\eps\in Q_1^\spe\}.
\end{align*}             
Here, $e_i$ denotes the trivial path concentrated at the vertex $i\in Q_0$.
Thus, $\Ka\uQ$ is by definition the path algebra of the underlying quiver $Q$ of $\uQ$ modulo the, possibly non-admissible, ideal which is generated by $\cR(\uQ)$.  
If $\kar(\Ka)\neq 2$, the algebra $\Ka\uQ$ is skewed-gentle in the sense of~\cite{GP99}. 
Skewed-gentle algebras are tame,  since they are in particular clannish algebras in the sense of Crawley-Boevey~\cite{CB89b}.  

It is not hard to see that the Gabriel quiver $\tQ$ of a skewed-gentle algebra $\Ka\uQ$ has the vertex set
\begin{equation} \label{eqn:tQ}
\begin{split}
  \gls{tQo} &:=\{(i,\rho)\mid i\in Q_0 \text{ and }\rho\in S(i)\}\quad\text{where}\\
\gls{S(i)} &:= \begin{cases} \{+,-\} &\text{if } i\in Q_0^\spe,\\
     \{o\}   &\text{if } i\in \Qordv.
   \end{cases}
\end{split}   
\end{equation}
Here, we do not need the description of the arrows or the relations for $\tQ$.
See~\cite[Sec.~5.1]{Ge23} for more details on this.
We note however, that the Grothendieck group $K_0(\Ka\uQ)$ is
naturally identified with $\ZZ^{\tQ_0}$.  

In~\cite{Ge23}, the first author introduced a set
$\Adm^*(\uQ)$  of decorated admissible strings and bands,
together with an involution $?^{-1}$ of $\Adm^*(\uQ)$.  
This induces an equivalence relation $\simeq$ on $\Adm^*(\uQ)$. We extend this set to $\DAdm^*(\uQ)$
by adding symbols for the negative simple representations of 
$\Ka\uQ$.  

Moreover, a combinatorially defined symmetric pairing
\[
e_\uQ\df \DAdm^*(\uQ)\times\DAdm^*(\uQ)\ra\NN,
\]
and a combinatorially defined function
\[
\bg_\uQ\df\DAdm^*(\uQ)\ra\ZZ^{\tQ_0}.
\]
were introduced. Both $e_\uQ$ and $\bg_\uQ$ descend to well-defined functions on the respective domains $[\DAdm^*(\uQ)]/_{\simeq}\times [\DAdm^*(\uQ)]/_{\simeq}$ and $[\DAdm^*(\uQ)]/_{\simeq}$.
See Section~\ref{ssec:adm} and Section~\ref{ssec:combEg} for more details. 

Finally, consider the subset
\[
[\DAdm^*_\tau(\uQ)]:=\{(\bx,c)\in[\DAdm^*(\uQ)]\mid e_\uQ((\bx,c), (\bx,c))=0\}.
\]
One of the  main results of~\cite{Ge23} can now be stated as follows:

\begin{Thm} \label{thm:SkewedGIso}
  Let $\uQ$ be a skewed-gentle polarized quiver and $\Ka$ an algebraically
  closed field. There is an isomorphism of tame, framed
  partial KRS-monoids
  \[
    \widetilde{M}^o\df \KRS([\DAdm^*_\tau(\uQ)]/_{\simeq}, e_\uQ, \bg_\uQ)\ra
    (\DecIrr^\tau(\Ka\uQ), e_{\Ka\uQ}, \oplus, \bg_{\Ka\uQ})
  \]
  In particular, $\KRS([\DAdm^*_\tau(\uQ)]/_{\simeq}, e_\uQ, \bg_\uQ)$ is free
  of rank $\abs{\tQ_0}$.
\end{Thm}

\begin{proof}
In fact,  by~\cite[Thm.~6.4]{Ge23}  we have a natural bijection
\[
  \widebar{M}^o_{?}\df[\Adm^*_\tau(\uQ)]/_{\simeq}\ \ra \Irr^\tau_\ind(\Ka\uQ)
\]
such that
$e_{\Ka\uQ}(\widebar{M}^o_{(\bx,s)}, \widebar{M}^o_{(\by,t)})=e_\uQ((\bx,s), (\by,t))$ and 
$\bg_{\Ka\uQ}(\widebar{M}^o_{(\bx,s)})=\bg_\uQ(\bx,s)$ 
for all $(\bx,s), (\by,t)\in[\Adm^*_\tau(\uQ)]$.  
With our definitions, it is clear, that this extends to a bijection  
\[
[\DAdm^*_\tau(\uQ)]/_{\simeq}\ra\DecIrr^\tau_\ind(\Ka\uQ)
\]
with the same properties. 
Since $\Ka\uQ$ is tame, we have isomorphisms of tame, framed KRS-monoids: 
\[
  \KRS([\DAdm^*_\tau(\uQ)]/_{\simeq}, e_\uQ, \bg_\uQ)\rightarrow
  \KRS(\DecIrr^\tau_\ind, e_{\Ka\uQ}, \bg_{\Ka\uQ})\ra
  (\DecIrr^\tau(\Ka\uQ), e_{\Ka\uQ}, \oplus, \bg_{\Ka\uQ})
\]
where the first isomorphism is induced from  the bijection
$[\DAdm^*_\tau(\uQ)]/_{\simeq}\ra\DecIrr^\tau_\ind(\Ka\uQ)$, and the 
second isomorphism comes from Theorem~\ref{cor:fd}~(c).
The claim about the freeness follows now from Theorem~\ref{cor:fd}~(b). 
\end{proof}

\subsection{Marked curves on marked surfaces}
Let $\gls{bSu}=(\Sig,\Ma,\Pu)$  be a marked surface as in the introduction.
This means that $\Sig$ is a compact Riemann surface,
$\Ma\subset \dSu= \cup_{i=1}^b S_i$ is a finite set of marked points with
$\Ma\cap S_i\neq\emptyset$ for each boundary component $S_i$ and
$\Pu\subset\Su\setminus\dSu$ is a finite set of punctures.
Inspired by Qiu and Zhou~\cite{QZ17}  we introduce in Section~\ref{sec:punct} the set $\cLC^*(\bSu)$ of homotopy classes of marked loops and curves on $\bSu$ that have no kinks. 
Moreover, we  introduce a (symmetric) marked intersection number
\[
\Intn^*_\bSu\df \cLC^*(\bSu)\times\cLC^*(\bSu)\ra\NN,
\]
also inspired by the work of Qiu and Zhou. 
Given a marked curve (or  loop) $(\gam, c)\in\cLC^*(\bSu)$, we denote by $(\gam,c)^{-1}$ the inversely oriented curve $\gam$  with accordingly swapped decoration $c$.  
The above involution $?^{-1}$ induces an equivalence relation 
$\simeq$ on $\cLC^*(\bSu)$, and we let $\cLC^*(\bSu)/_{\simeq}$ denote the corresponding set of equivalence classes. 
For us, like in~\cite{QZ17}, a \emph{simple marked curve} is a marked curve $(\gam, c)\in\cLC^*(\bSu)$ which has self-intersection number $0$. 
We let $\cLC^*_{\tau}(\bSu)$ denote the set of all simple marked curves, and let $\cLC^*_{\tau}(\bSu)/_{\simeq}$ denote this set considered up to the equivalence relation $\simeq$.  
With this at hand we define the tame partial KRS-monoid of laminations
\[
\MSW(\bSu):= \KRS(\cLC^*_{\tau}(\bSu)/_{\simeq}, \Intn^*_\bSu),
\]
see Section~\ref{ssec:taggedC} for more details.  As a set, $\MSW(\bSu)$ can be identified
with the set of laminations $\mathcal{C}^o(\bSu)$ considered by Musiker--Schiffler--Williams
in~\cite{MSW13}.  This is also compatible with the treatment of laminations in~\cite{FT18},
Section~\ref{ssec:arcs}, see Sections~\ref{subsec:dual-shear-coords-arb-triangs} and~\ref{ssec:arcs}.

\section{Skewed-gentle polarized quivers} \label{sec:PolQ}
We review the relevant definitions and constructions from~\cite{Ge23}.
\subsection{Basic Definitions} \label{ssec:polar}
A \emph{\gls{heredpolQ}} 
is a skewed-gentle polarized quiver $\uH$, where each loop is special, and for each pair of different arrows
$(\alp,\bet)\in H_1\times H_1$ with $s(\alp)=t(\bet)$  we have  $\us(\alp)=-\ut(\bet)$.  
It is easy to see that the underlying graph of a  connected hereditary polarized quiver with $n$ vertices is isomorphic to one of the quivers from Table~\ref{tab:AD-graphs}.

\begin{table}[ht]
\begin{align*}
  \def\objectstyle{\scriptstyle}
  \sfA_n\colon&  \xymatrix{\qquad 1\ar@{-}[r]^>>>>{\nu_1}&2\ar@{-}[r]^{\nu_2} &\cdots \ar@{-}[r]^{\nu_{n-2}} &{n\!\!-\!1\!}\ar@{-}[r]^>>>>>{\nu_{n-1}}& n}\\
  \sfD'_n\colon&  \xymatrix{\qquad 1\ar@{-}[r]^>>>>{\nu_1}&2\ar@{-}[r]^{\nu_2} &\cdots \ar@{-}[r]^{\nu_{n-2}} &{n\!\!-\!\!1}\ar@{-}[r]^>>>>>{\nu_{n-1}}& n \ar@(ur,dr)[]^{\eta_1}}\\
     &\!\!\!\xymatrix{\ar@(ul,dl)[]_{\eta_0} 1\ar@{-}[r]^{\nu_1}&2\ar@{-}[r]^{\nu_2} &\cdots \ar@{-}[r]^{\nu_{n-2}} &{n\!\!-\!\!1}\ar@{-}[r]^>>>>>{\nu_{n-1}}& n} \\     
  \sftA_n\colon&  \xymatrix{\qquad 1\ar@{-}[r]^>>>>{\nu_1}&2\ar@{-}[r]^{\nu_2} &\cdots \ar@{-}[r]^{\nu_{n-2}} &{n\!\!-\!\!1}\ar@{-}[r]^>>>>>{\nu_{n-1}}&\ar@{-}[lld]^{\nu_n} n\\
              &&\ar@{-}[llu]^{\nu_0}0}\\
    \sftD'_n\colon&  \xymatrix{\ar@(ul,dl)[]_{\eta_0} 1\ar@{-}[r]^{\nu_1}&2\ar@{-}[r]^{\nu_2} &\cdots \ar@{-}[r]^{\nu_{n-2}} &{n\!\!-\!\!1}\ar@{-}[r]^>>>>>{\nu_{n-1}}& n \ar@(ur,dr)[]^{\eta_1}}\\   
\end{align*}
\caption{Underlying graphs of connected hereditary polarized quivers}
\label{tab:AD-graphs}
\end{table}

We extend and slightly generalize the notion of fringing
from~\cite[Sec.~3]{BY20} to our setting of skewed-gentle polarized quivers.
Let $\uQ\subset\uQ^f$ be an inclusion  of  skewed-gentle polarized
quivers. We say that $\uQ^f$ is
a \emph{\gls{fringing}} of $\uQ$ if the following two conditions hold:
\begin{itemize}
\item[(1)]
For each  vertex  $i\in Q_0$ we have
\[
  \abs{\{\alp\in Q^f_1}\mid s(\alp)=i\}+\abs{\{\bet\in Q^f_1}\mid t(\bet)=i\}= 4.\]  
\item[(2)]
 For each \emph{\gls{fringearrow}} 
 $\alp\in Q^f_1\setminus Q_1$ we have:
 $s(\alp)\in Q_0$ implies $\ut(\alp)\in (Q_0^f\setminus Q_0)\times\{+1\}$,
 and $s(\alp)\in\uQ^f_0\setminus\uQ_0$ implies $s_1(\alp)=+1$ and
  $t(\alp)\in Q_0$.  
\end{itemize}  
We call the elements of $Q_0^f\setminus Q_0$ \emph{\gls{fringevertices}}. 

In particular, each fringe arrow connects a fringe vertex with a vertex in $Q_0$.  
Moreover, each fringe vertex has valency at most 2, since at a fringe vertex the polarization of each incident arrow is $+1$.
It is clear that each skewed-gentle polarized quiver admits a
fringing, however our fringings are  not unique. In the Example from Section~\ref{sssec:skgep} below $\uQ^f$ is indeed a fringing of $\uQ$.

\begin{Rem}  (1)  For a field $\Ka$ with $\kar(\Ka)\neq 2$ and a skewed-gentle polarized quiver, we constructed in~\cite[Sec.~5.1]{Ge23} an algebra
  $\Ka\uQ$ in terms of quiver $\tuQ$ with relations, which turns out to be a
  skewed-gentle algebra in the sense of~\cite{GP99}.  If $\uQ$ is moreover
  connected and hereditary, $\Ka\uQ$ is isomorphic to a path algebra of
  a quiver of type  $\sfA_n$, $\sfD_{n+1}$, $\sftA_n$ or $\sftD_{n+2}$.  

  (2) We will see that, as in the gentle case~\cite[Thm.~4.4]{BY20}, fringings are a convenient tool to calculate $E$-invariants between the representations of $\Ka\uQ$.  If
  $\bSu=(\Su,\Ma,\Pu)$ is a marked surface with non-empty boundary
  $\partial\Su$ and $T=(\tau_1,\ldots,\tau_n)$ is a signature zero tagged triangulation in the sense of~\cite{FST08}, then we can assign to
  $T$ naturally a skewed-gentle polarized quiver $\uQ(T)$ which retains the
  information about the orientation of the arcs $\tau_i$.  Moreover, if we
  take into account the boundary segments of $\bSu$, we obtain naturally
  a fringing $\uQ^f(T)$ of $\uQ(T)$, see Section~\ref{ssec:tagged3ang}
  for more details.
\end{Rem}

\subsection{Letters and words}
\label{ssec:let}
Associated to a skewed-gentle polarized quiver $\uQ$ we have a set $\cL=\gls{cL(uQ)}$ of letters, which come in 5 types.
We describe them in the  Table~\ref{tab:letters}.
In particular, we extend the functions $\us$ and $\ut$ from arrows
to letters. Note that the trivial (inverse) letters 
$\II_{i,\rho}$ resp.~$\II_{i,\rho}^{-1}$ start resp.~end at a ``virtual polarized" vertex $*\not\in Q_0\times\{-1,1\}$.
Sometimes, for typographical reasons, we will abbreviate 
$\II_{(i,\pm 1)}$ to $\II_{(i,\pm)}$.

\begin{table}[ht]
\begin{tabular}{|p{1.5cm} |l |p{4cm}|c|c|l|}\hline
  type     & symbol & condition       & $\us$      & $\ut$      & inverse    \\ \hline\hline
  ordinary & $\alp$ & $\alp\in\Qord$& $\us(\alp)$& $\ut(\alp)$& $\alp^{-1}$ \\ \hline
inverse &$\alp^{-1}$ & $\alp\in\Qord$& $\ut(\alp)$& $\us(\alp)$& $\alp$     \\ \hline
special &$\eps^*$   & $\eps\in\Qsp$ & $(s(\eps),-1)$ & $(t(\eps),-1)$& 
$\eps^*$  \\ \hline
  trivial & ${\glslink{IIirho}{\II_{(i,\rho)}}}$ & $(i,\rho)\in Q_0\times\{-1,1\}\setminus\{\us(\eps)\mid\eps\in\Qsp\}$ &
* &$(i,\rho)$  & $\II_{(i,\rho)}^{-1}$\\ \hline                               trivial {inverse} &${\glslink{IIirho}{\II_{(i,\rho)}^{-1}}}$ & $(i,\rho)\in Q_0\times\{-1,1\}\setminus\{\us(\eps)\mid\eps\in\Qsp\}$ &$(i,\rho)$  & *& $\II_{(i,\rho)}$\\\hline
\end{tabular}\vspace*{2ex}
\caption{Letters for a skewed-gentle} polarized quiver
\label{tab:letters}
\end{table}
For each $(i,\rho)\in Q_0\times\{-1,+1\}$, the set of letters
\[
\cL_{(i,\rho)}(\uQ):=\{l\in\cL\mid\ut(l)=(i,\rho)\}
\]
contains at most three letters. We agree that those sets are linearly ordered
by the following rules:
\begin{itemize}
\item $\alp\gls{<}\II_{\ut(\alp)}$ for each ordinary letter $\alp$ and
\item $\II_{\us(\bet)}\gls{<}\bet^{-1}$ for each inverse letter $\bet^{-1}$.  
\end{itemize}
Note, that in case $\cL_{(i,\rho)}(\uQ)$ has three elements, it is of the form
$\{\bet <\II_{i,\rho}<\alp^{-1}\}$ for certain ordinary arrows 
$\alp,\bet\in Q_1$.
Moreover, $\cL_{(i,-1)}=\{\eps_i^*\}$ for $i\in\Qspv$.   
See~\cite[Sec.~2]{Ge23} for more details.

A \emph{\gls{word}} for $\uQ$ is a sequence of letters 
$\bw=w_1w_2\cdots w_m$ with the $w_i\in\cL_\uQ$ such that
$\us(w_i)=-\ut(w_{i+1})$ for $i=1,2,\ldots,m-1$. In this case we set
\begin{align*}
\glslink{wop}{\bw_{[k]}}  &=w_1w_2\ldots w_{k}, &
\glslink{wop}{\us(\bw)}  &=\us(w_m),\\
\glslink{wop}{\bw^{[k]}}  &=w_kw_{k+1}\cdots w_m, &
\glslink{wop}{\ut(\bw)}  &=\ut(w_1),\\ 
\glslink{wop}{bw^{-1}}    &=w_m^{-1}w_{m-1}^{-1}\cdots w_1^{-1}, &
\glslink{wop}{l(\bw)}         &=m.  
\end{align*}
If $\bv=v_1v_2\cdots v_l$ is another word with 
$\ut(\bv)=-\us(\bw)$, we can form the concatenated word 
$\bw\bv=w_1\cdots w_mv_1\cdots v_l$.

Finally, $\bw$ is \emph{right \gls{inextensible}} if 
$w_m=\II_{i,\rho}$ for some
$(i,\rho)\in Q_0\times\{-1,1\}$, and it is 
\emph{left \gls{inextensible}} if $\bw^{-1}$ is right inextensible.

Note  that two right inextensible words with comparable first
letter are comparable under the usual (left to right) lexicographical order on words, which is induced by the order on the letters.
The same is true for two words of the same length, if their respective first letters are comparable.

\subsection{Strings and Bands} \label{ssec:StBa}
A \emph{\gls{string}} for $\uQ$ is a word which is left- and right inextensible.
Thus, the first letter of a string is always a trivial inverse letter.
We denote the set of strings for $\uQ$ by$\glslink{StrQ}{\St(\uQ)}$,
and  $\glslink{StrQ}{\St_{(i,\rho)}(\uQ)}$ denotes the subset of strings which have as first letter $\II_{i,\rho}^{-1}$.
This set is linearly ordered by the lexicographic order, see above.

A string $\bw$ is \emph{\gls{symmetric}} if $\bw=\bw^{-1}$.
It is easy to see that in this case there exists a unique left inextensible word $\bv$ and a special letter
$\eps^{*}=(\eps^*)^{-1}$ such that $\bw=\bv\eps^*\bv^{-1}$.

For example, we have the \emph{\gls{simplestring}}
\begin{equation} \label{eq:tbsi}
{\gls{tbsi}} :=\begin{cases} 
\II_{(i,+1)}^{-1}\eps_i^*\II_{(i,+1)} &\text{if } i\in\Qspv,\\
\II_{(i,+1)}^{-1}\II_{(i,-1)}      &\text{if }i\in\Qordv
\end{cases}
\end{equation}
for $\uQ$.
Thus, $\tbs_i$ is a symmetric string if and only if $i\in\Qspv$.

By the definition of $\uQ$, for
\[
(i,\rho)\in Q_0^+:=(Q_0\times\{-1,1\})\setminus\{\us(\eps)\mid\eps\in\Qsp\}
\]
there exists a unique, left inextensible word of maximal length  
$\bl_{(i,\rho)}=\II^{-1}_{(i',\rho')}l_{p(i,\rho)}\cdots l_2l_1$ such that $\us(l_1)=(i,-\rho)$, and where none of the letters $l_i$ are inverse letters. 
If $p(i,\rho)=0$, we agree that $\bl_{(i,\rho)}=\II^{-1}_{(i,-\rho)}$.
Similarly, there exists a unique, right inextensible word
$\br=r_1r_2\cdots r_{q(i,\rho)}\II_{(i",\rho")}$ with
$\ut(\br_1)=(i,-\rho)$, and where none of the letters 
 $r_i$ are inverse.  Note that for $i\in\Qspv$ we have $\rho=+1$ and
$l_1=\eps_i^*=r_1$. In this case, we set 
$\bl'_{(i,+1)}=\II_{(i',\rho')}l_{p(i,+1)}\cdots l_2$ and
$\br'_{(i,+1)}:= r_2r_3\cdots r_{q(i,+1}\II_{(i",\rho")}$.
Then we set
\begin{equation}
    {\glslink{tbpq}{\tilde{\bp}_{(i,\rho)}}}:= \begin{cases}
    \bl'_{(i,+1)}\bl_{(i,+1)}^{-1} &\text{if } i\in \Qspv,\\
    \bl_{(i,\rho)}\bl^{-1}_{(i,-\rho)} &\text{else,}
\end{cases}
\quad\text{and}\quad
{\glslink{tbpq}{\tilde{\bq}_{(i,\rho)}}}:= \begin{cases}
    \br^{-1}_{(i,+1)}\br'_{(i,+1)} &\text{if } i\in \Qspv,\\
    \br^{-1}_{(i,\rho)}\bl_{(i,-\rho)} &\text{else.}
\end{cases}
\end{equation}
Note, that 
$\tilde{\bp}_{(i,\rho)}^{-1}=\tilde{\bp}_{(i,-\rho)}$ if 
$i\in\Qordv$ and $\tilde{\bp}_{(i,+1)}$ is symmetric if $i\in\Qspv$. Similarly, 
$\tilde{\bq}_{(i,\rho)}^{-1}=\tilde{\bq}_{(i,-\rho)}$
if $i\in \Qordv$ and $\tilde{\bq}_{(i,+1)}$ is symmetric if $i\in\Qspv$.
The set $\glslink{tcPI}{\widetilde{\cP}_\uQ}:=
\{\tilde{\bp}_{(i,\rho)}\mid (i,\rho) \in Q_0^+\}$
is called the set of \emph{projective strings}. Similarly
$\glslink{tcPI}{\widetilde{\cI}_\uQ}:=
\{\tilde{\bq}_{(i,\rho)}\mid (i,\rho)\in Q_0^+\}$
is called the set of \emph{injective strings}.
See Section~\ref{sssec:exStBa} for examples. 

A \emph{\gls{band}} is a word $\bw$ such that $\bw\bw$ is also a word.
A band is \emph{\gls{primitive}} if there is no  word $\bv$ such that
$\bw=\bv^{n}$ for some $n\geq 2$.
If $\bw$ is a (primitive) band, then so are all its
\emph{\gls{rotations}} $\bw^{[k+1]}\bw_{[k]}$.
A primitive band is \emph{\gls{symmetric}} if 
$\bw^{-1}=\bw^{[k+1]}\bw_{[k]}$ for some $k$.
This is the case if and only if some rotation of $\bw$ is of the form $\eps^*\bv\zeta^*\bv^{-1}$ for some word $\bv$ and
special letters $\eps^*$ and $\zeta^*$. 
Note, that a special letter $\eps^*$ alone is not a symmetric band, since $\eps^*\eps^*$ is not a word. 

We denote by $\glslink{BaQ}{\pBa(\uQ)}$ resp.~$\glslink{BaQ}{\Ba(\uQ)}$ the set of (primitive) bands for $\uQ$.

Two strings are \emph{\gls{equivalent}}, in symbols,
$\bv\gls{sim}\bw$ if $\bv\in\{\bw,\bw^{-1}\}$.
Two bands are \emph{\gls{equivalent}}, in symbols $\bv\gls{sim}\bw$ if
$\bv^\rho=\bw^{[k+1]}\bw_{[k]}$ for some $\rho\in\{-1,+1\}$ and $k=0,1,\ldots, m\!-\!1$, 
where $m$ is the number of letters of $\bw$.
We denote by $\pBa'(\uQ)$ the set of primitive bands in \emph{standard form}.
By definition, $\glslink{BaQ}{\pBa'(\uQ)}\subset\pBa(\uQ)$ consists of all asymmetric primitive bands, and of the symmetric (primitive) bands, which are of the form 
$\eps^*\bv\zeta^*\bv^{-1}$, as above.
The equivalence classes of bands resp.~primitive bands modulo rotations  are denoted by  $[\Ba(\uQ)]$ resp. $[\pBa(\uQ)]$.

\subsection{\texorpdfstring{$\huQ$}{Q} and construction of \texorpdfstring{$A(\bw)$}{A(w)} for strings and
  bands} \label{ssec:huQA}
Following~\cite[Sec.~3.2]{Ge99},
the \emph{associated gentle polarized quiver} $\gls{huQ}=(\hQ_0,\hQ_1,\us,\ut)$
of a skewed-gentle polarized quiver $\uQ$ 
is obtained  by taking $\hQ_0=Q_0$, 
$\hQord=\Qord\cup\Qsp$, $\hQsp=\emptyset$ and otherwise the same functions
$\us,\ut$. Thus,  we have a morphism
\[
\gls{FQ}\df\huQ\ra\uQ
\]
of polarized quivers, which is bijective on vertices and arrows.  Note, 
that $\huQ$ is a gentle polarized quiver. 
In particular, for each special arrow $\eps_i\in\Qsp$ we now have a linearly ordered set
\[
  \cL_{(i,-1)}(\huQ)=\{\eps_i\gls{prec}\II_{(i,-1)}\prec\eps_i^{-1}\}
\]
of letters.  With our conventions, we may set
\[
\gls{bsi} :=\II_{(i,-1)}^{-1}\II_{(i,+1)}\in\St_{(i,-1)}(\huQ)  
\]  
for all $i\in Q_0$.

We can extend $\gls{FQ}$ to a surjective morphism  between the corresponding
sets of  words, which is compatible with the concatenation of words.
In particular,    $F^{-1}(\eps_i^*)=\cL_{(i,-1)}(\huQ)$ for each special letter $\eps_i^*$ and
$F^{-1}(l)=\{l\}$ for each letter $l$ which is not special.

For example, we have
\[
F(\bs_i)=\begin{cases} \bs_i=\tbs_i &\text{if } i\in \Qordv,\\
\eps^*_i\II_{(i,+1)} &\text{if } i\in\Qspv.
\end{cases}
\]

Following~\cite[Sec.~2.3]{QZ17} we call the trivial letters for $\huQ$ which are of the form 
  $\II^{\pm 1}_{(i,-1)}$ with $i\in\Qspv$ 
  \emph{\gls{puncturedletters}}.
We say that a string $\bx=x_0x_1\cdots x_l\in\St(\huQ)$ 
is of \emph{\gls{type}} $(a,b)\in\{u,p\}^2$, with $a=p$ if and only if $x_0$ is
punctured and $b=p$ if and only if $x_l$ is punctured.

We have for example
\[
  \type(\bs_i)=\begin{cases} (u,u) &\text{if } i\in\Qordv,\\
                             (p,u) &\text{if } i\in\Qspv.
                           \end{cases}
\]

Similarly, we say that the elements of $\pBa(\huQ)$ are of \emph{\gls{type}} $(b)$. 
Thus, we have defined a map
\[
\type\df\St(\huQ)\cup\Ba(\huQ)\ra \{u,p\}^2\cup\{b\}.
\]
Note that we use the symbol $\gls{prec}$ for the partial order on the letters for $\huQ$, as well as on the induced lexicographic order on the corresponding words, rather than the symbol 
$\leq$, which we reserve for the strings and bands of $\uQ$.

We construct now, closely inspired by~\cite[Sec.~3]{CB89b}, for each
$\bw\in\St(\uQ)\cup\pBa'(\uQ)$ an element
$\gls{A(bw)}\in\St(\huQ)\cup\pBa(\huQ)$.   

For $\bw=w_0w_1\cdots w_n$  an \textbf{asymmetric string}, we set
$A(w):=a_0a_1\cdots a_n$ with $a_i=w_i$ whenever $w_i$ is not a special letter.
If however $w_i=\eps^*$ is a special letter, then
$a_i=\eps$ if $(\bw_{[i-1]})^{-1} >\bw^{[i+1]}$,  and $a_i=\eps^{-1}$ if
$(\bw_{[i-1]})^{-1} <\bw^{[i+1]}$.  

For $\bw=v_0v_1\cdots v_{m-1}\eps_b^* v_{m-1}^{-1}\cdots v_0^{-1}$  with
$b\in\Qspv$, a \textbf{symmetric string}, we set
$A(\bw):=a_0a_1\cdots a_{m-1} \II_{b,-1}$ with $a_i=w_i$ if $w_i$ is not a special
letter. If, however, $w_i=\eps^*$ is a special letter, then
$a_i=\eps$ if $(\bw_{[i-1]})^{-1} >\bw^{[i+1]}$,  and $a_i=\eps^{-1}$ if
$(\bw_{[i-1]})^{-1} <\bw^{[i+1]}$.

For $\bw=w_0w_1\cdots w_n$ an \textbf{asymmetric band}, we set
$A(\bw):=a_0a_1\cdots a_n$ with $a_i=w_i$ if $w_i$ is not a special letter.
If, however, $w_i=\eps^*$ is a special letter, 
then $a_i:=\eps$ if
$\bw_{[i-1]}^{-1}(\bw^{[i-1]})^{-1}>\bw^{[i+1]}\bw_{[i-1]}$ and $a_i:=\eps^{-1}$ 
if $\bw_{[i-1]}^{-1}(\bw^{[i-1]})^{-1}<\bw^{[i+1]}\bw_{[i-1]}$.

For $\bw=\eps^*_av_1v_2\cdots v_{n-1}\eps^*_bv_{n-1}^{-1}\cdots v_1^{-1}$ a \textbf{symmetric band} in standard form we set
$A(\bw):=\II^{-1}_{a,-1}a_1a_2\cdots a_{n-1}\II_{b,-1}$ with
$a_i=v_i$ if $v_i$ is not a special letter.  Else,  for $v_i=\eps^*$  special,
we set $a_i=\eps$ if $\bw_{[i-1]}^{-1}(\bw^{[i+1]})^{-1}>\bw^{[i+1]}\bw_{[i-1]}$ and
$a_i=\eps^{-1}$ if $\bw_{[i-1]}^{-1}(\bw^{[i+1]})^{-1}<\bw^{[i+1]}\bw_{[i-1]}$.
See Section~\ref{sssec:ExAdm} below, for examples
of the construction of $A(\bw)$ in the context of our running
example~\ref{sssec:skgep}.
  
\subsection{Admissible words vs.~strings and bands}
\label{ssec:adm}
Let $\uQ$ be a skewed-gentle polarized quiver, and $\huQ$ the associated gentle polarized quiver from Section~\ref{ssec:huQA}. 
We recall that $F\df\huQ\ra\uQ$ induces a surjective map between the respective sets of words, which we also denote by $F$.  However, $F$ does not restrict directly to a map between the respective sets of strings and (primitive) bands. For example, for $i\in\Qspv$ we have
$\II_{(i,-1)}^{-1}\II_{(i,+1)}^{\phantom{-1}}\in\St(\huQ)$, however
$F(\II_{(i,-1)}^{-1}\II_{(i,+1)}^{\phantom{-1}})=\eps_i^*\II_{(i,+1)} \not\in\St(\uQ)$. 

Thus, we define the \emph{\gls{completion}}
$\gls{cpl}\df \St(\huQ)\cup\Ba(\huQ)\ra\St(\uQ)\cup\Ba(\uQ)$ by 
\begin{equation} \label{eq:cpl}
  \bx=x_0x_1\cdots x_l \mapsto\begin{cases}
    F(\bx) &\text{if } \bx\in\St(\huQ) \text{ is of type } (u,u),\\
    F(\bx)F(x_{l-1}^{-1}\cdots x_1^{-1} x_0^{-1})
    &\text{if } \bx\in\St(\huQ) \text{ is of type } (u,p),\\
    F(x_{l}^{-1}\cdots x_2^{-1}x_1^{-1})F(\bx)
    &\text{if } \bx\in\St(\huQ) \text{ is of type } (p,u),\\
    F(\bx) F(x_{l-1}^{-1}\cdots x_2^{-1}x_1^{-1}) 
    &\text{if } \bx\in\St(\huQ) \text{ is of type } (p,p),\\
    F(\bx) &\text{if } \bx\in\Ba(\huQ).
  \end{cases}
\end{equation}
In the first three cases $\widebar{\bx}\in\St(\uQ)$, whilst
$\widebar{\bx}\in \Ba(\uQ)$ in the last two cases. Note, that for a primitive band $\bx\in\pBa(\huQ)$, possibly 
$\widebar{\bx}\in\Ba(\uQ)$ is not primitive.
Indeed, if we take $\uQ$ as in the Example from Section~\ref{sssec:skgep} below,
$\bx=\eps_1\gam\eps_3\gam^{-1}\eps_1^{-1}\gam\eps^{-1}_3\gam^{-1}\in\pBa(\huQ)$, but $\widebar{\bx}=F(\bx)=(\eps_1^*\gam\eps_3^*\gam^{-1})^2\in\Ba(\uQ)$ is 
\emph{not} primitive.

We extend slightly a definition proposed by Qiu and Zhou~\cite[Sec.~2.5]{QZ17}: 

\begin{Def}\label{def:admsb}
Consider for a string $\bx=x_0x_1\cdots x_l\in\St(\huQ)$ the following
conditions.
\begin{itemize}
\item[(s1)]
For any $k\in\{1,2,\ldots, l-1\}$ with  $F(x_k)\in\Qsp$ we have 
  $\bx_{[k-1]}^{-1}\neq\bx^{[k+1}]$, and in this case the letter $x_k$ is  direct if and only if 
  $\bx_{[k-1]}^{-1}\succ \bx^{[k+1}]$.
\item[(s2)]
  If $\bx$ is of type $(p,p)$, then $\widebar{\bx}\in\Ba(\uQ)$ is a primitive
  band.
\end{itemize}
Then, $\bw$ is \emph{\gls{admissible}} (with respect to $F$) if it fulfils
conditions (s1) and (s2). 

Similarly, consider for a band $\bx=x_0x_1\cdots x_l\in\Ba(\huQ)$
the following two conditions:
\begin{itemize}
\item[(b1)]  
For any $k\in\{1,2,\ldots, l\}$ with  $F(x_k)\in\Qsp$ we have
$(\bx^{[k-1]})^{-1}\bx_{[k-1]}^{-1}\neq\bx^{[k+1]}\bx_{[k-1]}$,
and in this situation $x_k$ is  direct if and only if
$(\bx^{[k-1]})^{-1}\bx_{[k-1]}^{-1}\succ\bx^{[k+1]}\bx_{[k-1]}$,
\item[(b2)]  
$F(\bx)\in\Ba(\uQ)$ is primitive and asymmetric.     
\end{itemize}
Then $\bx$ is \emph{\gls{admissible}} (with respect to $F$), if the  conditions
(b1) and (b2) are fulfiled.

We denote by $\glslink{AdmSBQ}{\AdmSt(\uQ)}\subset\St(\huQ)$ resp. 
$\glslink{AdmSBQ}{\AdmBa(\uQ)}\subset\pBa(\huQ)$ the set of strings and bands for $\huQ$ which are admissible with respect to $F$.
Moreover, $\glslink{AdmSBQ}{[\AdmBa(\uQ)]}$ denotes the equivalence classes of elements of $\AdmBa(\uQ)$ modulo rotations.

Next, we  abbreviate
$\glslink{AdmQ}{\Adm(\uQ)}:=\AdmSt(\uQ)\cup\AdmBa(\uQ)$,
and define (sub)sets
\begin{align*}
\glslink{AdmQ}{\Adm_s(\uQ)} &:=\{\bx\in\AdmSt(\uQ)\mid \bx \text{ is not of type } (p,p)\},\\
\glslink{AdmQ}{\Adm_b(\uQ)} &:=\{\bx\in\AdmSt(\uQ)\mid \bx \text{ is  of type } (p,p)\}\cup \AdmBa(\uQ),\\
\glslink{AdmQ}{[\Adm_b(\uQ)]} &:=\{\bx\in\AdmSt(\uQ)\mid \bx \text{ is  of type } (p,p)\}\cup[\AdmBa(\uQ)].
\end{align*}
\end{Def}

Obviously, for all $i\in Q_0$ the strings $\bs_i^{\pm 1}$ are admissible,
and we have $\overline{\bs}_i=\tbs_i$ for all $i\in Q_0$. It is also clear that a string $\bx$ is admissible if and only if 
$\bx^{-1}$ is admissible.

We have the following quite non-trivial result from~\cite{Ha22}.  A sketch of the proof can also be found in~\cite[Appendix~B]{Ge23}.  

\begin{Prop}[Hansper] \label{prp:Adm}
Let $\uQ$ be a skewed-gentle polarized quiver and $F\df\huQ\ra\uQ$
the canonical morphism starting in the corresponding gentle polarized
quiver $\huQ$.  Then $\bx\in\St(\huQ)\cup\pBa(\huQ)$ is admissible
if and only if $\bx=A(\bw)^{\pm 1}$ for some $\bw\in\St(\uQ)\cup\pBa'(\uQ)$.

In particular, all the $A(\bw)$ are admissible strings or bands, and we have obviously $\widebar{A(\bw)}=\bw$.  
\end{Prop}

\begin{Rem} \label{rem:Adm}
The operation  $\widebar{?}$ induces a map
\[
  (\widebar{?})_s \df\Adm_s(\uQ)\ra\St(\uQ), \bx\mapsto\widebar{\bx}.
\]
By  Proposition~\ref{prp:Adm}, this map is surjective and
$(\widebar{\bx})_s=(\widebar{\by})_s$ implies either $\bx=\by$, or $\bx=\by^{-1}$
maps to  a symmetric string.  
Similarly, $\widebar{?}$ induces a map
\[
  [\widebar{?}]\df [\Adm_b(\uQ)]\ra [\pBa(\uQ)],
\]
which is also surjective.  Moreover, $[\widebar{\bx}]=[\widebar{\by}]$ implies
$\bx=\by$, or $\bx=\by^{-1}$ maps to a symmetric band in $\pBa'(\uQ)$, again
by Proposition~\ref{prp:Adm}. The situation may be visualized as follows:\\[.5cm]
\hspace*{2cm}\begin{tikzpicture} 
\node at (0,0)  {$\St(\huQ)$} ;
\node at (1,.05)  {$\supset$} ;
\node at (2.2,0)  {$\AdmSt(\uQ)$} ;
  \matrix at (4,-.3)
  {\node(l1){$(u,u)$};\\ \node{$(u,p)$}; \\ \node(l3){$(p,u)$};\\
    \node(l4){$(p,p)$};\\ \node(l5){$\phantom{(x,x)}$};\\};
\draw[decorate,decoration=brace] (l1.east) -- (l3.east);
\draw[decorate,decoration=brace] (l4.north east) -- (l5.south east);
\draw[decorate,decoration=brace] (l4.south west) -- (l1.north west);

\node at (0,-1.8)   {$\pBa(\huQ)$};
\node at (1,-1.75)  {$\supset$} ;
\node at (2.2,-1.8) {$\AdmBa(\uQ)$} ;
\path (5.8,.3) node[outer xsep=1](As) {$\Adm_s(\uQ)$} -- +(3, 0) node(St){$\St(\uQ)$} --  ++(0,-1.6)  node[outer xsep=1](Ab) {$\Adm_b(\uQ)$} -- ++(3,0) node(pB){$\pBa'(\uQ)$};
\draw[thick, ->] (As) -- node[above]{$\overline{?}$} (St);
\draw[thick, ->] (Ab) -- node[above]{$\overline{?}$} (pB);
\draw[rounded corners=4pt] (1.2,1.4) rectangle (6.6,-2.1);
\node at (3.9,-2.6) {$\Adm(\uQ)$};
\end{tikzpicture}
\end{Rem}

\subsection{The operator \texorpdfstring{$\tau_f$}{τf}} \label{ssec:AR-adm}
Inspired by~\cite{BY20} and~\cite{BR87} we introduce, relative to a fringing $\uQ\subset\uQ^f$, the combinatorial fringe AR-translate
$\tau_f\df\Adm(\uQ)\ra\Adm(\uQ^f)$ as a convenient device to calculate later on the combinatorial $E$-invariant in symmetric way. 
This is crucial for our goal to relate later on  the 
$E$-invariant with our intersection number. 
See Remark~\ref{rem:kisses} (2), Section~\ref{ssec:combEg} and Theorem~\ref{thm:intersec}.
Let $\uQ$ be a skewed-gentle polarized quiver with fringing $\uQ^f$, and $\huQ$
resp.~$\huQ^f$ the corresponding gentle quivers from Section~\ref{ssec:huQA}.
For each arrow $\alp\in\huQ^f_1$, let $\bh'_\alp$ be the longest right inextensible word for $\huQ^f$ with the following two properties:
\begin{itemize}
\item $\ut(\bh'_\alp)=-\ut(\alp)$,
\item $\bh'_\alp$ contains no inverse letters.
\end{itemize}
In particular, $\bh_\alp:=\alp^{-1}\bh'_\alp$ is a right inextensible word.

Let $\bx=\bx'\II_{i,\rho}\in\St(\uQ)$. By the construction of 
$\uQ^f$, there exists a unique $\alp(\bx)\in\huQ^f_1$ with 
$\us(\alp(\bx))=(i,\rho)$, and thus
$\bx'\bh_{\alp(\bx)} \in\St(\huQ^f)$.  We define now successively
\begin{align*}
  \bx[1]_f &:= \begin{cases} \bx'\bh_{\alp(\bx)} &\text{if } \II_{i,\rho}
    \text{ is unpunctured,}\\
     \bx &\text{if } \II_{i,\rho} \text{ is punctured,}
   \end{cases}\\
   {_f[1]}\bx &:= (\bx^{-1}[1]_f)^{-1},\\ 
  \tau_f(\bx) &:= {_f[1](\bx[1]_f)}= ({_f[1]\bx})[1]_f.
\end{align*}                
Moreover, we agree that $\tau_f(\bx)=\bx$ for all $\bx\in\pBa(\huQ)$.
We leave it as an exercise to show that the operator $\tau_f$ restricts to an injective map from $\Adm(\uQ)$
to $\Adm(\uQ^f)$.  With our notation we have moreover 
$\tau_f(\bx)=\bx$ for all $\bx\in\Adm_b(\uQ)$. 
See Section~\ref{sssec:ExTauf} below for an illustration
of this construction in the context of our running example.

\subsection{The category of windings}
\label{ssec:windings}
A \emph{homomorphism} $G\df\uH\ra\uQ$ between skewed-gentle polarized quivers is given by a pair of maps $G=(G_0, G_1)$  
with $G_0\df H_0\ra Q_0$, $G_1\df H_1\ra Q_1$ and $G_1(H_1^\spe)\subset Q_1^\spe$,  such that
\[
  \us(G_1(\alp))=(G_0(s(\alp)),s_1(\alp)),\quad\text{and}\quad
  \ut(G_1(\alp))=(G_0(t(\alp)),t_1(\alp))
\]
for all arrows $\alp\in H_1$.
The homomorphism $G$ is called \emph{\gls{strict}}, if it sends ordinary arrows to ordinary arrows.

By definition, a \emph{\gls{morphism}} $\phi\df G\ra G'$ between homomorphisms of skewed-gentle polarized quivers
$G\df\uH\ra\uQ$ and $G'\df\uH'\ra\uQ$ is given by a
strict homomorphism
$\phi\df\uH\ra\uH'$ of skewed-gentle  polarized quivers, such that $G'\phi=G$.

If in this situation $\uH$  and $\uH'$ are  connected hereditary and $\uH$ is
moreover of type $\sftA_n$,
then $\phi$ is surjective, and  $\uH'$ is
of type $\sftA_m$ with $m\leq n$. 
If $\uH$ is of type $\sftD'$, then $\phi$ is already an isomorphism.

A \emph{\gls{winding}} of the skewed-gentle quiver $\uQ$ is a homomorphism
$G\df\uH \ra\uQ$ of skewed-gentle polarized quivers, where $\uH$ is a connected hereditary polarized quiver.  
If $\uH$ is of type $\sftA$, we request moreover,
that each morphism which starts in $G$ is in fact an isomorphism.
In view of the above remark, our notion of 
windings is compatible  with the one from~\cite{Kr91}.

For a fixed, skewed-gentle polarized quiver $\uQ$, we are interested in the category of windings $G\df\uH \ra\uQ$.      

\subsection{From words to windings}
\label{ssec:Hw}
Let from now on $\uQ$ be a skewed-gentle polarized quiver.

Following closely~\cite[Sec.~3]{CB89b},
we construct for each  $\bx\in\St(\huQ)\cup\Ba(\huQ)$ a winding
\[
  \glslink{GHx}{G_\bx}\df \glslink{GHx}{\uH(\bx)} \ra\uQ.  
\]
We consider first the case when $\bx=x_0x_1\cdots x_n$
is a string for $\huQ$.  Then 
the underlying graph of $H(\bx)$ is a subgraph of the diagram
$\sftD'_n$ from Table~\ref{tab:AD-graphs}.  
If $x_0$ is unpunctured,
we remove the loop $\eta_0$, else $G(\eta_0)=\eps_{s(x_0)}^*$.  
Similarly, if $x_n$ is unpunctured, the loop $\eta_1$ is removed and else
$G_\bx(\eta_1)=\eps_{t(x_n)}^*$.
For $i=1,2,\ldots, n-1$ the edge $\nu_i$ is oriented to the left if and only if
$w_i$ is a direct letter, and we set $G_\bx(\nu_i)=x_i$.
Else, $\nu_i$ is oriented to the right, and we set $G_\bx(\nu_i)=x_i^{-1}$.
Finally, $G_\bx(i):=t(x_i)$ for $i=1,2,\ldots, n$.

If $\bx=x_0x_1\cdots x_n$ is a band for $\huQ$, the underlying graph of
$H(\bx)$ is of type $\sftA_n$ from Table~\ref{tab:AD-graphs}. In this case, all letters of $\bx$ are direct or inverse ordinary letters. 
For $i=0,1, \ldots, n$ we orient $\nu_i$ anti-clockwise if $x_i$ is a direct letter
and set $G_\bx(\nu_i)=x_i$. Else, $\nu_i$ is oriented clockwise, and we set
$G_\bx(\nu_i)=x_i^{-1}$. As for the vertices, we set $G_\bx(i):=t(x_i)$
for $i=0, 1, \ldots, n$.

We can upgrade $H(\bx)$ to a hereditary skewed-gentle polarized quiver
$\uH(\bx)$ by declaring all 
loops to be special, and taking $s_1(\alp)=s_1(G_\bx(\alp))$
as well as $t_1(\alp)=t_1(G(\alp))$ for all arrows $\alp\in H_1(\bx)$. 

The morphisms $G_\bx\df\uH(\bx)\ra\uQ$ with
$\bx\in\St(\uQ)\cup\pBa(\uQ)$, which we just constructed,
are special cases of windings. See Section~\ref{sssec:ExWindings} below for several examples of windings in the context of our running example~\ref{sssec:skgep}.

\subsection{(Polarized) boundary vertices} \label{ssec:PBV}
Let $\uH$ be a hereditary polarized quiver. 
We say that $i\in H_0$ is a \emph{\glslink{bvertex}{boundary vertex}}
if 
\[
\abs{\{\alp\in H_1\mid s(\alp)=i\} \cup\{\bet\in H_1\mid t(\bet)=i\}}\leq 1.
\]
Similarly, $(i,\rho)\in H_0\times\{-1,+1\}$ is a
\emph{\glslink{bvertex}{polarized} boundary vertex},
if $i$ is a boundary vertex,
and $(i,\rho)\not\in\Ima(\us)\cup\Ima(\ut)$.  
Thus, $\uH$ has exactly two polarized boundary vertices if  its underlying graph is of type $\sfA_n$, and exactly one polarized boundary vertex if the underlying graph is of type $\sfD'_n$. This is in particular the case for $n=1$.  
Note, that the trivial quiver of type $\sfA_1$ has obviously only one boundary vertex, but \emph{two} polarized boundary vertices.
If the underlying graph of $\uH$ is of type $\sftA_n$ or of type $\sftD'_n$, then $\uH$ has no (polarized) boundary vertices.

If, for example, 
$\bx=\II^{-1}_{(i,\rho)}x_1 x_2\cdots x_{n-1}\II_{(i',\rho')}\in\St(\huQ)$ is of type $(u,u)$, the hereditary polarized quiver $\uH(\bx)$ from Section~\ref{ssec:Hw} is of type  $\sfA_n$.  
It has exactly $(1,\rho)$ and $(n,\rho')$ as polarized boundary vertices.
On the other hand, for $i\in\Qspv$, we can consider the hereditary polarized quiver $\uH$ of type $\sfA_1$ with the single (ordinary) vertex $1$, and $G\df\uH\ra\uQ, 1\mapsto i$.  In this case, we have the two polarized boundary vertices $(1,\pm 1)$, and notice that $G$ is \emph{not} of the form $G_\bx$ for some $\bx\in\St(\huQ)$.

\subsection{H-triples and K-triples} \label{ssec:triph}
In~\cite{Kr91} Krause introduced ``admissable'' 
triples as a framework for
the combinatorial description of homomorphisms between representations
of gentle algebras. We adapt his construction here to our setting of
skewed-gentle polarized quivers.

We consider the following properties  of a morphism  
$\phi\df G\ra G'$ of windings $G\df\uH\ra\uQ$ and
$G'\df\uH'\ra\uQ$.  
\begin{itemize}
\item[(q)] If $(j,\rho) \in H_0\times\{-1,1\}$ is a polarized boundary vertex, 
  then there exists no arrow $\alp\in (H'_1)^\ord$ with
  $\ut_{\uH'}(\alp)=(\phi(j),\rho)$. 
\item[(s)] If $(j,\rho)\in H_0\times\{-1,+1\}$ is a polarized boundary vertex,  then there exists no arrow 
$\bet\in (H'_1)^\ord$ with $\us_{\uH'}(\bet)=(\phi(j),\rho)$.
\item[(k)]
  If $j\in H_0$ is a boundary vertex of $\uH$, then $\phi(j)\in H'_0$
  is \emph{not} a boundary vertex of $\uH'$.
\end{itemize}

For $\bx,\by\in\Adm(\uQ)$ an \emph{\gls{H-triple}} 
$(G,\phi_q,\phi_s)$ for $(\bx,\by)$ consists of a winding $G\df\uH_G \ra\uQ$, together with a (strict) morphism 
$\phi_q\df G\ra G_\bx$ with property (q), and a  strict morphism 
$\phi_s\df G\ra G_\by$ with property (s).
Moreover, we request that in case $\phi_q(j)\in H_0^\spe(\bx)$
\emph{and} $\phi_s(j)\in H_0^\spe(\by)$ that 
$j\in (H_G)_0^\spe$.  
Recall, that by definition, $\phi_q$ and $\phi_s$ are strict homomorphisms of (hereditary) polarized quivers.

An H-triple $(G,\phi_q,\phi_s)$ is a \emph{\gls{K-triple}}, if 
$\phi_q$ and $\phi_s$ both have moreover property (k).

Two H-triples $(G,\phi_q,\phi_s)$ and $(G',\phi'_q,\phi'_s)$ are
\emph{equivalent}, if there is an isomorphism of windings $\psi\df G\ra G'$
such that $\phi_q=\phi'_q\psi$ and $\phi_s=\phi'_s\psi$.  In other words,
in the following diagram all triangles must be commutative. 
\[
  \xymatrix{&\uH_G\ar_{\phi_q}[ld]\ar^G[d]\ar^{\phi_s}[rd]\ar@/_1.2pc/_(.62){\psi}[dd]\\
    \uH(\bx)\ar^{G_\bx}[r] &\uQ &\ar_{G_\by}[l] \uH(\by)\\
   &\uH_{G'}\ar^{\phi'_q}[lu]\ar_{G'}[u]\ar_{\phi'_s}[ru] }
\]
We denote the set of equivalence classes of H-triples for 
$(\bx,\by)$ by $\gls{cHQxy}$. 
The subset of  $\cH_\uQ(\bx,\by)$, which consists of the equivalence classes of K-triples, is denoted by
$\glslink{cKQxy}{\cK_\uQ(\bx,\by)}$.  
 The elements of $\cK_\uQ(\bx, \by)$, i.e. the equivalence classes of K-triples, are called \emph{kisses}.

For $(G,\phi_q,\phi_s)\in\cH_\uQ(\bx,\by)$ with
$G\df\uH_G\ra\uQ$, the \emph{\gls{type}} of this triple (or for short of $G$) is the type of the underlying graph of 
$\uH_G$. Thus, the type of
$(G,\phi_q,\phi_s)$ belongs to the list from Table~\ref{tab:AD-graphs}.
Clearly, the type of a  triple is invariant under equivalence.
For later, use we define the set of K-triples of type $\sfA$ as
\[
  \glslink{cKQxy}{\cK_\uQ(\bx,\by)_\sfA}:=\{[(G,\phi_q,\phi_s)]\in\cK_\uQ(\bx,\by)\mid
  \uH_G\text{ is of type } \sfA_m \text{ for some } m\}.
\]

\begin{Def} \label{def:bdyvert}
  Let $\bx, \by\in\Adm(\uQ)$ and 
  $(G,\phi_q,\phi_s)\in\cK_\uQ(\bx,\by)$
  a K-triple with $G\df\uH\ra\uQ$.  If $(i,\rho)$ is a polarized boundary
  vertex of $\uH$,  it follows from the definitions, that there exists an unique arrow $\mu\in H_1(\bx)$ with $\us(\mu)=(\phi_q(i),\rho)$, and an unique arrow $\nu\in H_1(\by)$ with $\ut(\nu)=(\phi_s(i),\rho)$.
  Clearly $\us_\uQ (G_\bx(\mu))=(G(i),\rho)=\ut_\uQ(G_\by(\nu))$. It is easy to see that precisely one of the following four cases occurs,
  see Remark~\ref{rem:kisses}~(3) below:
  \begin{enumerate}[(I)]
  \item
    $\{G_\bx(\mu), G_\by(\nu)\}\subset \Qord$,
  \item
    $G_\bx(\mu)=G_\by(\nu)\in\Qsp$ with $\mu$ and $\nu$ ordinary,
  \item
    $G_\bx(\mu)=G_\by(\nu)\in\Qsp$ with $\mu$ ordinary and $\nu$ special,
 \item
    $G_\bx(\mu)=G_\by(\nu)\in\Qsp$ with $\mu$ special an $\nu$ ordinary.
  \end{enumerate}
Accordingly, we will refer to these Roman numbers as the 
\emph{type} of the polarized boundary vertex 
$(i,\rho)\in H_0\times\{-1,+1\}$ of our K-triple.  It follows from Section~\ref{ssec:PBV} that a kiss of type $\sfA$ has precisely two polarized boundary vertices, and a kiss of type $\sfD'$ has precisely one polarized boundary vertex. A kiss of type $\sftA$ or of type $\sftD'$ has no boundary vertex.
\end{Def}
See Section~\ref{sssec:Ex-H-K-trip} below for several examples of H-triples and K-triples, based on the development of our running example in Section~\ref{ssec:expl-uQ}. We illustrate there also the different cases mentioned in Definition~\ref{def:bdyvert}.

\begin{Rem} \label{rem:kisses}
(1) Clearly $\cK_\uQ(\bx,\by)$ is our generalization of the
kisses from~\cite{BY20} to the context of skewed-gentle quivers. The somehow
more bulky notion of triples, inspired by~\cite{Kr91}, is necessary to deal
properly with bands.

(2) Let $\uQ^f$ be a fringing of $\uQ$ and recall from Section~\ref{ssec:AR-adm}
our construction of the AR-operator $\tau_f\df\Adm(\uQ)\ra\Adm(\uQ^f)$.
Let moreover $\cP_\uQ:=
\{A(\tilde{\bp})\mid\tilde{\bp}\in\widetilde{\cP}_\uQ\}$
be the set of \emph{admissible projective strings}, where
$\widetilde{\cP}_\uQ$ is the set of projective strings from
Section~\ref{ssec:StBa}. Similarly, 
$\cI_\uQ:=\{A(\tilde{\bq})\mid\tilde{\bq}\in\widetilde{\cI}_\uQ\}$
is the set of \emph{admissible injective strings}.
With a little more notation, one can define a bijective operator
$\tau\df\Adm(\uQ)\setminus\cP_\uQ\ra\Adm(\uQ)\setminus\cI_\uQ$, 
which is closely related to the Auslander-Reiten translation in the module category of the skewed-gentle algebra $\Ka\uQ$ associated to $\uQ$,
see~\cite[Sec.~2.6 \& 3.4]{Ge23}.  

In~\cite[Sec.~4.9]{Ge23}
it is shown, that there is a natural bijection
\[
  \cH_\uQ(\bx,\tau(\by))\ra\cK_{\uQ^f}(\tau_f(\bx), \tau_f(\by))
\]
for all $\bx\in\Adm(\uQ)$ and $\by\in\Adm(\uQ)\setminus\cP_\uQ$, where
$\cP_\uQ$ denotes the finite set of \emph{projective} admissible words.
If $\by\in\cP_\uQ$ we have moreover
$\cK_{\uQ^f}(\tau_f(\bx),\tau_f(\by))=\emptyset$ for all $\bx\in\Adm(\uQ)$.  

(3) Obviously, either $\{G_\bx(\mu), G_\by(\nu)\}\subset\Qord$ or
$\{G_\bx(\mu), G_\by(\nu)\}\cap\Qsp\neq\emptyset$. In the first
case we have type (I).  In the second case, we have
$G_\bx(\mu)=G_\by(\nu)\in\Qsp$. In fact, if say $G_\bx(\mu)\in\Qsp$, it
follows from the definitions that 
\[
\us_\uQ(G_\bx(\mu))=(G_\bx(\phi_q(j)),\rho)=(G_\by(\phi_s(j)),\rho)=\ut_\uQ(G_\by(\nu)),
\]
which forces $\rho=-1$ and $G_\bx(\phi_q(j))\in\Qspv$ since
$G_\bx(\mu)\in\Qsp$ by our hypothesis. However, by the definition of
polarized skewed-gentle quivers, this means that $G_\bx(\mu)$ is
the unique arrow $\eps$ with $\ut_\uQ(\eps)=(G_\by(\phi_s(j)),-1)$, which
implies $G_\by(\nu)=G_\bx(\mu)$. Similarly, $G_\by(\nu)\in\Qsp$ implies
also $G_\bx(\mu)=G_\by(\nu)$.
Now, since each arrow is either ordinary or special, this leaves us
a priori with exactly $4=2\times 2$ possibilities. However, it is in
this situation not possible to have $\mu$ and $\nu$ both special, since,
by the definition of an H-triple, in this case $j\in H_0^{\mathrm{sp}}$,
and thus $j$ wouldn't be a boundary vertex. Thus, in case
$\{G_\bx(\mu), G_\by(\nu)\}\not\subset\Qord$ we are indeed left with
precisely one of the types (II), (III) or (IV).
\end{Rem}

\subsection{Combinatorial invariants} \label{ssect: CombIn}
Let again  $\uQ^f$ be a fringing for our skewed-gentle polarized quiver $\uQ$.
\begin{Def}
For $\bx,\by\in\Adm(\uQ)$ we will consider the set
\begin{align*}
\glslink{APQfxy}{\sfA_{\uQ^f}(\bx,\by)} &:=          
  \cK_{\uQ^f}(\tau_f(\bx),\tau_f(\by))_\sfA \coprod
  \cK_{\uQ^f}(\tau_f(\by),\tau_f(\bx))_\sfA,\\
\intertext{of equivalence classes of K-triples of type $\sfA$ between $\tau_f(\bx)$ and
  $\tau_f(\by)$, and the set}
\glslink{APQfxy}{\sfP_\uQ(\bx,\by)} &:=\{(j,i)\in\{0,1\}\times\{0,1\}\mid
(\eta_j,\eta_i)\in\Hsp(\by)\times\Hsp(\bx)\text{ and }\\
&\qquad G_\by(\eta_j)=G_\bx(\eta_i)\} \setminus \sfP'_\uQ(\bx,\by),\text{ with}\\
\sfP'_\uQ(\bx,\by)&:=\begin{cases}
\{(0,0),(1,1)\}&\text{ if } \bx=\by^{\phantom{-1}}\in\AdmSt(\uQ)\\
\{(1,0),(0,1)\}&\text{ if } \bx=\by^{-1}\in\AdmSt(\uQ)\\
   \qquad\emptyset       &\text{ else.}
 \end{cases}
\intertext{ of \emph{pairs of punctured letters of type} $\sfD'$, as well as}
\glslink{APQfxy}{\Diag_b(\bx,\by)}&:=\begin{cases}
\phantom{-}  1 &\text{if } [\bx]=[\by]\quad\in\Adm_b(\uQ),\\
 -1 &\text{if } [\bx]=[\by^{-1}]\in\Adm_b(\uQ),\\
\phantom{-} 0  &\text{else,}
\end{cases}
\intertext{the \emph{band orientation.}}
\end{align*}                 
\end{Def}

\begin{Rem}
(1) The cardinality of $\sfA_{\uQ^f}(\bx,\by)$ is closely related to the
module-theoretic $E$-invariant for representations of the skewed-gentle algebra~$\Ka\uQ$ associated to $\uQ$.  
It does \emph{not} depend on the choice of the fringing 
$\uQ^f$, as discussed in~\cite[Rem.~5.3]{Ge23}.

(2)  By~\cite[Lem.~5.3]{Ge23} we have a natural bijection
\[
  \sfP_\uQ(\bx,\by)\ra
  \cK_{\uQ^f}(\tau_f(\bx), \tau_f(\by))_{\sfD'}\coprod
  \cK_{\uQ^f}(\tau_f(\by), \tau_f(\bx))_{\sfD'},
\]
where $\cK_{\uQ}(\tau_f(\bx),\tau_f(\by))_{\sfD'}$ denotes the set of K-triples
of type $\sfD'$ from $\tau_f(\bx)$ to $\tau_f(\by)$.

(3)  See~\cite[Rem.~5.10~(2)]{Ge23} for an interpretation of 
$\Diag_b(\bx,\by)$.
\end{Rem}

\begin{Def} \label{def:ADpm}
  For $\bx\in\Adm(\uQ)$  and $i\in Q_0$, we set 
  \begin{alignat*}{2}
\glslink{ADx}{A_i^+(\bx)} & :=\{j\in H_0(\tau_f(\bx)) &\mid G_{\tau_f(\bx)}(j)&=i
            \text{ and two ordinary arrows start at } j\},\\
\glslink{ADx}{A_i^-(\bx)} & :=\{j\in H_0(\tau_f(\bx))&\mid G_{\tau_f(\bx)}(j)&=i
            \text{ and two ordinary arrows end at } j\},\\
\glslink{ADx}{D_i^+(\bx)} & :=\{\eta_j\in H_1^{\mathrm{sp}}(\tau_f(\bx))&\ \mid G_{\tau_f(\bx)}(\eta_j)&=\eps_i
\text{ and } \{\alp\in H_1^{\mathrm{ord}}(\tau_f(\bx)) \mid
             t(\alp)=t(\eta_j)\}=\emptyset \},\\
\glslink{ADx}{D_i^-(\bx)} & :=\{\eta_j\in H_1^{\mathrm{sp}}(\tau_f(\bx))&\ \mid G_{\tau_f(\bx)}(\eta_j)&=\eps_i
\text{ and  } \{\bet\in H_1^{\mathrm{ord}}(\tau_f(\bx)) \mid
             s(\bet)=s(\eta_j)\}=\emptyset \}.    
\end{alignat*}
Note, that $D_i^\pm(\bx)=\emptyset$ if $i\in\Qordv$.
\end{Def}

\begin{Rem} Suppose that $\bx\in\Adm(\uQ)$ is not projective in the sense of Remark~\ref{rem:kisses}~(2), then we have
  by~\cite[Lem~5.14]{Ge23} the following equalities:
\begin{align*}
    \abs{A^+_i(\bx)}&=\abs{\cH_\uQ(\bx,\bs_i)_\sfA},&
    \abs{A^-_i(\bx)}&=\abs{\cH_\uQ(\bs_i,\tau(\bx))_\sfA},\\
    \abs{D^+_i(\bx)}&=\abs{\cH_\uQ(\bx,\bs_i)_{\sfD'}},&
    \abs{D^-_i(\bx)}&=\abs{\cH_\uQ(\bs_i,\tau(\bx))_{\sfD'}},
\end{align*}
where we used the ``simple'' strings $\bs_i$ from Section~\ref{ssec:huQA}.
However, we do not need this fact here. 
\end{Rem}

\subsection{Combinatorial E-invariant and g-vectors for decorated admissible words} \label{ssec:combEg}
Let again $\uQ^f$ be a fringing of our skewed-gentle polarized quiver $\uQ$.
For each $\bx\in\Adm(\uQ)$ we define a subset of
$S:=(\{-1,1\}\times\{-1,1\})\cup\{\ast, (\ast,\ast)\}$: 
\[
{\gls{S(x)}} :=\begin{cases}
\{(1,1)\}                        &\text{if } \type(\bx)=(u,u),\\
\{(1,-1), (1,1)\}                &\text{if } \type(\bx)=(u,p),\\
\{(-1,1), (1,1)\}                &\text{if } \type(\bx)=(p,u),\\
\{1,-1\}^2\cup\{(\ast,\ast)\}    &\text{if } \type(\bx)=(p,p),\\
\{\ast\}                &\text{if } \type(\bx)=(b).
\end{cases}
\]
For later use, we introduce several operations  on the elements of $S$, namely
\begin{alignat*}{2}
\glslink{wtix}{\wt(s)} &:=\begin{cases} 2 &\text{if } s=(*,*),\\ 1 &\text{else,}\end{cases} &
 \quad (s_1,s_2) &:= s \text{ for } s\in S\setminus\{*\},\\ 
\glslink{wtix}{s^\iota} &:=\begin{cases} (s_2,s_1)&\text{if } s\in\{-1,1\}^2,\\ s &\text{else}
  \end{cases} &
\glslink{wtix}{s^\chi} &:=\begin{cases} (-s_1,-s_2)&\text{if } s\in\{-1,1\}^2,\\ \ s&\text{else.}
   \end{cases}
\end{alignat*}
Moreover, we introduce symbols $\gls{-bsipm}$ 
for $i\in Q_0$, which represent (together with their decoration), in the spirit of~\cite[(4.8)]{DWZ2}, the \gls{negativesimple} representations of $\Ka\uQ$. 
With this at hand, we define, with the notation from Remark~\ref{rem:Adm}
\begin{align*}
\glslink{DAdm*Q}{\Adm^*(\uQ)}   &:=\{(\bx,s)\mid \bx\in\Adm(\uQ)\text{ and } s\in S(\bx)\},
  \text{ and}\\
\NeSi(\uQ)&:=\{(-\bs_i^{\pm 1},s)\mid i\in Q_0\text{ and }s\in S(\bs_i^{\pm 1})\}\\
\glslink{DAdm*Q}{\DAdm^*(\uQ)} &:=\Adm^*(\uQ)\cup\NeSi(\uQ),\\
  [\Adm^*(\uQ)] &:=\{(\bx,s)\in\Adm^*(\uQ)\mid\bx\in\AdmSt(\uQ)\} \cup
                  [\AdmBa(\uQ)]\times\{*\},\\
\glslink{DAdm*Q}{[\DAdm^*(\uQ)]} &:=[\Adm^*(\uQ)] \cup \NeSi(\uQ)                                \end{align*}
We introduce on $[\Adm^*(\uQ)]$, and on $[\DAdm^*(\uQ)]$
an equivalence relation $\simeq$ by
declaring $(\bx^{-1},s^\iota)\simeq (\bx,s)$. We denote the corresponding set of equivalence classes by $[\Adm^*(\uQ)]/_{\simeq}$ and $[\DAdm^*(\uQ)]/_{\simeq}$, respectively. Note, that 
with the definitions from~\eqref{eqn:tQ},
\[
  \tq\df\NeSi(\uQ)\ra \tQ_0\quad
  (-\bs_i,s)\mapsto \begin{cases}
    (i, s'_2) &\text{if } i\in\Qspv,\\
    (i, o)           &\text{if } i\in\Qordv,
  \end{cases}
\]
induces a bijection from  $\NeSi(\uQ)/_{\simeq}$ to $\tQ_0$, where we agree that for $s=(1, s_2)\in S(\bs_i)$ we have
$s'_2=+$ if $s_2=+1$ and $s'_2=-$ if $s_2=-1$. 

\begin{Def} \label{def:dimv}
Recall the construction of $\tQ_0$ in~\eqref{eqn:tQ}, see 
also~\cite[Sec.~5.1]{Ge23} for more details.
Recall also that in Section~\ref{ssec:Hw} we introduced for each
$\bx\in\Adm(\uQ)$ a winding $G_\bx\df\uH(\bx)\ra\uQ$. Our aim is to
introduce the \emph{dimension vector}
$\bd(\bx,s)\in\NN^{\tQ_0}$ for each $(\bx,s)\in\Adm^*(\uQ)$.
To this end, we introduce for each $\bx\in\Adm(\uQ)$ and $j\in H_0^\spe(\bx)$
the auxiliary vectors $\bd'(\bx), \bd^{(j)}(\bx)\in\QQ^{\tQ_0}$, which we define componentwise:
\begin{align*}
  \bd'(\bx)_{(i,\rho)} &:=\begin{cases}
\phantom{\frac{1}{2}}\abs{G^{-1}_\bx(i)}    &\text{if } i\in\Qordv,\\
\frac{1}{2}\abs{G^{-1}_\bx(i)} &\text{if } i\in\Qspv,
\end{cases}\qquad \text{ and }\\    
\bd^{(j)}(\bx)_{(i,\rho)} &:= \begin{cases}
\phantom{-}\frac{1}{2} &\text{if } (i,\rho)=(G_\bx(j),+),\\
-\frac{1}{2} &\text{if } (i,\rho)=(G_\bx(j),-).
\end{cases}
\end{align*}
Now, we can define
\[
{\gls{bd(xs)}}  :=\begin{cases}  
    \bd'(\bx)  &\text{if } \type(\bx)\in\{ (u,u), (b)\},\\
    \bd'(\bx)+s_2\bd^{(n)}(\bx) &\text{if } \type(\bx)=(u,p),\\
    s_1\bd^{(1)}(\bx)+\bd'(\bx) &\text{if } \type(\bx)=(p,u),\\
    s_1\bd^{(1)}(\bx)+\bd'(\bx)+s_2\bd^{(n)}(\bx) &\text{if } \type(\bx)=(p,p)
    \text{ and } s\neq (*,*),\\
     2\bd'(\bx) &\text{if } \type(\bx)=(p,p) \text{ and } s=(*,*),
\end{cases}
\]
where we assume that the underlying graph of $\uH(\bx)$ is drawn as in
Figure~\ref{tab:AD-graphs} and has $n$ vertices (for $\type(\bx)\neq (b)$). 
\end{Def}

Note, that we have
\[
  \bd(\bs_i,u)_{(j,\rho)}= \del_{\tq(\bs_i, u),(j,\rho)},
\]
where $\del$ is the usual Kronecker delta. 

\begin{Def} \label{def:comb-inv}
For 
  $\{(\bx,s), (\by,t)\}\subset\Adm^*(\uQ)$, 
  $\{(-\bs_j,u), (-\bs_l,v)\}\subset\NeSi(\uQ)$ and
  $(i,\rho)\in\tQ_0$ we define with the above notations   the
  \emph{combinatorial $E$-invariant} and \emph{combinatorial g-vector}
  on $\DAdm^*(\uQ)$ as follows:
\begin{align*}
  \begin{split}
\glslink{cegQ}{e_\uQ}((\bx,s), (\by,t)) &:= \abs{\sfA_{\uQ^f}(\bx,\by)}\times \wt(s)\cdot\wt(t)\; \\
&\qquad +  \sum_{(j,i)\in \sfP_\uQ(\bx,\by)} \!\!\!
d_2(s_{i+1},(t^\chi)_{j+1}) \;
+\; 2\abs{\Diag_b(\bx,\by)} \cdot d_3(s',t^\chi)\quad\in\NN,
\end{split}\\[1ex]
\glslink{cegQ}{e_\uQ}((\bx,s), (-\bs_j,u)) &:=\bd(\bx,\bs)_{\tq(\bs_j,u)},\\
\glslink{cegQ}{e_\uQ}((-\bs_j,u),(-\bs_l,v)) &:= 0,\\
  \begin{split}  
\glslink{cegQ}{\bg_\uQ}(\bx,s)_{(i,\rho)}&:= (\abs{A_i^-(\bx)}-\abs{A_i^+(\bx)})\times \wt(s)\\
  &\qquad
  +\sum_{\eta_j\in D^-_i(\bx)} d_2(\rho\cdot 1, -s_{j+1})
  -\sum_{\eta_j\in D^+_i(\bx)} d_2(s_{j+1},\rho\cdot 1)\ \in\ZZ,
\end{split}\\
\glslink{cegQ}{\bg_\uQ}(-\bs_j,u)_{(i,\rho)} &:= \delta_{\tq(\bs_j,u), (i,\rho)}   
\end{align*}    
where we use additionally the functions
$\glslink{d_i}{d_2}\df\{-1,1,*\}^2\ra\{0,1,2\}$, $\glslink{d_i}{d_3}\df S^2\ra\{0,1\}$, which are
defined as follows:
\[
\begin{array}{r|r|r|r|}
  d_2 & -1 & 1 & *\\ \hline
   -1 &  1 & 0 & 1\\ \hline
    1 &  0 & 1 & 1\\ \hline
    * &  1 & 1 & 2\\ \hline
\end{array}\qquad
d_3(s,t):=\begin{cases}
1 &\text{if } s=t\in\{-1,1\}^2,\\ 0 &\text{else.}
\end{cases}
\]
Moreover, we agree that $s'=(s_2,s_1)$ if $\Diag_b(\bx,\by)=-1$ and $s'=s$ else.
\medskip

We will need also the crucial subset
\begin{equation} \label{deqn:admts}
[\DAdm_\tau^*(\uQ)]:=\{[p]\in[\DAdm^*(\uQ)]\mid e_\uQ(p, p)=0\}
\end{equation}
of $\tau$-\emph{generic} decorated words in $[\DAdm^*(\uQ)]$.
\end{Def}

See Section~\ref{sssec:ExComb-gvect} for some sample calculations of combinatorial g-vectors, 
and Section~\ref{sssec:ExCombE-inv} for some sample calculations of
combinatorial E-invariants in the context of our running example. 

\begin{Rem}
Let $\Ka$ be an algebraically closed field with  $\kar(\Ka)\neq 2$. 
In~\cite[Sec.~6]{Ge23} one of the authors introduced for each $(\bx,s)\in\Adm^*(\uQ)$ a closed subset $\overline{M}^o_{\bx}$ of the representation variety  
$\rep_{\Ka\uQ}^{\bd(\bx,s)}(\Ka)$ of representations of $\Ka\uQ$ with dimension vector $\bd(\bx,s)$, which contains a natural dense set of indecomposable representations.  We extend here the definitions from~\cite[Sec.~6]{Ge23}
  from $\Adm^*(\uQ)$ to the slightly larger set
  $\DAdm^*(\uQ)=\Adm^*(\uQ)\cup\NeSi(\uQ)$.  The set 
  $\NeSi(\uQ)$ should  here be seen as a parameter set for the negative simple representations
  of $\Ka\uQ$.  One of the main results of loc.cit., \cite[Thm.~6.4]{Ge23},
  states, that the functions $e_\uQ$ and $\bg_\uQ$ calculate the generic values of the representation theoretic $E$-invariant resp. g-vector on  the sets of the form 
  $\widebar{M}^o_{(\bx,s)}$ with $(\bx,s)\in\Adm^*(\uQ)$.
  It is straight forward, that our extension here works now also in the expected way for negative simple representations.  In particular,
  $[\DAdm^*_\tau(\uQ)]/_{\simeq}$ parametrizes naturally the \emph{decorated} generically $\tau$-regular indecomposable components of the decorated representation varieties for 
  $\Ka\uQ$, and $e_\uQ$ calculates the generic value of the $E$-invariant between those components.  
\end{Rem}

\subsection{\texorpdfstring{$\tau$}{τ}-rigid strings} 
Let us slightly relax the conditions on the set $\Adm(\uQ)$ from
Definition~\ref{def:admsb}.
We say that a string $\bx=x_0x_1\cdots x_n\in\St(\huQ)$ is
\emph{weakly admissible} if it fulfils the following condition:
\begin{itemize}
\item[(s1')]
  For any $k\in\{1,2,\ldots, n-1\}$ with  $F(x_k)\in\Qsp$ and
  $\bx_{[k-1]}^{-1}\neq\bx^{[k]}$, the letter $x_k$ is 
  direct if and only if $\bx_{[k-1]}^{-1}\succ \bx_{[k+1]}$.
\end{itemize}
Similarly, we say that a band 
$\bx=x_0x_1\cdots x_n\in\pBa(\huQ)$ is \emph{weakly admissible} if it fulfils the following condition:
\begin{itemize}
  \item[(b1')]
For any $k\in\{1,2,\ldots, n\}$ with  $F(x_k)\in\Qsp$ and
$(\bx^{[k-1]})^{-1}\bx_{[k-1]}^{-1}\neq\bx^{[k+1]}\bx_{[k-1]}$ the letter $x_k$ is direct if and only if
 $(\bx^{[k-1]})^{-1}\bx_{[k-1]}^{-1}\succ\bx^{[k+1]}\bx_{[k-1]}$.    
\end{itemize}
Moreover, we denote by $\tAdm(\uQ)$ the union of weakly admissible strings and bands.  
Thus, $\Adm(\uQ)$ is a proper subset of  $\tAdm(\uQ)$.

\subsubsection*{Examples} 
Consider, for the skewed-gentle polarized
quiver from Section~\ref{sssec:skgep} below, the string
$\bx:=\II_{(1,-)}^{-1}
\gam\eps_3\gam^{-1}\II_{(1,-)}\in\St(\huQ)$. 
Obviously, $\bx$ fulfils condition (s1´), but not condition (s1). However, 
$\widebar{\bx}= (\eps_1^*\gam\eps_3^*\gam^{-1})^2$
is not a primitive band.  
So, $\bx$ does not fulfil condition (s2).

Similarly, $\bbe:=\gam\eps_1^{-1}\gam^{-1}\eps_3^{-1}\in\pBa(\huQ)$
fulfils condition (b1) and thus condition (b1'). However,
$\widebar{\bbe}=\gam\eps_1^*\gam^{-1}\eps_3^*$ is a (primitive)
\emph{symmetric} band. It would be interesting to find a primitive
band $\bbe'\in\pBa(\huQ')$ which fulfils condition
(b1') but not condition (b1).

Given $\bx\in\tAdm(\uQ)$ and a
fringing $\uQ^f$ of $\uQ$, it clearly makes sense to consider the set
\[
\sfA_{\uQ^f}(\bx):=\cK_{\uQ^f}(\tau_f(\bx),\tau_f(\bx))_\sfA
\]
of equivalence classes of kisses of type $\sfA$ from $\tau_f(\bx)$ to itself.
We have the following interesting result:

\begin{Lem} \label{lem:tadm}
  Let $\bx\in\tAdm(\uQ)$ be a string with $\sfA_{\uQ^f}(\bx)=\emptyset$.
  If moreover $\bx$  is not of the form $\bx=\by^{-1}\eps_i^{\pm 1}\by$
  for some $i\in\Qspv$ and some word $\by$,   then $\bx\in\Adm(\uQ)$.
\end{Lem}

\begin{proof}  Suppose that 
$\bx=x_0x_1\cdots x_l\in\tAdm(\uQ)\setminus\Adm(\uQ)$
is a \emph{string}, which is not of the form 
$\by^{-1}\eps_a^{\pm 1}\by$  for some word $\by$ and 
some $a\in\Qspv$. Then $\bx$ fulfils condition (s1'),
and thus not condition (s2). Since $\bx$ is a string,
it must then be of type $(p,p)$ and 
\begin{equation}\label{eqn:pf3.14}
\widebar{\bx}=F(x_0x_1\cdots x_l)F(x_{l-1}^{-1}\cdots x_1^{-1})=\bh^k
\end{equation}
for some primitive band $\bh=h_0h_1\cdots h_m\in\pBa(\uQ)$ and some $k\geq 2$. Since $\type(\bx)=(p,p)$, we have $h_0=F(x_0)=\eps_i^*$
for some $i\in \Qspv$.  Moreover, we deduce from~\eqref{eqn:pf3.14} that 
$h_1h_2\cdots h_m=h_m^{-1}\cdots h_1^{-1}$, which implies
$\bh=\eps_i^*\bh'\eps_j^*(\bh')^{-1}$ for some $j\in\Qspv$ and
$\bh'=h_1h_2\cdots h_{m'}$ with $2m'+2=m$. In particular,
$\bh\in\pBa'(\uQ)$. Using again that $\bx$ is not of
the form $\by^{-1}\eps_a^*\by$, we conclude now that $k$ must be odd, and so in particular $k\geq 3$. 
On the other hand, by the last part of Proposition~\ref{prp:Adm}, we 
have now $\widebar{\bz}=\bh$ for 
$\bz:=\II_{(i,-1)}^{-1}\bz'\II_{(j,-1)}= A(\bh)$.  

Since $\bx\in\tAdm(\uQ)$, it follows, that for some
$\del\in\{-1,1\}$ we have
\[    \bx=\II_{(i,-1)}^{-1}\bz_1\eps_j^\del\bz_2\eps_i^\del\cdots\bz_k\II_{(j,-1)},
\]
with $\bz_s=\bz'$ if $s$ is odd and $\bz_s=(\bz')^{-1}$ if $s$ is even. Without restriction, we may assume that $\del=+1$. Using the notation from
Section~\ref{ssec:triph}, there exists an H-triple
$(G,\phi_q,\phi_s)\in\sfA_{\uQ^f}(\bx)$,  where the polarized quiver 
$\uH_G$ of type $\sfA$ is obtained from
$\uH(\bz)$ by removing the loops $\eta_0$ and $\eta_1$. Moreover, $\phi_q$
sends $\uH_G$ to the copy of $\uH_G$ which sits at the right end of
$\uH(\bx)$ and $\phi_s$ sends $\uH_G$ which sits at the left end of $\uH(\bx)$.  Note that, in the sense of Definition~\ref{def:bdyvert}, the ``left'' polarized boundary vertex of $\uH_G$ is of type (III) and the ``right'' polarized boundary vertex is of type (IV).
Thus, there exists no string 
$\bx\in\tAdm(\uQ)\setminus\Adm(\uQ)$, which fulfils both
hypothesis of the Proposition.
\end{proof}

\subsection{Examples} \label{ssec:expl-uQ}
Let us illustrate several  of the notions
and constructions from this section on a running example. 

\subsubsection{A skewed-gentle polarized quiver with fringing}
 \label{sssec:skgep}
Consider the following  'polarized' quivers $Q$ and $Q^f$: 
\[\xymatrix{
\phantom{4}&&\phantom{5}\\
&2 \ar[rd]^\bet^<<<-^>>>+\\
1 \ar@(ul,dl)[]^{\eps_1}_<<-_>>-\ar[ru]^\alp^<<<+^>>>- & & 
3\ar[ll]^\gam^<<<+^>>>+\ar@(ur,dr)[]_{\eps_3}^<<-^>>-
}
\qquad\qquad
\xymatrix{
4 && 5\ar[ld]^{\bet'}^<<<+^>>>+ \\
&2 \ar[rd]^\bet^<<<-^>>>+\ar[lu]^{\alp'}^<<<+^>>>+\\
1 \ar@(ul,dl)[]^{\eps_1}_<<-_>>-\ar[ru]^\alp^<<<+^>>>- & & 
3\ar[ll]^\gam^<<<+^>>>+\ar@(ur,dr)[]_{\eps_3}^<<-^>>-
}
\]
The signs at the start- and endpoint of each quiver indicate the
value of the functions $s_1, t_1\df Q_1\ra\{-1,+1\}$. For example,
we have here $\us(\gam)=(3,+1)=\ut(\bet)$ and $\us(\eps_1)=(1,-1)=\ut(\eps_1)$. If we set moreover $\Qsp:=\{\eps_1,\eps_3\}=(Q^f_1)^{\mathrm{sp}}$, we obtain  skewed-gentle polarized quivers $\uQ$ and $\uQ^f$, since
the longest admissible paths are  in this case $\alp\eps_1\gam\eps_3\bet$
resp. $\alp'\alp\eps_1\gam\eps_3\bet\bet'$.

\subsubsection{Strings and their (linear) orderings}
\label{sssec:exStBa}
Consider first the
following list  of strings for $\uQ$: 
\begin{align*}
\br_{2,k} &:= \II_{(2,+)}^{-1}\bet^{-1}
(\eps_3^*\gam^{-1}\eps_1^*\gam)^k\eps_3^*\gam^{-1}\eps_1^*\alp^{-1}
\II_{(2,+)}, &
\br_{3,k} &:=\II_{(3,+)}^{-1}(\eps_3^*\gam^{-1}\eps_1^*\gam)^k
\eps_3^*\gam^{-1}\eps_1^*\alp^{-1}\II_{(2,+)},\\
\br'_{2,k} &:= \II_{(2,+)}^{-1}\bet^{-1}(\eps_3^*\gam^{-1}\eps_1^*
\gam)^k\eps_3^*\gam^{-1}\eps_1^*\II_{(1,+)}, &
\br'_{3,k} &:=\II_{(3,+)}^{-1}(\eps_3^*\gam^{-1}\eps_1^*\gam)^k
\eps_3^*\gam^{-1}\eps_1^*\II_{(1,+)},\\
\bq'_{2,l} &:= \II_{(2,+)}^{-1}\bet^{-1}(\eps_3^*
\gam^{-1}\eps_1^*\gam)^l\eps_3^*\II_{(3,+)}, &
\bq'_{3,l} &:=\II_{(3,+)}^{-1}(\eps_3^*
\gam^{-1}\eps_1^*\gam)^l\eps_3^*\II_{(3,+)},\\
\bq_{2,l} &:= \II_{(2,+)}^{-1}\bet^{-1}(\eps_3^*
\gam^{-1}\eps_1^*\gam)^l\eps_3^*\bet\II_{(2,+)}, &
\bq_{3,l} &:= (\bq'_{2,l})^{-1},\\
\bs_2 &:= \II_{(2,+)}^{-1}\II_{(2,-)},\\
\bp_{2,m} &:= \II_{(2,+)}^{-1}\alp
(\eps_1^*\gam\eps_3^*\gam^{-1})^m\eps_1^*\alp^{-1}\II_{(2,+)}, &
\bp_{1,m} &:= (\bp'_{2,m})^{-1},\\
\bp'_{2,m} &:= \II_{(2,+)}^{-1}\alp
(\eps_1^*\gam\eps_3^*\gam^{-1})^m\eps_1^*\II_{(1,+)}, &
\bp'_{1,m} &:= \II_{(1,+)}^{-1}
(\eps_1^*\gam\eps_3^*\gam^{-1})^m\eps_1^*\II_{(1,+)},\\
\bo'_{2,n} &:= \br_{3,n}^{-1}, & 
\bo'_{1,n} &:= (\br'_{3,n})^{-1},\\
\bo_{2,n}  &:= \br_{2,n}^{-1}, &
\bo_{1,n}  &:= (\br'_{2,n})^{-1},
\end{align*}
where $k, l, m, n\in\ZZ_{\geq 0}$.
Note that precisely the strings
$\bq_{2,l}$, $\bp_{2,m}$, $\bq'_{3,l}$ and $\bp'_{1,m}$  
(for all $l, m\in\ZZ_{\geq 0}$) are symmetric. 
Note also, that, with the notation from  Section~\ref{ssec:StBa},
we have $\tilde{\bp}_{(1,+1)}=\bp_{2,0}$, 
$\tilde{\bp}_{(2,+1)}=\bo_{2,0}=\tilde{\bq}_{(2,+1)}$  and
$\tilde{\bp}_{(3,+1)}=\bp_{2,1}$.  Similarly, we have
$\tilde{\bq}_{(1,+1)}=\bq_{2,1}$ and $\tilde{\bq}_{(3,1)}=\bq_{2,0}$.  Obviously, we have also $\tilde{\bs}_1=\bp'_{1,0}$
and $\tilde{\bs}_3=\bq'_{3,0}$. 

We leave it as an easy exercise to check that this list, together
with $\bs_2^{-1}$, provides (without repetitions) all elements of 
$\St(\uQ)$.  More precisely, we have the linearly ordered sets
$\St_{(1,+1)}(\uQ)=\{\bp_{1,m}, \bp'_{1,m}\mid m\in\ZZ_{\geq 0}\}
\cup\{\bo'_{1,n}, \bo_{1,n}\mid n\in\ZZ_{\geq 0}\}$
with 
\[
\bp_{1,m}>\bp'_{1,m}>\bp_{1,m+1}>\bo_{1,n+1}>\bo'_{1,n}>\bo_{1,n}
\]
for all $m,n\in\ZZ_{\geq 0}$.
$\St_{(2,-1)}(\uQ)=\{\bs_2^{-1}\}$,
$\St_{(2,+1)}(\uQ)=\{\br_{2,k}, \br'_{2,k}\mid k\in\ZZ_{\geq 0}\}
\cup\{\bq'_{2,l}, \bq_{2,l}\mid l\in\ZZ_{\geq 0}\}
\cup\{\bs_2\}\cup\{\bp_{2,m},\bp'_{2,m}\mid m\in\ZZ_{\geq 0}\}
\cup\{\bo'_{2,n},\bo_{2,n}\mid n\in\ZZ_{\geq 0}\}$ with
\[
\br_{2,k}>\br'_{2,k}>\br_{2,k+1}>\bq_{2,l+1}>\bq'_{2,l}>\bq_{2,l}
>\bs_2>
\bp_{2,m}>\bp'_{2,m}>\bp_{2,m+1}>\bo_{2,n+1}>\bo'_{2,n}>\bo_{2,n}
\]
for all $k,l,m,n\in\ZZ_{\geq 0}$, and finally
$\St_{(3,+1)}(\uQ)=\{\br_{3,k}, \br'_{3,k}\mid k\in\ZZ_{\geq 0}\}
\cup\{\bq'_{3,l}, \bq_{3,l}\mid l\in\ZZ_{\geq 0}\}$ with
\[
\br_{3,k}>\br'_{3,k}>\br_{3,k+1}>\bq_{3,l+1}>\bq'_{3,l}>\bq_{3,l}
\]
for all $k,l\in\ZZ_{\geq 0}$.  The set of primitive bands
in normal form has precisely two elements, which are moreover
equivalent, namely
\[
\pBa'(\uQ)=\{\eps_1^*\gam\eps_3^*\gam^{-1},
\eps_3^*\gam^{-1}\eps_1^*\gam\}.
\]

\subsubsection{The operator $A(\bw)$ and admissible strings}
\label{sssec:ExAdm}
We illustrate the construction $A(\bw)$, from Section~\ref{ssec:huQA}, in our situation: 
\begin{align*}
\hbr_{2,k}:= A(\br_{2,k}) &= \II_{(2,+)}\bet^{-1}
(\eps_3^{-1}\gam^{-1}\eps_1^{-1}\gam)^k\eps_3^{-1}\gam^{-1}\eps_1^{-1}\alp^{-1}\II_{(2,+)},\\
\hbr'_{2,k}:= A(\br'_{2,k}) &= \II_{(2,+)}^{-1}\bet^{-1}(\eps_3^{-1}\gam^{-1}\eps_1^{-1}\gam)^k
\eps_3^{-1}\gam^{-1}\eps_1^{-1}\II_{(1,+)},\\
\hbq'_{2,l}:= A(\bq'_{2,l}) &= \begin{cases}
 \II_{2,+}\bet^{-1}(\eps_3^{-1}\gam^{-1}\eps_1^{-1}\gam)^{l'}\eps_3^{-1}
 (\gam^{-1}\eps_1\gam\eps_3)^{l'}\II_{(3,+)} &\text{if } l=2l',\\
\II_{2,+}\bet^{-1}(\eps_3^{-1}\gam^{-1}\eps_1^{-1}\gam)^{l'+1}
\eps_3(\gam^{-1}\eps_1\gam\eps_3)^{l'}\II_{(3,+)} &\text{if } l=2l'+1,\\
\end{cases}\\
\hbq_{2,l}:=A(\bq_{2,l})  &= \begin{cases}
 \II_{(2,+)}^{-1}\bet^{-1}(\eps_3^{-1}\gam^{-1}\eps_1^{-1}\gam)^{l'}\II_{(3,-)}
   & \text{if } l=2l',\\
 \II_{(2,l)}^{-1}\bet^{-1}(\eps_3^{-1}\gam^{-1}\eps_1^{-1}\gam)^{l'}
 \eps_3^{-1}\gam^{-1}\II_{(1,-)} & \text{if } l=2l'+1,\end{cases}\\
\hbp_{2,m}:=A(\bp_{2,m}) &= \begin{cases}
 \II_{(2,+)}^{-1}\alp (\eps_1\gam\eps_3\gam^{-1})^{m'}\II_{1,-}
 &\text{if } m=2m',\\
  \II_{(2,+)}^{-1}\alp (\eps_1\gam\eps_3\gam^{-1})^{m'}
  \eps_1\gam\II_{3,-} &\text{if } m=2m'+1,
\end{cases}\\
\hbp'_{2,m}:= A(\bp'_{2,m}) &= \begin{cases}    
\II_{(2,+)}^{-1}\alp(\eps_1\gam\eps_3\gam^{-1})^{m'}\eps_1
(\gam^{-1}\eps_3^{-1}\gam\eps_1^{-1})^{m'}\II_{1,+}& \text{if } m=2m',\\
\II_{(2,+)}^{-1}\alp(\eps_1\gam\eps_3\gam^{-1})^{m'+1}\eps_1^{-1}
(\gam^{-1}\eps_3^{-1}\gam\eps_1^{-1})^{m'}\II_{(1,+)}& \text{if }m=2m'+1,
\end{cases}\\
\hbr_{3,k}:= A(\br_{3,k}) &=\II_{(3,+)}^{-1}(\eps_3^{-1}\gam^{-1}\eps_1^{-1}\gam)^k
\eps_3^{-1}\gam^{-1}\eps_1^{-1}\alp^{-1}\II_{(2,+)},\\
\hbr'_{3,k}:= A(\br'_{3,k}) &=\II_{(3,+)}^{-1}(\eps_3^{-1}\gam^{-1}\eps_1^{-1}\gam)^k
\eps_3^{-1}\gam^{-1}\eps_1^{-1}\II_{(1,+)},\\
A(\eps_1^*\gam\eps_3\gam^{-1}) &=\II_{(1,-)}^{-1}\gam\II_{(3,-)}.
\end{align*}
Note, that in our situation we have 
$\Adm_b(\uQ)=\{\II_{(1,-)}^{-1}\gam\II_{(3,-)}, 
\II_{(3,-)}^{-1}\gam^{-1}\II_{(1,-)} \}$.

\subsubsection{The operator $\tau_f$}  \label{sssec:ExTauf}
It is  here easy to describe explicitly the 
operator $\tau_f$  from section~\ref{ssec:AR-adm}. In fact, for
each admissible string $\bx\in\AdmSt(\uQ)$, to this end we have just to 
substitute in $\bx$ each trivial letters of the form 
$\II_{(i,+)}^{\pm 1}$ with $i= 1, 2, 3$ and $\II_{(2,-)}^{\pm 1}$
by $(\alp^{-1}\bet'\II_{(5,-)})^{-1}$, 
$((\alp')^{-1}\II_{(4,-)})^{\pm 1}$,
$(\gam^{-1}\eps_1\gam\eps_3\bet\bet'\II_{(5,-)})^{\pm 1}$, 
and $(\bet^{-1}\eps_3\bet\bet'\II_{(5,-)})^{\pm 1}$ respectively.

\subsubsection{Windings} \label{sssec:ExWindings}
In order to illustrate H- and K-triples, as well as the  combinatorial E-invariants and g-vectors below, 
we display several windings $G_\bx: H(\bx)\ra\uQ$ for  admissible 
$\bx$ from our running example. In each case we agree that
$G_\bx(i_j)=i$ for each vertex $i_j$, and
the label of each arrow indicates the value of $G_\bx$ on this arrow, rather than the name of this arrow. For typographical reasons we omit in most cases the (obvious) values of the polarization on the arrows.
\begin{align*}
\uH(\hbq_{2,2})&=\xymatrix{
2_1\ar[r]^<<-^{\bet}^>>+ & 3_1\ar[r]^<<-^{\eps_3}^>>- &3_2\ar[r]^<<+^{\gam}^>>+& 1_1
\ar[r]^<<-^{\eps_1}^>>-& 1_2 &\ar[l]_>>+_{\gam}_<<+
3_3\ar@(ur,dr)[]^<-^{\eps_3}^>-}\\
\uH(\hbq'_{2,1}) &=\xymatrix{
2_1\ar[r]^{\bet} & 3_1\ar[r]^{\eps_3} &3_2\ar[r]^{\gam}& 1_1\ar[r]^{\eps_1} & 1_2&\ar[l]_{\gam} 3_3&\ar[l]_{\eps_3}3_4}\\
\uH(\hbq_{2,1})&=\xymatrix{
2_1\ar[r]^{\bet} & 3_1\ar[r]^{\eps_3} &3_2\ar[r]^{\gam}& 1_1\ar@(ur,dr)[]^{\eps_1}}\\
\uH(\hbp_{2,1})&=\xymatrix{
2_1&\ar[l]_{\alp} 1_1&\ar[l]_{\eps_1} 1_2&\ar[l]_{\gam} 3_1\ar@(ur,dr)[]^{\eps_3}}\\
\uH(\hbp'_{2,1}) &=\xymatrix{
2_1&\ar[l]_{\alp} 1_1&\ar[l]_{\eps_1} 1_2&\ar[l]_{\gam} 3_1&\ar[l]_{\eps_3} 3_2\ar[r]^{\gam} &1_3\ar[r]^{\eps_1}&1_4}\\
\uH(\tau_f(\hbr'_{3,0})) &=\vcenter{\xymatrix{
5_1\ar[r]^{\bet'}& 2_1\ar[r]^{\bet}& 3_1\ar[d]^{\eps_1} &1_2 &\ar[l]_{\gam} 3_3\ar[d]^{\eps_3}& 1_4\ar[r]^{\alp}& 2_2 &\ar[l]_{\bet'} 4_1  \\
&& 3_2\ar[r]_{\gam}& 1_1\ar[u]^{\eps_1} &3_4\ar[r]_{\gam}& 1_3\ar[u]^{\eps_1}&
}}\\
\uH(\II_{1,-}^{-1}\gam\II_{3,-}) &=\xymatrix{
\ar@(ul,dl)[]_{\eps_1} 1_1&\ar[l]_{\gam} 2_1\ar@(ur,dr)[]^{\eps_3}}
\end{align*}

\subsubsection{Examples of H- and K-triples} \label{sssec:Ex-H-K-trip}
(1) Consider the trivial (polarized) quiver $\uH^{(1)}= 2$, together with the morphisms $\phi_q^{(1)}\df \uH^{(1)}\ra\uH(\hbq_{2,2}), 2\mapsto 2_1$
and $\phi_s^{(1)}\df\uH^{(1)}\ra\uH(\hbp_{2,1}), 2\mapsto 2_1$.  Then
$(\uH^{(1)},\phi_q^{(1)},\phi_s^{(1)})$ is an H-triple of type $\sfA$, but it is  \emph{not} a K-triple, because the boundary vertex $2$ is sent, for example, by $\phi_q^{(1)}$ to the boundary vertex $2_1$ of $H(\hbq_{2,2})$.   
    
(2) Consider the quiver 
\[
\uH^{(2)}=\ \xymatrix{
2\ar[r]^{\bet} &3\ar[r]^{\eps_1} & 3'\ar[r]^{\gam} & 1
},\]
together with the canonical inclusions 
$\phi_q^{(2)}\df \uH^{(2)}\ra\uH(\hbq'_{2,1})$ and 
$\phi_s^{(2)}\df\uH^{(2)}\ra\uH(\hbq_{2.1})$. Then, 
$(\uH^{(2)},\phi_q^{(2)}, \phi_s^{(2)}) \in\cH_\uQ(\hbq'_{2,1},\hbq_{2,1})$. 
This H-triple of type $\sfA$ is \emph{not} a K-triple, because the boundary vertex $2$ is sent, for example, by $\phi_q^{(2)}$ to a boundary vertex.  

(3) Consider the quiver 
\[
\uH^{(3)}=\ \xymatrix{1&\ar[l]_{\gam} 2 \ar@(ur,dr)[]^{\eps_3}},
\]
together with the morphism $\phi_q^{(3)}\df\uH^{(3)}\ra\uH(\hbq_{2,2})$ defined
by the assignation  $1\mapsto 1_2$, $3\mapsto 3_3$, and the morphism
$\phi_s^{(3)}\df \uH^{(3)}\ra\uH(\hbp_{2,1})$ defined by the assignation
$1\mapsto 1_2$, $3,\mapsto 3_1$.  Then, 
$[(\uH^{(3)}, \phi_q^{(3)}, \phi_s^{(3)})]\in\cK_\uQ(\hbp_{2,1},\hbq_{2,2})$ is a K-triple of type $\sfD'$.
Here, for the unique polarized boundary vertex  $(1,-1)$ of $\uH^{(3)}$, we are in case (II). 

(4) Consider the quiver 
\[
\uH^{(4)}=\xymatrix{1&\ar[l]_{\eps_1} 1'&\ar[l]_\gam 3} 
\]
together  with
the morphism $\phi_q^{(4)}\df\uH^{(4)}\ra \uH(\hbp_{2,1})$ defined by prescribing $3\mapsto 3_1$,  
and the morphism $\phi_s^{(4)}\df\uH^{(4)}\ra\uH(\hbq_{2,2})$ defined by prescribing
$3\mapsto 3_2$. Then,
$[(\uH^{(4)}, \phi_q^{(4)}, \phi_s^{(4)})]\in\cK_\uQ(\hbp_{2,1},\hbq_{2,2})$
is of type $\sfA$. Note, that for the polarized boundary vertex $(3, -1)$, we have here the case $(IV)$. 

(5) Consider the quiver 
\[
\uH^{(5)} = \xymatrix{3\ar[r]^{\eps_3} &3'\ar[r]^{\gam} &1\ar[r]^{\eps_1}& 1'},
\]
together with the morphism 
$\phi_q^{(5)}\df\uH^{(5)}\ra\uH(\tau_f(\hbr'_{3,0}))$
which is uniquely defined by  prescribing $3\mapsto 3_3$,
and the morphism
$\phi_s^{(5)}\df\uH^{(5)}\ra\uH(\tau_f(\hbr'_{3,0}))$ which is uniquely defined by prescribing $3\mapsto 3_1$.  Then 
$(\uH^{(5)}, \phi_q^{(5)}, \phi_s^{(5)})\in\cK_\uQ(\tau_f(\hbr'_{3,0}),\tau_f(\hbr'_{3,0}))$.  For both polarized boundary vertices, $(3,-1)$ and
$(1,-1)$, correspond to case (I). 

\subsubsection{Combinatorial g-vector} \label{sssec:ExComb-gvect}
Observe, that 
$\tau_f(\hbq_{2,1})=\II_{(4,-^)}{-1}
\alp'\bet^{-1}\eps_3^{-1}\gam^{-1}\II_{(1,-)}$.  
Thus, 
\[
\uH(\tau_f(\hbq_{2,1}))=\xymatrix{4_1&\ar[l]_{\alp'}
2_1\ar[r]^{\bet} & 3_1\ar[r]^{\eps_3} &3_2\ar[r]^{\gam}& 1_1\ar@(ur,dr)[]^{\eps_1}}
\]
It is then straightforward to check that $A_2^+(\hbq_{2,1})=\{2_1\}$ and
$D_1^-(\hbq_{2,1})=\{\eta_1\}$, where we recall that, by our conventions,
the unique (special) loop of $\uH(\tau_f(\hbq_{2,1}))$ is named $\eta_1$
and $G_{\tau_f(\hbq_{2,1})}(\eta_1)=\eps_1$.  All other sets of the  form
$A_i^\pm(\hbq_{2,1})$, or $D_i^\pm(\hbq_{2,1})$, are empty.  
We also recall, that in our situation $\tQ_0=\{(1,+), (1,-), (2,o), (3,+), (3,-)\}$,  and that the possible decorations of $\hbq_{2,1}$ are
$s^+$ and $s^-$ with $s^\pm_1=1$ and $s^\pm_2=\pm 1$. In particular, 
$\wt(s^\pm)=1$. With this at hand, we see from Definition~\ref{def:comb-inv}, that
$\bg_\uQ(\hbq_{2,1},s^\pm)_{(2,o)}=-1$  and
$\bg_\uQ(\hbq_{2,1}, s^{\rho'})_{1,\rho}=d_2(\rho 1, -s_2^{\rho'})=\del_{\rho,-\rho'}$ for $\rho,\rho'\in\{+,-\}$.  
Moreover, $\bg_\uQ(\hbq_{2,1},s^\pm)_{3,\rho}=0$ for $\rho\in\{+,-\}$.    
Similarly, we find $\bg_\uQ(\hbq_{2,2}, s^{\rho'})_{(1,\rho)}=1$,
$\bg_\uQ(\hbq_{2,2}, s^{\rho'})_{(2,o)}=-1$ and 
$\bg_\uQ(\hbq_{2,2}, s^{\rho'})_{(3,\rho)}=-\del_{-\rho',\rho}$ for
$\rho, \rho'\in\{+,-\}$.

\subsubsection{Combinatorial E-invariant}\label{sssec:ExCombE-inv} 
(1) In view of Section~\ref{sssec:ExTauf} we have
\begin{align*}
\tau_f(\hbq_{2,2})&=
\II_{(4,-)}\alp'\bet^{-1}\eps_3^{-1}\gam^{-1}\eps_1^{-1}\gam\II_{(3,-)} 
\quad\text{and}\\ 
\tau_f(\hbp_{2,1})&=\II_{(4,-)}\alp'\alp\eps_1\gam\II_{(3,-)}.  
\end{align*}
Thus, the quivers $\uH(\tau_f(\hbq_{2,2}))$ and 
$\uH(\tau_f(\bp_{2,1}))$ can be obtained from the quivers
$\uH(\tau_f(\hbq_{2,2}))$ and $\uH(\tau_f(\bp_{2,1}))$,
described in Section~\ref{sssec:ExWindings}, by adding in
both cases an arrow $\xymatrix{4&\ar[l]_{\alp'} 2}$ on the left-hand side. 
With this, it is easy to see that 
\[
\sfA_{\uQ^f}(\hbq_{2,2},\hbp_{2,1})=\{[(\uH^{(4)},\phi_q^{(4')},\phi_s^{(4')})]\},
\]
where $\phi_q^{(4')}$ and $\phi_s^{(4')}$
are similarly defined as $\phi_q^{(4)}$ and $\phi_s^{(4)}$ in
Section~\ref{sssec:Ex-H-K-trip}~(4). 
$\sfP_\uQ((\hbq_{2,2},\hbp_{2,1})=\{(1,1)\}$.  This element 
is closely related to the kiss of type $\sfD'$, which was described in Section~\ref{sssec:Ex-H-K-trip}~(3). 
For both strings, $\tau_f(\hbp_{2,1})$ and $\tau_f(\hbq_{2,2})$,
the possible decorations are again $s^\pm$ with $s_1^{\pm}=1$ and $s_2^\pm =\pm 1$. 
Thus, in view of Definition~\ref{def:comb-inv} we have
\[
e_\uQ((\hbq_{2,2}, s^\rho), (\hbp_{2,1}, s^{\rho'})) = 1+d_2(s_2^\rho, -s_2^{\rho'})
=1+\del_{\rho, -\rho'}.
\]
(2) Recall from Section~\ref{sssec:Ex-H-K-trip}~(5) that
\[
\cK_\uQ(\tau_f(\hbr'_{3,0}), \tau_f(\hbr_{3,0}))=
\{[(\uH^{(5)},\phi_q^{(5)}, \phi_s^{(5)})]\} 
\]
consists of a unique K-triple, which is of type $\sfA$.
Thus, by definition, 
\[
\sfA_{\uQ^f}(\hbr'_{3,0}, \hbr'_{3,0})=
\{[(\uH^{(5)},\phi_q^{(5)}, \phi_s^{(5)})], 
[(\uH^{(5)},(\phi'_q)^{(5)}, (\phi'_s)^{(5)})] \}, 
\]
where in $[(\uH^{(5)},(\phi'_q)^{(5)}, (\phi'_s)^{(5)})]$ the
role of the first and second copy of $\tau_f(\hbr'_{3,0})$ are swapped. 
Since $\hbr'_{3,0}$ is of type $(u,u)$,  the only
possible decoration is $s^+=(1,+1)$ and trivially 
$\sfP_\uQ(\hbr'_{3,0}, \hbr'_{3,0})=\emptyset$. 
We conclude, that 
\[
e_\uQ((\hbr'_{3,0}, s^+), (\hbr'_{3,0}, s^+))=
\abs{\sfA_{\uQ^f}(\hbr'_{3,0}, \hbr'_{3,0})}=2.
\]
(3) Finally, note that  $\br:=\II_{1,-1}^{-1}\gam\II_{3,-}$ is of type $(p,p)$,
and thus we have $\tau_f(\br)=\br$.  We find easily
\[
\sfA_{\uQ^f}(\hbp'_{2,1},\br)=\{(\uH^{(4)}, \phi^{(4')}_q, \phi^{(4')}_s)\},
\]
where $\phi^{(4')}\df\uH^{(4)}\ra \tau_f(\hbp'_{2,1})$ is defined by the
assignation $3\mapsto 3_2$, and $\phi^{(4')}_s\df\uH^{(4)}\ra\uH(\br)$ 
is defined by the assignation $3\mapsto 3$. Trivially, 
$\sfP_{\uQ}(\hbp'_{2,1}, \br)=\emptyset$. Thus,  
\[
e_\uQ((\hbp'_{2,1}, s^+), (\br, t))=\wt(t)\in\{1, 2\}.
\]

\section{Curves on punctured surfaces} \label{sec:punct}
\subsection{Basic definitions}
Let $\bSu:=(\Su,\Pu,\Ma)$ be a surface with marked points and
non-empty boundary.
This means  that $\Su$ is a compact, connected oriented surface of genus $g$ with boundary $\dSu=S_1\cup S_2\cup\cdots\cup S_b$ consisting of $b\geq 1$ connected components.
Each $S_i$ is homeomorphic to the unit circle $S^1\subset\CC$,
and the \emph{induced orientation} of $S_i$ is such, that when following this orientation, the surface lies on the left. We assume that $S^1$ is anticlockwise oriented.
$\Ma\subset\dSu$ is a finite set of \emph{marked points} on the
boundary such that $\Ma\cap S_i\neq\emptyset$ for all $i$ and
$\Pu\subset\Su\setminus\dSu$ is a finite set, called \emph{punctures}.
By definition,
\begin{equation} \label{eq:rksu}
n(\bSu) := 6(g-1) +3(b+\abs{\Pu})+\abs{\Ma}
\end{equation}
is \emph{rank} of the marked surface. We will always assume $n(\bSu)\geq 1$ in
order to avoid pathological situations. An oriented \emph{\gls{curve}}  on $(\Su,\Pu,\Ma)$ is a 
piecewise differentiable, locally injective map
$\gam\df [0,1]\ra\Su$, such that (compare to \cite[Definition 5.1 and Theorem 5.2]{GLF23})
\[
  \gam(\{0,1\})\subset\Pu\cup\Ma\text{ and }
  \gam(]0,1[)\subset\Su\setminus (\dSu\cup\Pu).
\]
We say that $\gam$ for  $(\Su,\Pu,\Ma)$ is of \gls{type}
\begin{alignat*}{2}
  (u,u) &\text{ if } \{\gam(0),\gam(1)\}\subset\Ma,\quad &
  (u,p) &\text{ if } \gam(0)\in\Ma, \gam(1)\in\Pu,\\
  (p,p) &\text{ if } \{\gam(0),\gam(1)\}\subset\Pu,\quad &
  (p,u) &\text{ if } \gam(0)\in\Pu,\ \gam(1)\in\Ma.
\end{alignat*}

We consider such curves up to homotopy relative to $\{0,1\}$,  and we
write $\gam\gls{sim}\gam'$ if two curves are homotopic in that sense.
In this situation, we denote by   $\gls{gam-1}$ 
the curve with
\emph{opposite orientation},
that is $\gam^{-1}(t):=\gam(1-t)$ for all $t\in [0,1]$.
An oriented \emph{\gls{loop}} is a locally differentiable, locally injective
map (compare too to \cite[Definition 5.1 and Theorem 5.2]{GLF23})
\[
  \lam\df S^1\ra\Su\setminus (\dSu\cup\Pu).
\]
We consider loops usually  up to free homotopy in $\Su\setminus\Pu$, and we
write in this situation also $\gam\sim\gam'$ if $\gam$ and $\gam'$ are
homotopic in that sense. In this case, the loop $\gam^{-1}$ with
the opposite orientation is given by $\gam^{-1}(z):=\gam(\overline{z})$ for all
$z\in S^1\subset\CC$ and $\overline{z}$ is the complex conjugate of $z$.
We say that a loop $\lam'$ is \emph{\gls{primitive}} if it is not of the form $\lam'~\sim\lam^m$ for some loop $\lam$ and $m\geq 2$.  By definition, the \emph{\gls{type}} of each loop $\lam$
is $(b)$.

Recall, that for $E\subset\Su\setminus\Pu$ the classical
\gls{fundamentalgroupoid} $\gls{fgpoid}$ 
has the elements of $E$ as objects, and morphisms between $e_1$ and  $e_2$ are the homotopy classes of oriented
paths in $\Su\setminus\Pu$.
In this situation, Amiot and Plamondon~\cite[Sec.~5.1]{AP17} discuss the
\gls{fundorbgroupoid}  $\glslink{forgpoid}{\piorb(\Su,E)}$.
By definition, this is the quotient of $\pi_1(\Su\setminus\Pu, E)$ by the
equivalence relation on paths, which is generated by the relations
described in Figure~\ref{fig:orb}, see also \cite[\S2.1]{CG16}. We denote by
$\glslink{forgpoid}{\piorb(\Su, E)^*}$ the set
of \emph{all} morphisms in $\piorb(\Su, E)$ which are not identities.  
\begin{figure}[ht]
\begin{tikzpicture}
\draw[blue,thick](0,0)
  .. controls +(0:1) and +(0:2) .. node[near start, sloped]{$>$} node[near end, sloped]{$>$} (2,2)
  .. controls +(180:2) and +(180:1).. node[near start, sloped]{$>$}(4,0);
 \draw (2.1,1.5) node{$\times\ p$};
 \draw (4.5,1.2) node[scale=1.6]{$\simeq$}; 
 \draw[blue, thick] (5,0)
  .. controls +(0:1.5) and +(-60:1.5) .. node[near start, sloped]{$>$}  ++(1.5,1.5)
  .. controls +(120:1) and +(60:1)   .. node[very near end, sloped]{$>$} node[very near start, sloped]{$>$} ++(1.2,0)
  .. controls +(-120:1.5) and +(180:1.5) .. ++(1.5,-1.5);
 \draw (5,0) + (2.2,1.5) node{$\times\ p$};   
\end{tikzpicture}
\caption{Orbifold homotopy relations}
\label{fig:orb} 
\end{figure}
Following~\cite[Sec.~5.3]{AP17}, the \gls{freeorbifoldfundamentalgroup}
$\glslink{forgpoid}{\piorbfree(\Su)}$ of $\piorb(\Su,E)$ consists
of a set, which by some abuse of notation, is also named  
$\piorbfree(\Su)$, together with natural bijections
$c_e\df\piorbfree(\Su)\ra\piorb(\Su,E)(e,e)_{cl}$ for each $e\in E$.  Here
$\piorb(\Su,E)(e,e)_{cl}$ denotes the set of conjugacy classes of the orbifold fundamental group $\piorb(\Su,E)(e,e)$ with base point $e$.  This is well-defined because we assume $\Sig$ to be connected, and thus the groupoid
$\piorb(\Su,E)$ is also connected. Note that the underlying set of $\piorbfree(\Sig)$ does not depend on the choice of $E$. 
Despite its name, $\piorbfree(\Sig)$ is \emph{not} a group.
The elements of $\piorbfree(\Su)$ can
be identified with the free homotopy classes of loops in $\Su\setminus\Pu$
modulo the relations from Figure~\ref{fig:orb}. This follows for example
from~\cite[Lemma~2.2]{Am16}.   We will discuss in Section~\ref{ssec:orb} a more
concrete description of those objects in terms of signature zero tagged
triangulations.  

\subsection{Admissible curves} \label{ssec:ALC}
Following~\cite[Sec.~3.2]{QZ17}, we define the \emph{\gls{completion}}
$\gls{bargam}$
of a curve, which is  not of type (u, u),
as in Figure~\ref{fig:cpl}. In particular, the completion of a curve, which
connects two (not necessarily different) punctures, is a loop. Notice,
that we orient $\widebar{\gam}$ always as in Figure~\ref{fig:cpl}, regardless
of the orientation of $\gam$.
We agree that  the remaining curves,  and all loops are complete, i.e. we have
in those cases $\widebar{\gam}=\gam$.
\begin{figure}[ht]
\begin{tikzpicture}[scale=2.5]
\fill[lightgray] (.5,0) --  (1,0) arc[start angle = 20, end angle = 90, radius=1.5] --  (-.3, 0.5) -- cycle;
\draw [name path = curve1, thick, <-]
(1,0) arc[start angle = 20, end angle = 90, radius=1.5];
\path[name path = vline1] (0.5, 0) -- (0.5, 1.5);
\fill [black, name intersections={of= curve1 and vline1, by=mk-1}]
   (mk-1) circle (1pt); 
\draw[gray] (mk-1) node[right,black]{$\scriptstyle{m}$} -- +(45:0.7) node[black, name = pt1] {$\times$} node[right, black]{\footnotesize $p$};
\draw[gray] (mk-1) ++(45:0.45) node {$\scriptstyle{\ \gamma}$};
\draw[blue,->] (mk-1) -- +(55:0.7);
\draw[blue] (mk-1)  ++(55:0.7) arc [start angle=145, end angle=-34, radius=.13] node[below]{$\bar{\gam}$} --(mk-1);
\end{tikzpicture}
\qquad
\begin{tikzpicture}[scale=2, baseline=-0.2]
  \draw[gray] (-1,1) node[black]{$\times$} node[black, below right]{\footnotesize $p_1$} -- +(2,0)  node[black] {$\times$} node[black, below left]{\footnotesize $p_2$};
  \draw[blue, ->] (-1, 1.25) -- ++(2,0) node[above left] {\footnotesize $\bar{\gamma}$};
  \draw[blue] (1, 1.25) arc [start angle=90, end angle=-90, radius=.25] -- ++(-2,0) arc [start angle= -90, end angle=-270, radius=0.25]; 
\draw[gray] (-1,1) ++(1,.1) node {\footnotesize $\gamma$};
\end{tikzpicture}
\caption{Completion of curves}
\label{fig:cpl} 
\end{figure}
We say, that a curve, or a loop,  has a \emph{proper resp.~based  kink},
if it encloses a
once-punctured monogon by a self-intersection (including endpoints),
see Figure~\ref{fig:kink}. For details on the treatment of kinks and their invariance under homotopy, we refer the reader to \cite{GLF23}.
\begin{figure}[ht]
\begin{tikzpicture}
  \draw[blue, ->]
  (-1,0) .. controls (0, 0.5) and (0.5, 1) .. (0.5, 1.5) node[right]{$\gamma$}
  arc [start angle=0, end angle=180, radius=0.5]
 .. controls (-0.5,1.2) and  (-0.5,1) .. (1,0);
\draw[black] (0,1) +(0, 0.5) node{$\times$} node[below]{$p$};
\draw (0,-.2) node{$\phantom{x}$};
\draw (0,-.5) node{proper};
\end{tikzpicture}
\qquad
\begin{tikzpicture}[scale=2]
\fill[lightgray]
(.5,0) --  (1,0) arc[start angle = 20, end angle = 90, radius=1.5] --  (-.3, 0.5) -- cycle;
\draw [name path = curve1, thick, <-]
(1,0) arc[start angle = 20, end angle = 90, radius=1.5];
\path [name path = vline1] (0.5, 0) -- (0.5, 1.5);
\fill [black, name intersections={of= curve1 and vline1, by=mk-1}]
   (mk-1) circle (1pt); 
\draw[gray] (mk-1) node[right,black]{$\scriptstyle{m}$}  ++(45:0.7) node[black, name = pt1] {$\times$} node[below left, black]{\footnotesize $p$};
\draw[blue,->] (mk-1) -- +(55:0.7);
\draw[blue] (mk-1)  ++(55:0.7) arc [start angle=145, end angle=-34, radius=.13] node[below]{$\gam$} --(mk-1);
\draw (0.5,-0.2) node{based};
\end{tikzpicture}   
\qquad
\begin{tikzpicture}[scale=2.5]
\coordinate (mk-1) at (0.2,0.2);
\draw (mk-1) node[black] {$\times$}; 
\draw[gray] (mk-1) node[below right,black]{$\scriptstyle{p'}$}  ++(45:0.7) node[black, name = pt1] {$\times$} node[below left, black]{\footnotesize $p$};
\draw[blue,->] (mk-1) -- +(55:0.7);
\draw[blue] (mk-1)  ++(55:0.7) arc [start angle=145, end angle=-34, radius=.13] node[below]{$\gam$} --(mk-1);
\draw (0.5,-0.2) node{based};
\end{tikzpicture}

\caption{Kinks}
        \label{fig:kink} 
     \end{figure}

\begin{Def} \label{def:tcLC}
We denote by $\glslink{tcALC}{\tcAC(\bSu)}$ the set of homotopy classes of curves $\gam$for $\bSu$  which fulfil the following conditions:
\begin{itemize}
\item $\gam$ has no proper kinks,
\item $\gam$ is not null homotopic,
\item $\gam$ is not homotopic to a boundary segment.
\end{itemize}       
Similarly, we denote by $\glslink{tcALC}{\tcAL(\bSu)}$ the set of free homotopy classes of
 non-trivial, primitive loops without kinks, which do not cut out a single 
 puncture.  Moreover, we set
\begin{align*}
\glslink{tcALC}{\tcLC(\bSu)}   &:=\tcAC(\bSu)\cup\tcAL(\bSu),\\
\glslink{tcALC}{\tcAC(\bSu)_b} &:=\{\gam\in\tcAC(\bSu)\mid \type(\gam)=(p,p)\},\quad\text{and}\\
\glslink{tcALC}{\tcLC(\bSu)_s} &:=\tcAC(\bSu)\setminus\tcAC_b(\bSu).
\end{align*}
\end{Def}

However, we also need  a slightly smaller class of curves, which we can relate to admissible strings and bands.
\begin{Def}
$\glslink{cALC}{\cAC(\bSu)}$ is the set of homotopy classes of curves $\gam$ for $(\Su,\Pu,\Ma)$  which fulfil the following conditions:
\begin{itemize}
\item $\gam$ has no kinks,
\item $\gam$ is not null homotopic,
\item $\gam$ is not homotopic to a boundary segment.
\item If $\type(\gam)=(p,p)$,  the completion $\widebar{\gam}$ is primitive \emph{up to orbifold} homotopy.  
\end{itemize}   
Next, we set 
\begin{align*}
\glslink{tcAL}{\tcAL(\bSu)_\sym} &:=
      \{\overline{\gam}^{\pm 1}\mid\gam\in\tcAC_b(\bSu)\}/\sim\\
\glslink{tcAL}{\cAL(\bSu)}  &:=\tcAL(\bSu)\setminus \tcAL(\bSu)_\sym,\\
\glslink{cALC}{\cLC(\bSu)}  &:=\cAL(\bSu)\cup\cAC(\bSu),\\
\glslink{cALC}{\cAC_b(\bSu)}  &:=\{ \gam\in\cAC(\bSu)\mid\type(\gam)=(p,p)\},\\
\glslink{cALC}{\cLC_s(\bSu)}  &:=\cAC(\bSu)\setminus\cAC_b(\bSu),\quad\text{and}\\
 \glslink{cALC}{\cLC_b(\bSu)}  &:=\cAC_b(\bSu)\cup\cAL(\bSu).   
 \end{align*}                    
\end{Def}

This parallels the situation for admissible words:
\\[.5cm]
\hspace*{2cm}\begin{tikzpicture} 
\node at (2.2,0)  {$\cAC(\bSu)$} ;
  \matrix at (4,-.3)
  {\node(l1){$(u,u)$};\\ \node{$(u,p)$}; \\ \node(l3){$(p,u)$};\\
    \node(l4){$(p,p)$};\\ \node(l5){$\phantom{(x,x)}$};\\};
\draw[decorate,decoration=brace] (l1.east) -- (l3.east);
\draw[decorate,decoration=brace] (l4.north east) -- (l5.south east);
\draw[decorate,decoration=brace] (l4.south west) -- (l1.north west);

%
\node at (2.2,-1.8) {$\cAL(\bSu)$} ;
\path (5.8,.3) node[outer xsep=1](As) {$\cLC_s(\bSu)$} -- +(3, 0) node(St){$\tcAC_s(\bSu)$} --  ++(0,-1.6)  node[outer xsep=1](Ab) {$\cLC_b(\bSu)$} -- ++(3,0) node(pB){$\tcAL(\bSu)$};
\draw[thick, ->] (As) -- node[above]{$\overline{?}$} (St);
\draw[thick, ->] (Ab) -- node[above]{$\overline{?}$} (pB);
\draw[rounded corners=4pt] (1.2,1.4) rectangle (6.6,-2.1);
\node at (3.9,-2.6) {$\cLC(\bSu)$};
\end{tikzpicture}

\begin{Rem}
(1) Note, that $\tcAL(\bSu)$ and $\cAL(\bSu)$ do not depend on the set $\Ma$.

(2) The discussion in  \cite[Rem.~5.23~(2)]{AP17} illustrates, why \emph{orbifold} homotopy is needed in the definition for curves of type $(p,p)$ in $\cAC(\bSu)$. 
Apparently, this point was missed  in~\cite[Def.~3.1~(T2)]{QZ17}.
\end{Rem}  

Finally, we denote by $\cAC(\Su,\Pu,\dSu\setminus\Ma)$ the set of homotopy
classes
of curves $\gam\df [0,1]\ra \Su$ without kinks, such that
 \[
   \gam(\{0,1\})\subset (\dSu\setminus\Ma) \cup\Pu \text{ and }
   \gam(]0,1[)\subset\Su\setminus (\dSu\cup\Pu),
 \]
and which do not cut out a, possibly once punctured, monogon, and which are not
 contractible.  Our homotopy relation allows in these cases to ``move'' the
endpoints, which lie on a boundary segment, within the given boundary segment.
See for example~\cite[Sec.~10.3]{GLFS22} for more details.

\subsection{Marked curves and their intersection numbers}
\label{ssec:taggedC}
Similar to section~\ref{ssec:combEg},  we define for each 
$\gam\in\cLC(\bSu)$ a subset $S(\gam)$ of the set of symbols
$S:= (\{-1,1\}\times\{-1,1\})\cup\{*, (*,*)\}$ as follows:
\[
{\gls{S(gam)}}:=\begin{cases}
\{(1,1)\}                        &\text{if } \type(\gam)=(u,u),\\
\{(1,-1), (1,1)\}                &\text{if } \type(\gam)=(u,p),\\
\{(-1,1), (1,1)\}                &\text{if } \type(\gam)=(p,u),\\
\{1,-1\}^2\cup\{(\ast,\ast)\}    &\text{if } \type(\gam)=(p,p),\\
\{\ast\}                &\text{if } \type(\gam)=(b).
\end{cases}
\]
We will use the same operations on $S$ as in Section~\ref{ssec:combEg}.
With this at hand, we define the set of \emph{\glslink{mtcurve}{marked curves}}
\begin{align*}
\glslink{mtALC}{\cLC^*(\bSu)}&:=\{(\gam,s)\mid \gam\in\cLC(\bSu), s\in S(\gam)\},\\
  \intertext{as well as the subset of \glslink{mtcurve}{tagged curves}}
\glslink{mtALC}{\cAC^{\bowtie}(\bSu)}&:=\{(\gam,s)\in\cLC^*(\bSu)\mid s\in\{-1,1\}\times\{-1,1\}\}
\end{align*}                
Note, that $(\gam,s)\in\cAC^{\bowtie}(\bSu)$ implies $\gam\in\cAC(\bSu)$. Thus, $\cAC^{\bowtie}(\bSu)$ can be identified with the set $C^\times(\bSu)$ of tagged curves from~\cite[Def.~3.1]{QZ17}.

 For $(\gam,s)\in\cLC^*(\bSu)$ the  \emph{inverse  orientation} is  $(\gam,s)^{-1}:=(\gam^{-1},s^\iota)$.  
Inverting the orientation defines an equivalence relation on
$\cLC^*(\bSu)$ which we denote by $\simeq$.

We visualize (equivalence classes of) marked curves  $(\gam,s)\in\cLC^*$ by drawing the ordinary curve $\gam\in\cAC(\bSu)$ together with a decoration close to each possible endpoint $\gam(t)\in\Pu$. 
The decoration close to $\gam(0)$ is ``plain˝ if $s_1=1$, it is a ``notch˝ 
$\bowtie$ if $s_1=-1$, and it is $"*"$ if $s_1=*$.
An analogous rule applies to the decoration close to
$\gam(1)$ and the ``value˝ of $s_2$.  Curves of type $(b)$ receive no
decoration at all. 

For $\gam_1\df I_1\ra\Su$ and $\gam_2\df I_2\ra\Su$ belonging to
$\cLC(\bSu)$ we define, generalizing slightly~\cite[Sec.~3.3]{QZ17},
the \emph{\gls{internalintersectionnumber}}
\[
  \intn (\gam_1,\gam_2):=\min\{\abs{\gam'_1\cap\gam'_2}\mid
  \gam'_1\sim \gam_1, \gam'_2\sim\gam_2\},
\]
where
\[
  \gam'_1\cap\gam'_2:=\{(t_1,t_2)\in I_1\times I_2\mid
    \gam'_1(t_1)=\gam'_2(t_2) \not\in (\Pu\cup\Ma)\}.
\]
As observed in~\cite[Rem.~3.5]{ZZZ13}, for a curve $\gam$  with one self-intersection, we have $\intn (\gam,\gam)=2$.

Recall that in Section~\ref{ssect: CombIn} we introduced the notions of \emph{pairs of punctured letters of type $D'$} and \emph{band orientation} in order to define the combinatorial $E$-invariant found in Section~\ref{ssec:combEg}.
We shall now introduce analogous definitions. Namely, for $\gam_1,\gam_2 \in \cA(\bSu)$ we set
\begin{align*}
\glslink{preInt}{\sfP_{\bSu}(\gam_1,\gam_2)}&:=\{(t_1,t_2)\in\{0,1\}\times\{0,1\}\mid
  \gam_1(t_1)=\gam_2(t_2)\in\Pu\}\setminus \sfP'_{\bSu}(\gam_1,\gam_2),
\quad\text{where}\\
  \sfP'_{\bSu}(\gam_1,\gam_2)&:=
  \begin{cases}
\{(0,0), (1,1)\}&\text{if }\gam_1\sim\gam_2^{\phantom{-1}}\in\cAC(\bSu),\\
\{(1,0), (0,1)\}&\text{if }\gam_1\sim\gam_2^{-1}\in\cAC(\bSu),\\
\qquad\emptyset &\text{else.}
\end{cases},\quad\text{and}\\
\glslink{preInt}{\Diag_b(\gam_1,\gam_2)}&:=\begin{cases}
\phantom{-}  1 &\text{if } \gam_1=\gam_2\quad\in\cLC_b(\bSu),\\
 -1 &\text{if } \gam_1=\gam_2^{-1}\in\cLC_b(\bSu),\\
\phantom{-} 0  &\text{else.}
\end{cases}
\end{align*}
Trivially, $\sfP_{\bSu}(\gam_1,\gam_2)\neq\emptyset$ is only possible if
$\gam_1$ and $\gam_2$ share at least one puncture.
\begin{Def} \label{def:mIntersectN}
With this at hand,  we define for each pair of elements
$(\gam,s), (\lam,t)\in\cAC^*(\bSu)$ the \emph{\gls{markedintersectionnumber}}
\begin{multline*}
\glslink{Intn*}{\Intn^*((\gam,s), (\lam,t))}:= \intn(\gam, \lam)\times \wt(s)\cdot\wt(t) \\
   + \sum_{(j,i)\in \sfP_\bSig(\gam,\lam)} \!\!\!d_2(s_{i+1},(t^\chi)_{j+1}) 
  \;+\; 2\abs{\Diag_b(\gam,\lam)} \cdot d_3(s',t^\chi)\quad\in\NN,
\end{multline*}  
where we use the same functions $d_2$ and $d_3$ as in Definition \ref{def:comb-inv}.
  This allows us to introduce the important
subset
\[
  \cLC^*_\tau(\bSu):=
  \{(\gam, s)\in\cLC^*(\bSig)\mid \Intn^*((\gam,s), (\gam,s))=0\}
\]  
of \emph{simple marked curves}. We let $\cLC^*_{\tau}(\bSu)/_{\simeq}$ denote this set of all simple marked curves considered up to the equivalence relation $\simeq$.
\end{Def}

On $\cAC^{\bowtie}(\bSu)$ the marked intersection number $\Intn^*$ is basically an expansion of the quite dense definition
of the tagged intersection number from~\cite[Def.~3.3]{QZ17}.  This expanded version will be convenient when we compare it later on with the combinatorial $E$-invariant $e_\uQ$. See Section~\ref{ssec:ExIntersecN} below for some sample calculations.

We now introduce the notion of a lamination, which coincides with the collection $C^{\circ}(S,M)$ considered by Musiker--Schiffler--Williams~\cite[Definition 8.2]{MSW13}

\begin{Def} \label{def:MSWL}
Following ~\cite{GLFS22}, we say a pair $L = (\gamma, m)$ is a \emph{\gls{lamination}} of $\bSu$ if $\gamma$ is a (finite) subset of $\cLC^*_\tau(\bSu)/{\simeq}$ such that $\Intn^*((\gamma_i,s_i),(\gamma_j, s_j)) = 0$ for all $(\gamma_i,s_i),(\gamma_j, s_j) \in \gamma$ and $m: \gamma \ra \mathbb{Z}_{>0}$ is a function. In other words, with the notation from
Example~\ref{expl:KRS1} we let 
\[
\glslink{LamSu}{\MSW(\bSu)}:=\KRS(\cLC^*_\tau(\bSu)/\simeq, \Intn^*)
\]
denote the set of all laminations of $\bSu$.
\end{Def}

Note that $\MSW(\bSu)$ may also be naturally identified with the set of \emph{integral, unbounded measured laminations} considered by Fomin--Thurston~\cite[Ch.~12]{FT18}. The reader may also be interested in the like-minded work of Fock--Goncharov~\cite[Section 7]{FG07}.

\begin{Rem} \label{rem:taggedA}
The set of \emph{\glslink{taggedarc}{tagged arcs}} is the subset
\[
\glslink{tArc}{\cAC^{\bowtie}_\tau(\bSu)}:=  \{(\gam, s)\in\cAC^{\bowtie}(\bSu)\mid
\Intn^*((\gam, s),(\gam,s))=0\}.
\]
Clearly, the tagged arcs are closed under the equivalence relation
$\simeq$ of inverting the orientation.  Thus, it makes sense
to consider the set of tagged arcs modulo their orientation, which we denote by $\glslink{tArc}{\cAC^{\bowtie}_\tau(\bSu)/_{\simeq}}$.

It follows from the above definitions, that for a tagged arc $(\gam, s)$
the underlying curve $\gam$ is in
fact an arc, since in particular the internal intersection number $\intn(\gam,\gam)=0$. 
Note moreover, that for a tagged arc with
$\gam(0)=\gam(1)\in\Pu$ we have $s_0=s_1\in\{-1,1\}$.  
In view  of Proposition~\ref{Lem:arcs}  below, we may  identify moreover
the set $\cAC^{\bowtie}_\tau(\bSu)/_{\simeq}$
with the set of tagged arcs from~\cite[Def.~7.1]{FST08}.
\end{Rem}

\subsection{Tagged triangulations  and skewed-gentle  polarized quivers}
\label{ssec:tagged3ang}
\begin{Def}\label{def: tagged T} A \emph{\gls{taggedtriangulation}} of $\bSu$ is a maximal collection of
pairwise different equivalence classes of tagged curves
$T=(\gam_i,c_i)_{1\leq i\leq n}$ in $\cA^{\bowtie}(\bSu)$ such that
\[
\Intn^*(\gam_i,c_i),(\gam_j, c_j))=0 \quad\text{for all}\quad 1\leq i,j\leq n.
\]
In particular, all elements of a tagged triangulation are tagged arcs. \end{Def}
It will be convenient for us to draw each arc $\gam_i$ with a specific
orientation. Recall also from~\cite[Prop.~2.10]{FST08} that in this case
$n=n(\bSu)$ is the rank of the marked surface, i.e.~the rank determines the
number of tagged arcs in each triangulation.

\begin{Def}\label{def: sig-zero}In view of~\cite[Def.~9.1]{FST08},
we say that a tagged triangulation $T$ is of \emph{\glslink{sign0}{signature zero}}
if for \emph{each} puncture $p\in\Pu$ there is a 
pair of tagged arcs $(\tau_p,\tau'_p)=((\gam_p,c_p),(\gam_p,c'_p))$ in $T$
with $\gam_p(1)=p$, $\gam_p(0)=m_p\in\Ma$ and $c_{p,2}=1=-c'_{p,2}$.  Note
that we have $c_{p,1}=1=c'_{p,1}$ since $\gam_p(0)\in\Ma$.  
We agree that in this situation the remaining (tagged) arcs of $T$ are labelled as $\tau_1, \tau_2,\ldots \tau_{n-2\abs{\Pu}}$. These arcs necessarily
have both endpoints in $\Ma$ and thus have all the trivial tagging
$(1,1)$. \end{Def} 

In~\cite[Def.~3.7]{QZ17}, signature zero tagged triangulations were called
\emph{admissible}.  Triangulations of this kind were also studied
in~\cite{GLFS16} and in~\cite{AP17}. Furthermore, the corresponding ideal triangulations played an important role in~\cite{GLF23} and may be found there under the name of \emph{signature zero triangulations}. Recall that the defining property of such (ideal) triangulations is that each puncture is enclosed in a self-folded triangle.

It is quite easy to see that $\dSu\neq\emptyset$ implies  the
existence of signature zero tagged triangulations, see for example~\cite[App.~A]{QZ17}.
For $T$ a signature zero tagged triangulation of 
$\bSu$, we cut the surface along the edges of $T$ which are of type $(u,u)$ and call the corresponding
connected components \emph{(generalized) triangles}.

Inspired by~\cite[Sec.~7]{FST08}
and~\cite[Sec.~4.1]{QZ17},  we associate to each signature zero tagged triangulation 
$T=(\tau_p,\tau'_p)_{p\in\Pu}\cup (\tau_i)_{i=1,\ldots n-2\abs{\Pu}}$  a
skewed-gentle polarized
quiver $\uQ=\glslink{Q(T)f}{\uQ(T)}$, together with a fringing $\uQ^f=\glslink{Q(T)f}{\uQ^f(T)}$, as follows:
For $i=1,2,\ldots n-2\abs{\Pu}$, the ordinary vertices $i$ of $\Qordv$
correspond to the arcs $\tau_i$ of type $(u,u)$ of $T$.
The special vertices are identified with the punctures $\Pu$.
The additional $\abs{\Ma}$ fringe vertices correspond to the boundary segments, which connect the marked points $\Ma$. 
From each triangle we obtain an oriented 3-cycle of ordinary arrows, by following the ``vertices˝ in clockwise direction.
Moreover, each arrow which starts or terminates on the left (resp.~right) side of an edge $\gamma$ with respect to its orientation, obtains at this end the polarization $+1$ (resp.~$-1$).  See Figure~\ref{fig:tagtria}.

\begin{figure}
\begin{tikzpicture} 
\clip (-2,2) circle (2.5cm);
  \draw[lightgray, line width=1.0cm, line cap=round] (0,0) +(-120:.5cm)
    arc[start angle=60, end angle=120, radius=3.5];
  \draw[name path = bottom1, thick] (0,0)
  arc[start angle=60, end angle=120, radius=4] node[near end, sloped]{$>$};
  \path [name path= bottom2]  (-2,0) -- (-2,3);
\fill [black, name intersections={of= bottom1 and bottom2, by=mb1}]
   (mb1) circle (2pt) node[below]{$m$}; 
\draw[blue, thick]
   (mb1) .. controls +(170:1.5) and +(180:4) .. node[sloped, near end]{$>$} (-2,4) node[blue, name=a1]{$\cdot$}
   .. controls +(0:4) and +(10:1.5) .. node[sloped]{$<$} (mb1);
   \draw[blue, thick, double distance=3pt]
   (mb1) -- node[very near end, sloped, below=-4.5pt]{$\bowtie$}  +(120:2) node[name=pt1, black, above left=-3pt]{$\times$} node[black, below left=1pt]{$p$};
\draw[blue, thick, double distance=3pt]
(mb1) -- node[very near end, sloped, below=-4.5pt]{$\bowtie$}  +(60:2) node[name=pt2, black, above right=-3pt]{$\times$} node[black, below right]{$q$};
\draw[teal, thick, ->] (pt2) -- node[near start, above]{$+$} node[near end, above]{$+$} (pt1);
\draw[teal, thick, ->] (pt1) -- node[very near start, above]{$+$} node[very near end, left=-1pt]{$-$}(a1);
\draw[teal, thick, ->] (a1) -- node[very near start, right=-1pt]{$-$} node[very near end, above]{$+$} (pt2);
\draw[teal, thick, dotted, ->]  (pt1) + (120:0.3)  arc[start angle=50, end angle=310, radius=0.3];
\draw[teal, thick, dotted, ->]  (pt2) + (60:0.3)  arc[start angle=-220, end angle=-490, radius=0.3];
\fill[teal] (a1) circle (1.5pt);
\fill (mb1) circle (2pt);
\end{tikzpicture}
\quad
\begin{tikzpicture} 
\clip (-2,2) circle (2.5cm);  
 \draw[lightgray, line width=1.0cm, line cap=round] (0,0) +(-120:.5cm)
    arc[start angle=60, end angle=120, radius=3.5];
  \draw[name path = bottom1, thick] (0,0)
  arc[start angle=60, end angle=120, radius=4] node[near end, sloped]{$>$};
  \path [name path= bottom2]  (-2,0) -- (-2,3);
\fill [black, name intersections={of= bottom1 and bottom2, by=mb1}]
   (mb1) circle (2pt) node[below]{$m_1$}; 
\draw[lightgray, line width=1.0cm, line cap=round] (-4,4) +(60:.5cm)
    arc[start angle=240, end angle=300, radius=3.5];
  \draw[name path = top1, thick] (-4,4)
  arc[start angle=240, end angle=300, radius=4] node[near end, sloped]{$<$};
  \path [name path= top2]  (-2,4) -- (-2,2);
\fill [black, name intersections={of= top1 and top2, by=mb2}]
(mb2) circle (2pt) node[above]{$m_2$};
\draw[blue, thick]
(mb1) .. controls +(170:2.8) and +(190:2.8) .. node[sloped, near start]{$<$} node[near end, name=a1]{$\cdot$} (mb2);
\draw[blue, thick]
(mb1) .. controls +(10:2.8) and +(-10:2.8) .. node[sloped, near start]{$>$} node[near end, name=a2]{$\cdot$} (mb2);
\draw[blue, thick, double distance=3pt]
(mb1) -- node[near end, sloped, below right=-4.3pt]{$\!\bowtie$}  +(90:1.2) node[name=pt2, black, above=-3pt]{$\times$} node[black, above=6pt]{$p$};
\fill[teal] (a1) circle (1.5pt);
\fill[teal] (a2) circle (1.5pt);
\fill[black] (mb1) circle (2pt);
\fill[black] (mb2) circle (2pt);
\draw[teal, thick, ->] (a1) -- node[near start, above]{$-$} node[near end, above]{$+$}(a2);
\draw[teal, thick, ->] (a2) -- node[very near start, below]{$+$} node[near end, below]{$+$}(pt2);
\draw[teal, thick, ->] (pt2) -- node[near start, below]{$+$} node[very near end, below]{$-$} (a1);
\draw[teal, thick, dotted, ->] (pt2) +(190:0.15)  arc[start angle=110, end angle=410, radius=.4];
\end{tikzpicture}
\begin{tikzpicture}
\clip (-1.2,.3) rectangle (3,-4.3);
\draw[lightgray, line width=1.0cm, line cap=round] (0,0) +(150:.5cm)
    arc[start angle=150, end angle=210, radius=4.5];
\draw[name path = left1, thick] (0,0)
  arc[start angle=150, end angle=210, radius=4] node[near end, sloped]{$>$} ;
\path [name path= left3]  (1,-0.2) -- (-2,-0.2);
\fill [black, name intersections={of= left1 and left3, by=mb1}]
(mb1) circle (2pt) node[left]{$m_1$};
\path [name path= left2]  (1,-1.0) -- (-2,-1.0);
\draw [black, name intersections={of= left1 and left2, by=lb1}]
(lb1) node[green, name=LB1]{$\cdot$}; 
\path [name path= left4]  (1,-3.7) -- (-3,-3.7);
\fill [black, name intersections={of= left1 and left4, by=mb2}]
(mb2) circle (2pt) node[left]{$m_2$};
\draw[blue, thick, double distance=3pt]
   (mb2) -- node[very near end, sloped, below=-4.6pt]{$\bowtie$}  +(50:1.8) node[name=pt1, black, above right=-3pt]{$\times$} node[black, above =6pt]{$\quad p$};
\draw[blue, thick]  (mb2) .. controls +(10:4) and +(-10:4) .. node[sloped]{$<$} node[near end, name=lb2]{$\cdot$} (mb1);
\draw[green, thick, ->] (LB1) -- node[very near start, above]{$+$} node[near end, above=-2pt]{$-$}  (lb2);
\draw[teal, thick, ->] (lb2) -- node[near start, below right=-8pt]{$\quad -$} node[very near end, right]{$+$} (pt1);
\draw[green, thick, ->] (pt1) --  node[very near start, below left=-2pt]{$+$} node[very near end, below]{$+$}(LB1);
\draw[teal, thick, dotted, ->] (pt1) + (250:0.2) arc[start angle=170, end angle=450, radius=0.3];
\fill (mb2) circle (2pt);
\fill[green] (lb1) circle (1.5pt);
\fill[teal] (lb2) circle (1.5pt);
\end{tikzpicture}
\\[1cm]
\begin{tikzpicture}
\clip (-2,1.7) circle (2.4cm);
\draw[lightgray, line width=1.0cm, line cap=round] (0,0) +(-120:.5cm)
    arc[start angle=60, end angle=120, radius=3.5];
\draw[name path = bottom1, thick] (0,0)
  arc[start angle=60, end angle=120, radius=4] node[near start, sloped]{$>$};
  \path [name path= bottom2]  (-2,0) -- (-2,3);
\fill [black, name intersections={of= bottom1 and bottom2, by=mb1}]
   (mb1) circle (2pt) node[below]{$m_1$};   

 \draw[lightgray, line width=1.0cm, line cap=round] (-1,4) +(0:.5cm)
    arc[start angle=180, end angle=240, radius=3.5];
  \draw[name path =right1, thick] (-1,4)
  arc[start angle=180, end angle=240, radius=4] node[near end, sloped]{$<$};
  \path [name path=right2]  (2,2.5) -- (-2,2.5);
\fill [black, name intersections={of= right1 and right2, by=mb3}]
   (mb3) circle (2pt) node[right]{$m_3$}; 

 \draw[lightgray, line width=1.0cm, line cap=round] (-3,4) +(180:.5cm)
    arc[start angle=0, end angle=-60, radius=3.5];
  \draw[name path = left1, thick] (-3,4)
  arc[start angle=0, end angle=-60, radius=4] node[near end, sloped]{$<$};
  \path [name path= left2]  (-4,2.5) -- (-2,2.5) ;
\fill [black, name intersections={of= left1 and left2, by=mb2}]
   (mb2) circle (2pt) node[left]{$m_2$};    
\draw[blue, thick] (mb1) .. controls +(10:1.3) and +(-80:1.3) .. node[teal, name=a13]{$\cdot$} node[near end, sloped]{$>$} (mb3);
\draw[blue, thick] (mb1) .. controls +(170:1.3) and +(-110:1.3) .. node[teal, name=a21]{$\cdot$} node[near start, sloped]{$>$} (mb2);   
\draw[blue, thick] (mb2) .. controls +(60:1.3) and +(120:1.3) .. node[teal, name=a32]{$\cdot$} node[near end, sloped]{$>$} (mb3);
\draw[teal, thick, ->] (a13) -- node[near start, below=-1pt]{$+$} node[near end, below=-1pt]{$+$} (a21);
\draw[teal, thick, ->] (a21) -- node[near start, left=-1pt]{$+$} node[very near end, left=-1pt]{$-$}(a32);
\draw[teal, thick, ->] (a32) -- node[very near start, right=-1pt]{$-$} node[near end, right=-1pt]{$+$}(a13);
\fill[teal] (a13) circle (1.5pt);
\fill[teal] (a21) circle (1.5pt);
\fill[teal] (a32) circle (1.5pt);
\end{tikzpicture}
\quad
\begin{tikzpicture}
\clip (-1.1,.3) rectangle (4.1,-4.2);
\draw[lightgray, line width=1.0cm, line cap=round] (0,0) +(150:.5cm)
    arc[start angle=150, end angle=210, radius=4.5];
\draw[name path = left1, thick] (0,0)
  arc[start angle=150, end angle=210, radius=4] node[near end, sloped]{$>$};
\path [name path= left3]  (1,-0.2) -- (-2,-0.2);
\fill [black, name intersections={of= left1 and left3, by=mb1}]
(mb1) circle (2pt) node[left]{$m_1$};
\path [name path= left2]  (1,-2.0) -- (-2,-2.0);
\draw [black, name intersections={of= left1 and left2, by=lb1}]
(lb1) node[green, name=LB1]{$\cdot$}; 
\path [name path= left4]  (1,-3.7) -- (-3,-3.7);
\fill [black, name intersections={of= left1 and left4, by=mb2}]
(mb2) circle (2pt) node[left]{$m_2$};
 \draw[lightgray, line width=1.0cm, line cap=round] (2.6,0) +(0:.5cm)
    arc[start angle=180, end angle=240, radius=3.5];
  \draw[name path =right1, thick] (2.6,0)
  arc[start angle=180, end angle=240, radius=4] node[very near start, sloped]{$<$};
  \path [name path=right2]  (4,-2) -- (2,-2);
\fill [black, name intersections={of= right1 and right2, by=mb3}]
   (mb3) circle (2pt) node[right]{$m_3$}; 
\draw[blue, thick] (mb1) .. controls +(10:1.3) and +(120:1.3) .. node[teal, name=a13]{$\cdot$} node[near end, sloped]{$<$} (mb3);
\draw[blue, thick] (mb2) .. controls +(-10:1.3) and +(-80:1.3) .. node[teal, name=a23]{$\cdot$} node[near end, sloped]{$<$} (mb3);
\draw[teal, thick, ->] (a13) -- node[near start, right=-1pt]{$+$} node[near end, right=-1pt]{$-$} (a23);
\draw[green, thick, ->] (a23) -- node[near start, below=-1pt]{$-$} node[near end, below=-1pt]{$+$} (LB1);
\draw[green, thick, ->] (LB1) -- node[near start, above=-1pt]{$+$} node[near end, above=-1pt]{$+$} (a13);

\fill[green] (lb1) circle (1.5pt);
\fill[teal] (a13) circle (1.5pt);
\fill[teal] (a23) circle (1.5pt);
\end{tikzpicture}
\quad
\begin{tikzpicture}
\clip (-1.5,.3) rectangle (2.5,-4.2);
\draw[lightgray, line width=1.0cm, line cap=round] (0,0) +(150:.5cm)
    arc[start angle=150, end angle=210, radius=4.5];
\draw[name path = left1, thick] (0,0)
arc[start angle=150, end angle=210, radius=4] node[very near start, sloped]{$<$};
\path [name path= left2]  (1,-0.2) -- (-2,-0.2);
\fill [black, name intersections={of= left1 and left2, by=mb1}]
(mb1) circle (2pt) node[left]{$m_1$};
\path [name path= left3]  (1,-1.0) -- (-2,-1.0);
\draw [name intersections={of= left1 and left3, by=lb1}]
(lb1) node[green, name=LB1]{$\cdot$};
\path [name path= left4]  (1,-2) -- (-2,-2);
\fill [black, name intersections={of= left1 and left4, by=mb2}]
(mb2) circle (2pt) node[left]{$m_2$};
\path [name path= left5]  (1,-3.0) -- (-2,-3.0);
\draw [black, name intersections={of= left1 and left5, by=lb2}]
(lb2) node[green, name=LB2]{$\cdot$}; 
\path [name path= left6]  (1,-3.7) -- (-3,-3.7);
\fill [black, name intersections={of= left1 and left6, by=mb3}]
(mb3) circle (2pt) node[left]{$m_3$};
\draw[blue, thick] (mb1) .. controls +(-10:3) and +(10:3) .. node[teal, name=a13]{$\cdot$} node[near start, sloped]{$>$} (mb3);
\draw[green, thick, ->] (LB1) -- node[near start, above]{$+$} node[near end, above]{$-$} (a13);
\draw[green, thick, ->] (a13) -- node[near start, below]{$-$} node[near end, below]{$+$}(LB2);

\fill[green] (lb1) circle (1.5pt);
\fill[green] (lb2) circle (1.5pt);
\fill[teal] (a13) circle (1.5pt);
\end{tikzpicture}
\caption{Skewed-gentle polarized quivers with fringing for 
  signature zero tagged triangulations}
\label{fig:tagtria} 
\end{figure}
Thus, the polarization of $\uQ$ depends on the orientation of the ``ordinary˝
edges, however the corresponding skewed-gentle algebra $\Ka\uQ$ does not depend
on those orientations.  It is isomorphic to the  Jacobian algebra  $\cP(T,W_T)$, see for example~\cite[Prop.~4.4]{QZ17}.

 \subsection{Relating admissible words and admissible curves}
\label{ssec:orb}
Recall, that we introduced in Section~\ref{ssec:tagged3ang} the signature zero tagged triangulations of $\bSu$ in the case $\Ma\neq\emptyset$.
In this case, each tagged arc of such a triangulation has at most one notch.

\begin{Def} \label{def-IdealT}
Let $T$ be a signature zero tagged triangulation of $\bSu$.  
We denote,  following for example in~\cite[Def.~5.8 and~5.10]{FT18}, by
$\glslink{Tro}{T^{\circ}}$ the \emph{associatedidealtriangulation}. 

Specifically, for each notched arc 
$\tau'_p=(\gam_p, c'_p)\in T$, let
$\tau_p=(\gam_p, c_p)\in T$ be its unnotched  version. Then we replace each
notched arc $(\gam, c')$ by
the minimal curve
$\gam'_p:=\overline{\gam_p}$,
which encloses $\gam$, see Figure~\ref{fig:idealT}.
\begin{figure}
\begin{tikzpicture}[scale=.8]
 \draw[lightgray, line width=1.0cm, line cap=round] (0,-0.1) +(-120:.5cm)
   arc[start angle=60, end angle=120, radius=3.5];
  \draw[name path = bottom1, thick] (0,0)
  arc[start angle=60, end angle=120, radius=4] node[near end, sloped]{$>$};
  \path [name path= bottom2]  (-2,0) -- (-2,3);
\fill [black, name intersections={of= bottom1 and bottom2, by=mb1}]
   (mb1)  node[below]{$m$}; 
\draw[blue, thick, double distance=3pt]
(mb1) -- node[left=-2pt]{$(\gam_p,c)=\tau_p\phantom{'}$} node[right]{$\tau_p'=(\gam_p,c')$} node[very near end, sloped, below=-5pt]{$\bowtie$}  +(90:1.5) node[name=pt2, black, above=-3pt]{$\times$} node[black, above=6pt]{$p$};
\fill[black] (mb1) circle (2pt);
\node[scale=2] at (1.2,2){$\stackrel{\sig}{\mapsto}$};
 \draw[lightgray, line width=1.0cm, line cap=round] (6,-0.1) +(-120:.5cm)
    arc[start angle=60, end angle=120, radius=3.5];
  \draw[name path = bottom3, thick] (6,0)
  arc[start angle=60, end angle=120, radius=4] node[near end, sloped]{$>$};
  \path [name path= bottom4]  (4,0) -- (4,3);
\fill [black, name intersections={of= bottom3 and bottom4, by=mb2}]
   (mb2)  node[below]{$m$}; 
\draw[blue, thick]
(mb2) -- node[near end, left=-2pt]{$\gam_p$}  +(90:1.5) node[name=pt3, black, above=-3pt]{$\times$} node[black, above=6pt]{$p$};
\draw[blue, thick]
(mb2) .. controls +(120:4) and +(60:4) .. node[very near end, right]{$\gam_p'$}(mb2);
\fill[black] (mb2) circle (2pt);
\end{tikzpicture}
\caption{Construction of $T^{\circ}$}
\label{fig:idealT}
\end{figure}
Strictly speaking, $\gam_p'\in\tcAC(\bSu)\setminus\cAC(\bSu)$ since it cuts out a based kink. In any case,
we will consider $\gam_p'\in T^{\circ}$ as an arc.
In the notation of~\cite[Def.~5.8]{FT18} $\gam_p'=\tau^{-1}(\gam_p,c')$, whilst
$\tau^{-1}(\bet,b)=\bet$ for all arcs $(\bet,b)\in T$ with $b=(1,1)$, i.e.
for all tagged arcs which are not notched.
In view of the frequent use of the letter $\tau$ in this manuscript,
we will write in this situation $\sig$ in place of $\tau^{-1}$. 
\end{Def}

Observe that for a signature zero tagged triangulation $T$,
each puncture $p\in\Pu$ is part of a self-folded triangle $\Del_p$ of $T^{\circ}$.
Moreover, we choose in each triangle $\Del$ of $T^{\circ}$ a point
$x_\Del\in\Del\setminus (T^{\circ}\cup\dSu)$. 

The \gls{dualgraph} $\glslink{Tro}{\Gam(T^{\circ})}$ has by definition as
vertices $\Gam_0(T^{\circ})$ the set of (ideal) triangles $\Del\subset T^{\circ}$,
and two triangles are connected by
 an edge, if they share an arc. In particular, each self-folded triangle 
 $\Del_p$ has a loop $\eps_p$ in 
 $\Gam(T^{\circ})$. We usually identify the vertices
 $\Del$ of $\Gam(T^{\circ})$ with the set of points
\[ 
 E_T:=\{s_\Del\mid \Del\overset{\text{triang.}}{\subset} T^{\circ}\}. 
\]
\subsubsection*{Example} 
We consider the twice punctured bigon 
$\bSig$ from Section~\ref{ssec:ExIntersecN}, together
with the same signature zero triangulation $T$. In
Figure~\ref{fig:dualGr} we  display in  blue the associated
ideal triangulation $T^o$. Note, that $T^o$ consists of 4
triangles, with two  of them self-folded. In the same figure,
we display the dual graph $\Gam(T^o)$ in red, where we place
each vertex inside the corresponding triangle.
\begin{figure}
\begin{tikzpicture}[scale=.6]
\coordinate (mk1) at (0,-3);
\coordinate (mk2) at (0,3);
\coordinate (pk3) at ($(mk1) + (61.5:3)$);
\coordinate (pk1) at ($(mk1) + (120:3)$);
\coordinate (tr1) at (0, 2.5);
\coordinate (tr2) at (0, 1.0);
\coordinate (tr3) at ($(tr2) + (-64:2)$); 
\coordinate (tr4) at ($(tr2) + (-118:2)$);
\draw[lightgray, line width =.5cm] (0,0) circle[radius=3.44cm];
\draw[thick, black] (0,0) circle [radius=3cm];
\fill [black] (mk1) node[below]{$m_1$} circle (2pt);  
\fill [black] (mk2) node[above]{$m_2$} circle (2pt); 
\draw[thick, blue] (0,-.5) circle [radius=2.5cm];
\draw[blue, thick] (mk1) -- (pk3) node[black, name = pt3] {$\times$} node[right, black]{\footnotesize $p_3$};
\draw[blue, thick] (mk1) -- +(75:3)  arc [start angle=165, end angle=-23, radius=.7]  --(mk1);
\draw[blue, thick] (mk1) -- (pk1) node[black, name = pt1]{$\times$} node[left, black]{\footnotesize $p_1$};
\draw[blue, thick] (mk1) -- +(133:3.2)  
arc [start angle=210, end angle=10, radius=.7]  --(mk1);
\fill[red] (tr1) node[name = nt1]{ }  circle (1.8pt);
\fill[red] (tr2) circle (1.8pt);
\fill[red] (tr3) circle (1.8pt);
\fill[red] (tr4) circle (1.8pt);
\draw[red] (tr1) -- (tr2) -- (tr3);
\draw[red] (tr2) -- (tr4);
\draw[red]  (tr3) + (-65:0.3) circle[radius=0.3cm];
\draw[red]  (tr4) + (-117:0.3) circle[radius=0.3cm];
\end{tikzpicture}
\caption{Dual graph of an ideal triangulation}
\label{fig:dualGr}
\end{figure}

As pointed out  in~\cite[Prop.~5.1]{AP17} 
(see~\cite[Sec.~2.2]{Am16} for more details), there  is a deformation retract 
$r\df \Su\setminus\Pu\ra\Gam(T^{\circ})$.
Following~\cite[Sec.~5.1]{AP17}, this induces an isomorphism of orbifold groupoids
 \[
   \glslink{barr}{\overline{r}}\df \piorb(\Su, E_T)\xrightarrow{\sim} \piorb(\Gam(T^{\circ}), E_T),
 \]
 where the orbifold groupoid $\piorb(\Gam(T^{\circ}), E_T)$ of the graph
 $\Gam(T^{\circ})$ is  the usual groupoid 
 $\pi_1(\Gam(T^{\circ}), E_{T^{\circ}})$ modulo  the relations  
 $\eps_p^2=\id_{\Del_p}$ for all $p\in\Pu$, see \cite[\S2.1 and \S6]{GLF23}.
 
 Thus, by slightly restating the discussion
 in~\cite[Sec.~5]{AP17} we have the following result:

 \begin{Prop} \label{prp:AP}
   Let $T=(\tau_p,\tau'_p)_{p\in\Pu}\cup (\tau_i)_{i=1,2,\ldots,n-2\abs{\Pu}}$
     be a signature zero tagged triangulation
   of a marked surface $\bSu$   with non-empty boundary.
   Then the isomorphism of orbifold groupoids   
   $\overline{r}\df \piorb(\Su, E_T)\xrightarrow{\sim} \piorb(\Gam(T^{\circ}), E_T)$
   induces  natural bijections:
\[\xymatrix{
  \piorb(\Su,\Ma)^* \ar[r]^{r_s}_{\sim}& \DSt(\uQ(T)) &\text{ and}\\ 
  \piorbfree(\Su)\setminus\{([\eps_p])_{p\in\Pu}\}\ar[r]^(0.6){r_b}_(0.6){\sim}& [\Ba(\uQ(T))]\\
  \piorbfreepr(\Su)\ar@{^{(}->}[u]\ar[r]_{\sim} &[\pBa(\uQ(T))]\ar@{^{(}->}[u]
}
\]
where $\eps_p$ denotes the loop, which encloses precisely the puncture $p$, and $\piorbfreepr(\bSu)$ denotes the classes of the primitive loops. Moreover, with the notation from Section~\ref{ssec:StBa} we set formally
\[
\DSt(\uQ):=\St(\uQ)\cup\{-\tbs_i^{\pm 1}\mid i\in Q_0\}.
\]
We have in particular with the above notation 
$r_s(\gam'_p)=-\tbs_p$ for all $p\in\Pu$ and 
$r_s(\gam_i)=-\tbs_i$ for $i=1,2,\ldots, n-2\abs{\Pu}$.
\end{Prop}

Note, that the loops $\eps_p$ with $p\in\Pu$ are \emph{not} primitive, since
they are torsion elements in the orbifold fundamental group. 

\begin{Rem} \label{rem:QZ}
In our notation, Qiu and Zhou defined in~\cite[Sec.~4.2]{QZ17}  (more
precisely in their Construction~4.7 and Construction~4.9) for a signature zero tagged triangulation $T$  an
injective map $\fra\df\St(\huQ(T))\ra\cAC'(\bSu)$. Here
$\cAC'(\bSu)$ denotes the homotopy classes of \emph{all} curves on
$\bSu$ with both endpoints in $\Pu\cup\Ma$.
Thus, $\cAC'(\bSu)\supset\cAC(\bSu)$, since in $\cAC'(\bSu)$
kinks are allowed.
Qiu and Zhou observed~\cite[Lem.~4.12]{QZ17}, that $\fra$ restricts to an injective map
$\glslink{frasb}{\fra_s}\df\DAdm_s(\uQ(T))\ra\cAC(\bSu)$, 
where we agree that $\fra_s(\bs_i)=\tau_i$.

Similarly, we obtain an injective map
$\glslink{frasb}{\fra_b}:[\AdmBa(\uQ(T))]\ra\cAL(\bSu)$.
With the notation from Remark~\ref{rem:Adm} it is clear tha we obtain two commutative diagrams, which are displayed in the 
Diagram~\eqref{eq:2com} below.
In both diagrams, the map $\ebrace{\widebar{?}}$ sends a curve $\gam$ to
the orbifold class $\ebrace{\widebar{\gam}}$ of its completion $\widebar{\gam}$.
These two maps are certainly surjective, since each curve is, up to the orbifold relations, equivalent to a curve without proper kinks. See Remark~\ref{rem:Adm} for the definition of the first row in both diagrams.
\begin{equation} \label{eq:2com}
\vcenter{\xymatrix{\DAdm_s(\uQ(T))\ar[d]_{\fra_s}\ar[r]^{(\widebar{?})_s}&\DSt(\uQ(T))\\
\cLC_s(\bSu)\ar[r]_{\ebrace{\widebar{?}}_s}\ &\piorb(\Su,\Ma)^*\ar[u]_{r_s}}}
\quad\text{and}\quad
\vcenter{\xymatrix{[\Adm_b(\uQ(T))]\ar[d]_{\fra_b}\ar[r]^(0.5){[\widebar{?}]}&[\pBa(\uQ(T))]\\
\cLC_b(\bSu)\ar[r]_(0.55){\ebrace{\widebar{?}}_b}\ &\piorbfree(\Su)\ar[u]_{r_b}}}
\end{equation}
\end{Rem}

\begin{Prop} \label{prp:fra-bij}
The maps $\fra_s$ and $\fra_b$ from Diagram~\eqref{eq:2com}
are bijective. Together they provide a natural bijection
$\fra\df [\DAdm(\uQ(T))]\ra\cLC(\bSu)$.
\end{Prop}

\begin{proof}  Both maps are injective, as we already observed.  Since the
maps $r_s$ and $r_b$ are bijective by Proposition~\ref{prp:AP}, in view of
Remark~\ref{rem:Adm} it is sufficient to show the following two claims:
\begin{itemize}
\item  
  If $\ebrace{\alp}_s=\ebrace{\bet}_s$ then either $\alp=\bet$, or
  $\alp=\bet^{-1}$ and $r_s(\ebrace{\alp}_s)$ is a symmetric string.
\item
  If $\ebrace{\lam}_b=\ebrace{\mu}_b$, then either $\lam=\mu$ or
  $\lam=\mu^{-1}\in\cC_b(\Su,\Pu,\Ma)$.
\end{itemize}
However, this is the content of the main result of~\cite{GLF23}.
\end{proof}

\subsection{Intersections of marked curves and combinatorial E-invariant} 
We  define the \emph{rotation}
$\gls{rho12}\df\cAC(\bSu)\ra\cAC(\Su,\Pu, \dSu\setminus\Ma)$,  which moves
the endpoints of a curve $\gam$, which lie in $\Ma$ to the
next boundary segment in direction of the induced orientation of $\dSu$ whilst
the endpoints which lie in $\Pu$ are not moved.

Thus,  $\rho^{1/2}(\gam)=\gam$ if $\gam$ is of type $(p,p)$.  We also agree,
that $\rho^{1/2}$ acts on $\cAL(\bSu)$ as identity.
On the other hand,  let us denote by $\gls{frasb}{\fra_f}\df\AdmSt(\uQ^f)\ra\cC(\Su,\Pu,
\dSu\setminus\Ma)$ the natural extension of the map
$\fra\df\AdmSt(\uQ)\ra\cC(\Su,\Pu,\Ma)$, where we might think of having an
additional marked point sitting just outside of each boundary segment.
See Figure~\ref{fig:rot1}.  Here $\uQ=\uQ(T)$ is the skewed-gentle polarized
quiver associated to a signature zero tagged triangulation $T$ of $\bSu$ and $\uQ^f=\uQ^f(T)$ is the corresponding canonical fringing, see Section~\ref{ssec:tagged3ang}.

Recall moreover, that in Section~\ref{ssec:AR-adm} we defined the AR-operator
$\tau_f\df\Adm(\uQ(T))\ra\Adm(\uQ^f(T))$. We extend this operator
to $\DAdm(\uQ)$ by setting $\tau_f(-\bs_i):=\bq_i^{(f)}$ for all $i\in Q_0$,
where 
\[
\bq_i^{(f)}:=\begin{cases}
\II_{i,-1}^{-1}\bx_{i,-1}^{\min}  &\text{if } i\in\Qspv,\\
(\bx_{i,+1}^{\min})^{-1}\bx_{i,-1}^{\min} &\text{if } i\in\Qordv
\end{cases}
\]
and $\bx_{i,\rho}^{\min}$ is the unique right inextensible word for $\huQ^f$
which contains no inverse letters and 
$\II_{i,\rho}^{-1}\bx_{i,\rho}^{\min}$ is a string.
Note, that with the notation from~Section~\ref{ssec:StBa} we have
\[
\overline{\bq_i^{(f)}}\in\cI_{\uQ^{(f)}}.
\]
Similarly, with the notation of~\cite[Sec.~2.6]{Ge23} we have
$\bar{\bq}_i^{(f)}=\tbq_i\in\cQ(\uQ^f)$ for all $i\in Q_0$.  

\begin{figure}[ht]
  \begin{tikzpicture}
\clip (-0.3,-.3)  rectangle  (-4, 4.2);    
 \draw[lightgray, line width=1.0cm, line cap=round] (0,0) +(-120:.5cm)
    arc[start angle=60, end angle=120, radius=3.5];
  \draw[name path = bottom0, thick] (0,0)
  arc[start angle=60, end angle=120, radius=4] node[very near end, sloped]{$>$};
  \path [name path= bottom1]  (-3,0) -- (-3,3);
\fill [black, name intersections={of= bottom0 and bottom1, by=mb1}]
   (mb1) circle (2pt) node[below]{$m_1$}; 
  \path [name path= bottom2]  (-1,0) -- (-1,2);
\fill [black, name intersections={of= bottom0 and bottom2, by=mb2}]
   (mb2) circle (2pt) node[below]{$m_2$}; 
   \coordinate (mb12) at (-2,0); \fill[gray] (mb12) circle (2pt);
\draw[gray] (mb1) -- (mb12) -- (mb2);    
\draw[lightgray, line width=1.0cm, line cap=round] (-4,4) +(60:.5cm)
    arc[start angle=240, end angle=300, radius=3.5];
  \draw[name path = top0, thick] (-4,4)
 arc[start angle=240, end angle=300, radius=4] node[very near end, sloped]{$<$};
  \path [name path= top2]  (-3,4) -- (-3,2);
\fill [black, name intersections={of= top0 and top2, by=mb4}]
(mb4) circle (2pt) node[above]{$m_4$};
  \path [name path= top1]  (-1,4) -- (-1,3);
\fill [black, name intersections={of= top0 and top1, by=mb3}]
   (mb3) circle (2pt) node[above]{$m_3$}; 
   \coordinate (mb34) at (-2,4); \fill[gray] (mb34) circle (2pt);
\draw[gray] (mb3) -- (mb34) -- (mb4);    
\draw[violet, thick] (mb3) ..controls  +(-160:1) and  +(20:1) .. node[near start, right]{$\alpha$} (mb1);
\draw[violet, opacity=.6, thick] (mb34) ..controls  +(-165:1) and  +(25:1) .. node[near end, above right, scale=.75]{$\rho^{1/2}(\alpha)$}(mb12);
\coordinate (mp) at (-3.8,2.6); \fill[draw] (mp) node{$\times$} node[left=3pt]{$p$};
\draw[magenta, thick]  (mp) .. controls +(-10:1.5) and +(65:1.0) .. node[near start, below left]{$\bet$} (mb1);
\draw[magenta, thick, opacity=.6]  (mp) .. controls +(5:1.5) and +(70:1.0) .. node[very near start, above=-1pt, sloped, scale=.75]{$\ \rho^{1/2}(\bet)$} (mb12);
\end{tikzpicture}
\quad
\begin{tikzpicture}
\clip (0.7,-.4)  rectangle  (-4.5, 4.1);      
\draw[lightgray, line width=1.0cm, line cap=round] (-4,4) +(60:.5cm)
  arc[start angle=240, end angle=300, radius=3.5];
\draw[name path = top0, thick] (-4,4)
 arc[start angle=240, end angle=300, radius=4] node[very near end, sloped]{$<$};
\path [name path= top1]  (-1,4) -- (-1,3);
\fill [black, name intersections={of= top0 and top1, by=mb1}]
   (mb3) circle (2pt) node[above]{$m_1$}; 
\coordinate (mb12) at (-2,4); \fill[gray] (mb12) circle (2pt);
\path [name path= top2]  (-3,4) -- (-3,2);
\fill [black, name intersections={of= top0 and top2, by=mb2}]
(mb4) circle (2pt) node[above]{$m_2$};
\draw[gray] (mb1) -- (mb12) -- (mb2);    
 \draw[lightgray, line width=1.0cm, line cap=round] (0,0) +(-120:.5cm)
    arc[start angle=60, end angle=120, radius=3.5];
\draw[name path = bottom0, thick] (0,0)
  arc[start angle=60, end angle=120, radius=4] node[very near end, sloped]{$>$};
\path [name path= bottom1]  (-3.9,0) -- (-3.9,3);
\fill [black, name intersections={of= bottom0 and bottom1, by=mb3}]
   (mb3) circle (2pt) node[below]{$m_3$}; 
\path [name path= bottom2]  (-2,0) -- (-2,2);
\fill [black, name intersections={of= bottom0 and bottom2, by=mb4}]
   (mb4) circle (2pt) node[below]{$m_4$}; 
\path [name path= bottom2]  (-.1,0) -- (-.1,2);
\fill [black, name intersections={of= bottom0 and bottom2, by=mb5}]
   (mb5) circle (2pt) node[below]{$m_5$}; 
\path [name path= bottom34]  (-3,0) -- (-3,2);
\fill [green, name intersections={of= bottom0 and bottom34, by=s34}]
     (s34) node[name=S34]{$\cdot$}; 
\path [name path= bottom45]  (-1,0) -- (-1,2);
\fill [green, name intersections={of= bottom0 and bottom45, by=s45}]
   (s45) node[name=S45]{$\cdot$}; 
   
\coordinate (mb34) at (-2.8,0); \fill[gray] (mb34) circle (2pt);
\draw[gray] (mb3) -- (mb34) -- (mb4);    
\coordinate (mb45) at (-1.1,0); \fill[gray] (mb45) circle (2pt);
\draw[gray] (mb4) -- (mb45) -- (mb5);    
\draw[blue, thick] (mb3) .. controls +(90:1) and +(-130:1) .. node[name=a1, teal, near end]{$\cdot$}   (mb2);
\draw[blue, thick] (mb2) -- node[name=a2, teal, near start]{$\cdot$} (mb4);
\draw[blue, thick] (mb4) -- node[name=a3, teal, near end]{$\cdot$}   (mb1);
\draw[blue, thick] (mb1) .. controls +(-50:1) and (90:1) .. node[name=a4, teal, near start]{$\cdot$} (mb5);
\draw[teal, thin, ->] (a1) -- (a2);
\draw[teal, thin, ->] (a3) -- (a2);
\draw[teal, thin, ->] (a3) -- (a4);
\draw[green, thin, ->] (a4) -- (S45);
\draw[green, thin, ->] (S45) -- (a3);
\draw[green, thin, ->] (a2) -- (S34);
\draw[green, thin, ->] (S34) -- (a1);
\coordinate (mpl) at (-5,2.5); \coordinate (mpr) at (1.1, 2.5);
\draw[magenta, thick, name path=bet] (mb4) .. controls +(120:1) and +(0:1) .. node[near end, below]{$\ \bet$} (mpl);
\draw[magenta, thick, opacity=.6, name path=rbet] (mb45) .. controls +(100:2) and +(0:1) .. node[near end, above=-1pt, scale=.75, sloped]{$\rho^{1/2}(\bet)$} (mpl);
\draw[violet,thick, name path=alp] (mb3) .. controls +(90:2) and +(180:2) .. node[near end, above right]{$\hspace{1em} \alp$}(mpr);
\draw[violet, thick, opacity=.6, name path=ralp] (mb34) .. controls +(82 :2) and +(185:1) .. node[near end, below=-1pt, scale=0.75, sloped]{$\hspace{.9em} \rho^{1/2}(\alp)$}(mpr);
\fill [gray, opacity=.5, name intersections={of= alp and bet, by=albe}]
   (albe) circle (2pt); 
\fill [gray, opacity=.5, name intersections={of= ralp and rbet, by=ralbe}]
   (ralbe) circle (2pt); 
 \end{tikzpicture}
\quad
\begin{tikzpicture}
\clip (-2.1,1.6) circle (2.45cm);
\draw[lightgray, line width=1.0cm, line cap=round] (0,0) +(-120:.5cm)
    arc[start angle=60, end angle=120, radius=3.5];
\draw[name path = bottom0, thick] (0,0)
  arc[start angle=60, end angle=120, radius=4] node[near start, sloped]{$>$};
  \path [name path= bottom1]  (-3.4,0) -- (-3.4,3);
\fill [black, name intersections={of= bottom0 and bottom1, by=mb1}]
   (mb1) circle (2pt) node[below]{$m_1$};   
  \path [name path= bottom2]  (-.6,0) -- (-.6,3);
\fill [black, name intersections={of= bottom0 and bottom2, by=mb2}]
   (mb2) circle (2pt) node[below]{$m_2$};   
\coordinate (mb12) at (-2.2,0); \fill[gray] (mb12) circle (2pt);
\draw[gray] (mb1) -- (mb12) -- (mb2);    
\path [name path= bottom12]  (-2.3, 0) -- (-2.3, 2);
\fill [green, name intersections={of= bottom0 and bottom12, by=s12}]
     (s12) node[name=S12]{$\cdot$}; 

 \draw[lightgray, line width=1.0cm, line cap=round] (-1,4) +(0:.5cm)
    arc[start angle=180, end angle=240, radius=3.5];
  \draw[name path =right0, thick] (-1,4)
  arc[start angle=180, end angle=240, radius=4] node[near end, sloped]{$<$};
  \path [name path=right1]  (2,2.8) -- (-2,2.8);
\fill [black, name intersections={of= right0 and right1, by=mb3}]
   (mb3) circle (2pt) node[right]{$m_3$}; 

 \draw[lightgray, line width=1.0cm, line cap=round] (-3,4) +(180:.5cm)
    arc[start angle=0, end angle=-60, radius=3.5];
  \draw[name path = left0, thick] (-3,4)
  arc[start angle=0, end angle=-60, radius=4] node[near end, sloped]{$<$};
  \path [name path= left1]  (-4,2.8) -- (-2,2.8) ;
\fill [black, name intersections={of= left0 and left1, by=mb4}]
(mb4) circle (2pt) node[left]{$m_4$};
\draw[blue, thick] (mb1) -- node[teal, name=a13, inner sep=1pt]{$\cdot$} (mb3);
\draw[blue, thick] (mb1) -- node[teal, name=a14]{$\cdot$} (mb4);
\draw[blue, thick] (mb3) -- node[teal, name=a34]{$\cdot$} (mb4);
\draw[blue, thick] (mb2) -- node[teal, name=a23]{$\cdot$} (mb3);
\draw[teal, thin, ->] (a14) -- (a34);
\draw[teal, thin, ->] (a34) -- (a13);
\draw[teal, thin, ->] (a13) -- (a14);
\draw[teal, thin, ->] (a13) -- (a23);
\draw[green, thin, ->] (a23) -- (S12);
\draw[green, thin, ->] (S12) -- (a13);
\coordinate (UP) at (-1.5, 4.5); 
\draw[violet, thick, name path= alp] (mb1) .. controls +(60:1) and +(270:1) .. node[near end, left]{$\alp$} (UP);
\draw[violet, thick, opacity=.6, name path= ralp] (UP) .. controls +(275:1) and +(75:1) .. node[near start, below right=-1pt, scale=.75, sloped]{$\hspace{-2em}\rho^{1/2}(\alp)$} (mb12);
\coordinate (LD) at (-4.5, 0.3); \coordinate (RD) at (0, 0.8);
\draw[magenta, thick, name path=bet] (LD) .. controls +(50:2.5) and +(125:2.5) .. node[very near start, above=-2pt]{$\bet$}(RD);
\coordinate (rLD) at (-4.5, 0); \coordinate (rRD) at (0, 0.6);
\draw[magenta, thick, opacity=.6, name path=rbet] (rLD) .. controls +(50:2.5) and +(125:2.5) .. node[very near start, scale=.75, sloped, below=-2pt]{$\rho^{1/2}(\bet)$} (rRD);
\fill [darkgray, opacity=.5, name intersections={of= alp and bet, by=ints}]
(ints) circle (2pt);
\fill [darkgray, opacity=.5, name intersections={of= ralp and rbet, by=rints}]
   (rints) circle (2pt);
\end{tikzpicture}
 \caption{Operator $\rho^{1/2}$ on curves and near boundary  intersections}
\label{fig:rot1} 
\end{figure}
\begin{Lem}  \label{lem:rot}
Let $\uQ=\uQ(T)$ be the skewed-gentle polarized quiver with
boundary fringing $\uQ^f=\uQ^f(T)$ corresponding to a signature zero
triangulation $T$  of  a punctured surface $\bSu$.  Then, for each
  $\bx\in\DAdm(\uQ)$ we have $\fra_f(\tau_{\uQ^f}(\bx))=\rho^{1/2}\fra(\bx)$.
\end{Lem}  

\begin{proof} In view of the above interpretation of $\fra_f$ (with an additional, virtual marked point sitting just outside each boundary segment), 
this is just combinatorial content of the well known statement
of~\cite[Thm.~5.2]{QZ17}.
\end{proof}

We are now ready to state and prove our key result, which extends the result of Yu Qiu and Yu Zhou~\cite[Sec.~7]{QZ17} from tagged curves to
modulated curves and loops. 

\begin{Thm} \label{thm:intersec}
Let $T$ be a signature zero tagged triangulation of a marked surface with non-empty boundary $\bSu$.
\begin{itemize}
\item[(a)]    
The map
\[  
\fra^*\df [\DAdm^*(\uQ(T))]\ra \cLC^*(\bSu), (\bx,s)\mapsto (\fra(x),s)
\]
is a bijection which is compatible with the respective equivalence relations
$\simeq$. 
\item[(b)]
For $(\bx,s), (\by,t)\in\DAdm^*(\uQ)$ we have
  \[
    e_\uQ((\bx,s), (\by,t)))=\Intn^*(\fra^*(\bx,s), \fra^*(\by,t)).
    \]
\end{itemize}
\end{Thm}

\begin{proof}
(a) In view of Remark~\ref{rem:QZ} together with the definitions of
  $\DAdm^*(\uQ(T))$ (see Section~\ref{ssec:combEg})
  resp.~of $\cLC^*(\bSu)$ (see Section~\ref{ssec:taggedC}),
  it is clear that for $(\bx,s)\in\DAdm^*(\bSu)$
we have indeed
  $(\fra(\bx),s)\in\cLC^*(\bSu)$, since
  under the map $\fra$ the elements of $\Adm(\uQ)$ of a given type
  are sent to curves of the same type.
Now it follows from Proposition~\ref{prp:fra-bij}, that $\fra^*$ is bijective.
  
(b-i) Assume first $\bx,\by\in\Adm(\uQ)$. 
By comparing the definition of $\Intn^*((\fra^*(\bx,s), \fra^*(\by,t))$
in Section~\ref{ssec:taggedC} with the formula for the combinatorial
$E$-invariant 
  $e_\uQ((\bx,s), (\by,t))$ in Section~\ref{ssec:combEg}, we
  see, that it is sufficient to prove 
\begin{equation} \label{eqn:int1-E1}
  \intn(\fra(\bx),\fra(\by))= \abs{\sfA_{\uQ^f(T)}(\bx,\by)}.
\end{equation}

 We  observe first that the number of internal intersection points between $\fra(\bx)$ and $\fra(\by)$ does not change under  simultaneous rotation by $\rho^{1/2}$.   Thus, by 
 Lemma~\ref{lem:rot},  it is sufficient to establish 
 a bijection from the set of internal intersection points
 between $\rho^{1/2}(\fra(\bx))=\fra_f(\tbx)$ and
 $\rho^{1/2}(\fra(\by))=\fra_f(\tby)$, to the set of  K-triples  of type
 $\sfA$ in  $\cK_{\uQ^f}(\tbx,\tby)\cup\cK_{\uQ^f}(\tby,\tbx)$, 
where $\uQ^f=\uQ^f(T)$ is the
boundary fringing for $\uQ=\uQ(T)$, and $\tbx:=\tau_{\uQ^f}(\bx)$,
$\tby:=\tau_{\uQ^f}(\by)$.  Now,  the rotated curves 
$\fra_f(\tbx)$ and $\fra_f(\tby)$  have no endpoint in $\Ma$, see
Figure~\ref{fig:rot1}. It follows that all internal intersections  between
$\fra_f(\tbx)$ and $\fra_f(\tby)$ are  of one of the $16 = 4\times 4$ types
described in Figure~\ref{fig:intersectA}, up to the possible exchange
of the roles of $\fra_f(\tbx)$ and $\fra_f(\tby)$.  More precisely, we
claim that the intersections described in Figure~\ref{fig:intersectA}
correspond precisely to the kisses of type $\sfA$ in $\cK_{\uQ^f}(\tbx,\tby)$.
\begin{figure}[ht]
\begin{tikzpicture}
\draw[blue, thick, dashed] (-2,1) node[black]{$\bullet$}  -- (2,1) node[black]{$\bullet$};
\draw[blue, thick] (2,1) .. controls +(290:1) and +(70:1) .. node[right]{$\beta$} (2,-1);
\draw[blue, thick, dashed] (2,-1) node[black]{$\bullet$} -- (-2,-1) node[black]{$\bullet$};
\draw[blue, thick] (-2,-1) .. controls +(110:1) and +(250:1) .. node[left]{$\alpha$} (-2,1);
\draw[violet, thick, name path=mx] (-2.5,0.5) .. controls +(0:3) and +(180:3) .. (2.5,-.5);
\draw[magenta, thick, name path=my] (-2.5,-0.5) .. controls +(0:3) and +(180:3) .. (2.5,.5);
\fill[red, name intersections={of=mx and my, by=int}] (int) circle (2pt) node[above]{$i$};
\coordinate (LT1) at (-5,3);
\draw[blue, thick] (LT1) node[black]{$\bullet$} .. controls   +(30:.5) and  +(255:.5) .. ++(1,1) node[black]{$\bullet$}
                   .. controls ++(255:.5) and ++(120:.5) .. node[right]{$\alpha$} ++(0,-2) node[black]{$\bullet$}
                   .. controls ++(120:.5) and ++(330:.5) .. ++(-1,1);
\draw[violet, thick] ($(LT1)+(0,1)$) .. controls +(300:.5) and +(180:1) .. +(1.5,-.7);
\draw[magenta, thick] ($(LT1)+(0,-1)$) .. controls +(60:.5) and +(180:1) .. +(1.5,.7);
\node at ($(LT1) + (-2,0)$){$(\mathrm{I}_l)$}; 
\coordinate (LT2) at (-6, 1);
\draw[blue, thick, double distance=3pt] (LT2) node[black]{$\bullet$} -- node[very near end, sloped, below=-5pt]{$\bowtie$}  +(0:1.5) node[name=pt1, black,  right=-3pt]{$\times$};
\draw[violet, thick] ($(LT2)+(2.5,.5)$) .. controls +(180:1) and +(75:1) .. ($(LT2)+(0.3,-.7)$);
\draw[magenta, thick] ($(LT2)+(2.5,-.5)$) .. controls +(180:1) and +(285:1) .. ($(LT2)+(0.3,.7)$);
\node at ($(LT2) + (-1,0)$){$(\mathrm{II}_l)$}; 
\coordinate (LT3) at (-6,-1);
\draw[blue, thick, double distance=3pt] (LT3) node[black]{$\bullet$} -- node[very near end, sloped, below=-5pt]{$\bowtie$}  +(0:1.5) node[name=pt2, black,  right=-3pt]{$\times$};
\draw[violet, thick] ($(LT3)+(2.5,.7)$) .. controls +(180:1) and +(75:1) .. ($(LT3)+(0.3,-.7)$);
\draw[magenta, thick] ($(LT3)+(2.5,0)$) -- (pt2);
\node at ($(LT3) + (-1,0)$){$(\mathrm{III}_l)$}; 
\coordinate (LT4) at (-6, -3);
\draw[blue, thick, double distance=3pt] (LT4) node[black]{$\bullet$} -- node[very near end, sloped, below=-5pt]{$\bowtie$}  +(0:1.5) node[name=pt3, black,  right=-3pt]{$\times$};
\draw[violet, thick] ($(LT4)+(2.5,0)$) -- (pt3);
\draw[magenta, thick] ($(LT4)+(2.5,-.7)$) .. controls +(180:1) and +(285:1) .. ($(LT4)+(0.3,.7)$);
\node at ($(LT4) + (-1,0)$){$(\mathrm{IV}_l)$}; 
\coordinate (RT1) at (5,3);
\draw[blue, thick] (RT1) node[black]{$\bullet$} .. controls   +(160:.5) and  +(290:.5) .. ++(-1,1) node[black]{$\bullet$}
                   .. controls ++(290:.5) and ++(70:.5) .. node[left]{$\beta$} ++(0,-2) node[black]{$\bullet$}
                   .. controls ++(70:.5) and ++(210:.5) .. ++(1,1);
\draw[magenta, thick] ($(RT1)+(0,1)$) .. controls +(250:.5) and +(0:1) .. +(-1.5,-.7);
\draw[violet, thick] ($(RT1)+(0,-1)$) .. controls +(130:.5) and +(0:1) .. +(-1.5,.7);
\node at ($(RT1) + (2,0)$){$(\mathrm{I}_r)$}; 
\coordinate (RT2) at (6, 1);
\draw[blue, thick, double distance=3pt] (RT2) node[black]{$\bullet$} -- node[very near end, sloped, below=-5pt]{$\bowtie$}  +(180:1.5) node[name=pr1, black,  left=-3pt]{$\times$};
\draw[magenta, thick] ($(RT2)+(-2.5,.5)$) .. controls +(0:1) and +(105:1) .. ($(RT2)+(-0.3,-.7)$);
\draw[violet, thick] ($(RT2)+(-2.5,-.5)$) .. controls +(0:1) and +(255:1) .. ($(RT2)+(-0.3,.7)$);
\node at ($(RT2) + (1,0)$){$(\mathrm{II}_r)$}; 
\coordinate (RT3) at (6,-1);
\draw[blue, thick, double distance=3pt] (RT3) node[black]{$\bullet$} -- node[very near end, sloped, below=-5pt]{$\bowtie$}  +(180:1.5) node[name=pr2, black,  left=-3pt]{$\times$};
\draw[magenta, thick] ($(RT3)+(-2.5,.7)$) .. controls +(0:1) and +(105:1) .. ($(RT3)+(-0.3,-.7)$);
\draw[violet, thick] ($(RT3)+(-2.5,0)$) -- (pr2);
\node at ($(RT3) + (1,0)$){$(\mathrm{III}_r)$}; 
\coordinate (RT4) at (6, -3);
\draw[blue, thick, double distance=3pt] (RT4) node[black]{$\bullet$} -- node[very near end, sloped, below=-5pt]{$\bowtie$}  +(180:1.5) node[name=pr3, black,  left=-3pt]{$\times$};
\draw[magenta, thick] ($(RT4)+(-2.5,0)$) -- (pr3);
\draw[violet, thick] ($(RT4)+(-2.5,-.7)$) .. controls +(0:1) and +(265:1) .. ($(RT4)+(-0.3,.7)$);
\node at ($(RT4) + (1,0)$){$(\mathrm{IV}_r)$}; 
\end{tikzpicture}
\caption{Types of intersection between $\viox{\fra_f(\tbx)}$ and
  $\magx{\fra_f(\tby)}$}
\label{fig:intersectA} 
\end{figure}
Now, let $(G,\phi_q,\phi_s)\in\cK_{\uQ^f}(\tbx,\tby)$ be a K-triple of 
type~$\sfA$ with $G\df\uH\ra\uQ^f$.
Recalling the description of K-triples of type $\sfA$ in terms of their two polarized boundary vertices in Definition~\ref{def:bdyvert}, we see that the
pair $(\fra_f\circ\phi_q,\fra_f\circ\phi_s)$ describes the ``middle part'' of an internal intersection between 
$\fra_f(\tbx)$ and $\fra_f(\tby)$, where the type of the two polarized boundary vertices determine the type of the left resp.~right part of the intersection in Figure~\ref{fig:intersectA}.
It is clear that this assignation is injective.  Conversely, the middle part of each internal intersection from  
$\fra_f(\tbx)$ to $\fra_f(\tby)$, as described in Figure~\ref{fig:intersectA}, determines a triple
$(G,\phi_q,\phi_s)$ and the properties of the four possible types of its two endpoints show that this triple is indeed a K-triple.  
This is the requested bijection, compare also with~\cite[Lemma~7.2]{QZ17}.  

(b-ii)  Next, assume $\bx=-\bs_i$ for some 
$i\in Q_0(T)$. Thus, $\fra(\bx,s)$ is a tagged arc which belongs to $T$, and we have to show
\[
\Intn^*(\fra(-\bs_i,s), \fra(\by,t))=\bd(\by,t)_{(i,s)}.
\]
This is clear from the construction of $\fra^*$ and  the
Definition~\ref{def:dimv} for $\bd(\by,t)$.

(b-iii) Finally, if $\bx=-\bs_i$ and $\by=-\bs_j$ we have by definition
$e_\uQ((-\bs_i,s), (-\bs_j,t))=0$ and
$\Intn^*(\fra^*(-\bs_i,s), \fra^*(-\bs_j,t))=0$, since
$\fra^*(-\bs_i,s)$, and  $\fra^*(-\bs_j,t)$ are tagged arcs which belong to
the triangulation $T$.
\end{proof}

\begin{Rem}
   (1) In the above proof, the case (b-i) corresponds to the discussion
  of~\cite[Sec.~7.2]{QZ17}.  The authors had to show there, using our
  terminology, that for their modified triangulation, all H-triples of type
  $\sfA$ are indeed K-triples.  Thanks to our preparation in
  Sections~\ref{ssec:triph} and~\ref{ssec:combEg} we can focus directly
  on K-triples.
  In the case of an unpunctured surface, with our construction we
  would have to consider only intersections where the left and right part are
  of type I, thus simplifying considerably the argument in~\cite[Sec.~10.9]{GLFS22}. Note that (as pointed out in the proof), in view of the
  definitions we have obviously
  \begin{align*}
\phantom{X=} & e_\uQ((\bx,s),(\by,t))) -\abs{\sfA_{\uQ^f(\bx, \by)}} \times\wt(s)\cdot\wt(t)\\ 
&=\sum_{(j,i)\in \sfP_\uQ(\bx,\by)} \!\!\!d_2(s_{i+1},(t^\chi)_{j+1}) \;
+\; 2\abs{\Diag_b(\bx,\by)} \cdot d_3(s',t^\chi)\\
&=  \hspace{-1em}\sum_{(j,i)\in \sfP_\bSig(\fra^*(\bx),\fra^*(\by))} \hspace*{-2em}d_2(s_{i+1},(t^\chi)_{j+1}) 
  \;+\; 2\abs{\Diag_b(\fra^*(\bx),\fra^*(\by))} \cdot d_3(s',t^\chi)\\
&=\qquad\Intn^*(\fra^*(\bx,s),\fra^*(\by,t))-\intn(\fra^*(\bx),\fra^*(\by))\times\wt(s)\cdot\wt(t).
\end{align*}
Moreover, in the case of (non-closed) curves on an unpunctered surface, this expression is $0$.

  (2) In Figure~\ref{fig:intersectA} the left-hand part  of the intersection $i$, the types $(\mathrm{II}_l)$ -- $(\mathrm{IV}_l)$ the
  arc $\alp$ of the middle part should be identified with one of the (untagged!)
  sides of the generalized triangle (see Figure~\ref{fig:tagtria})
  which contains the  puncture $\times$.  An analogous remark applies for the right-hand part. 
\end{Rem}

\subsection{Examples} \label{ssec:ExIntersecN}
Let us consider the twice punctured bigon $\bSig=(\Sig, \Ma,\Pu)$ with $\Ma=\{m_1, m_2\}$ and $\Pu=\{p_1, p_2\}$. Note, that we have here rank $n(\bSig)= 6(0-1) + 3(1+2)+2=5$. Moreover, we can find in this case a signature zero triangulation $T$ of $\bSig$, such
that $\uQ=\uQ(T)$  and its boundary fringing $\uQ^f$ 
can be identified with the skewed polarized quivers $\uQ$  and 
$\uQ^f$ from Section~\ref{sssec:skgep}. See Figure~\ref{Fig:ExpIntersect},
where this triangulation is drawn in blue.

\begin{figure}[ht]
\begin{tikzpicture}[scale=0.93]
\coordinate (mk1) at (0,-3);
\coordinate (mk2) at (0,3);
\coordinate (pk3) at ($(mk1) + (61.5:3)$);
\coordinate (pk1) at ($(mk1) + (120:3)$);
\coordinate (tk1)  at  (0,2);
\draw[lightgray, line width =.5cm] (0,0) circle[radius=3.25cm];
\draw[thick, black] (0,0) circle[radius=3cm];
\fill [black] (mk1) node[below]{$m_1$} circle (2pt);  
\fill [black] (mk2) node[above]{$m_2$} circle (2pt); 
\draw[thick, blue] (0,-.5) circle [radius=2.5cm];
\draw[blue, thick, double distance=3pt]
(mk1) -- node[very near end, sloped, below=-4.5pt]{$\bowtie$}  (pk1) node[name=pt1, black]{$\bigtimes$} node[black, above right=1pt]{$\ p_1$};
\draw[blue, thick, double distance=3pt]
(mk1) -- node[very near end, sloped, below=-4.5pt]{$\bowtie$} (pk3) node[name=pt3, black]{$\mathbf{\bigtimes}$} node[black, above left]{$p_3\quad$};
\fill [black] (mk1)  circle (2pt); 
\fill[teal] (tk1) node[above, name = nt1]{ } circle (1.5pt);
\draw[teal, ->] (pt3) -- node[above]{$\gam$}   (pt1) ;
\draw[teal, ->] (pt1) -- node[above]{$\alp$} (nt1);
\draw[teal,->]  (nt1) -- node[above, near end]{$\bet$}(pt3);
\draw[teal, thick, dotted, ->]  (pt1) + (120:0.3) node[below left=-1pt]{$\ \eps_1$}  arc[start angle=50, end angle=310, radius=0.3];
\draw[teal, thick, dotted, ->]  (pt3) + (60:0.3) node[below right=-1pt]{$\eps_3$}  arc[start angle=-220, end angle=-490, radius=0.3];
\fill[green]  (115:3) node[name=pfa]{ } circle (1.5pt);
\fill[green]  (65:3) node[name=pfb]{ } circle (1.5pt);
\draw[green, ->] (nt1) -- node[below, near end]{$\alp'$}(pfa);
\draw[green, ->] (pfb) -- node[below, near start]{$\bet'$}(nt1);
\draw[name path = lam, orange] (-2.9,1.8) .. controls +(-90:5.5) and +(200:2.5) .. node[near end, sloped, above=-1pt]{$\scriptstyle{\rho^{1/2}(\lambda)}$} node[very near end, sloped]{$\bowtie$} (pk3); 
\draw[name path = mu, magenta] (-2.6, 2.1) .. controls +(0:5.3) and +(0:4) .. (0,-1.7) .. controls +(180:4.5) and +(130:4.6) .. node[very near end, sloped, below=-1pt]{$\!\scriptstyle{\rho^{1/2}(\mu)}$}(pk3); 
\fill[red, name intersections={of= lam and mu, by=x1}] (x1) circle (1pt) node[below]{$x$};
\draw[magenta] (0,-4) node{$\rho^{1/2}(\mu)= \rho^{1/2}(\fra(\hbq_{2,2}))=\fra_f(\tau_f(\hbq_{2,2}))$};
\draw[orange] (0, -4.5) node{$\rho^{1/2}(\lam)= \rho^{1/2}(\fra(\hbp_{2,1}))=\fra_f(\tau_f(\hbp_{2,1}))$};
\end{tikzpicture}
\begin{tikzpicture}[scale=0.93]
\coordinate (mk1) at (0,-3);
\coordinate (mk2) at (0,3);
\coordinate (pk3) at ($(mk1) + (61.5:3)$);
\coordinate (pk1) at ($(mk1) + (120:3)$);
\coordinate (tk1)  at  (0,2);
\draw[lightgray, line width =.5cm] (0,0) circle[radius=3.25cm];
\draw[thick, black] (0,0) circle[radius=3cm];
\fill [black] (mk1) node[below]{$m_1$} circle (2pt);  
\fill [black] (mk2) node[above]{$m_2$} circle (2pt); 
\draw[thick, blue] (0,-.5) circle [radius=2.5cm];
\draw[blue, thick, double distance=3pt]
(mk1) -- node[very near end, sloped, below=-4.5pt]{$\bowtie$}  (pk1) node[name=pt1, black]{$\bigtimes$} node[black, above right=1pt]{$\ p_1$};
\draw[blue, thick, double distance=3pt]
(mk1) -- node[very near end, sloped, below=-4.5pt]{$\bowtie$} (pk3) node[name=pt3, black]{$\mathbf{\bigtimes}$} node[black, above left]{$p_3\quad$};
\fill [black] (mk1)  circle (2pt); 
\fill[teal] (tk1) node[above, name = nt1]{ } circle (1.5pt);
\draw[teal, ->] (pt3) -- node[above]{$\gam$}   (pt1) ;
\draw[teal, ->] (pt1) -- node[above]{$\alp$} (nt1);
\draw[teal,->]  (nt1) -- node[above, near end]{$\bet$}(pt3);
\draw[teal, thick, dotted, ->]  (pt1) + (120:0.3) node[below left=-1pt]{$\ \eps_1$}  arc[start angle=50, end angle=310, radius=0.3];
\draw[teal, thick, dotted, ->]  (pt3) + (60:0.3) node[below right=-1pt]{$\eps_3$}  arc[start angle=-220, end angle=-490, radius=0.3];
\fill[green]  (115:3) node[name=pfa]{ } circle (1.5pt);
\fill[green]  (65:3) node[name=pfb]{ } circle (1.5pt);
\draw[green, ->] (nt1) -- node[below, near end]{$\alp'$}(pfa);
\draw[green, ->] (pfb) -- node[below, near start]{$\bet'$}(nt1);
\draw[name path = rho1, magenta] (2.8, 1.8) .. controls +(-90:3.2) and +(0:2) .. (0,-1.8) .. controls +(180:3.1) and +(180:3.1) ..(0,1) .. controls +(0:3.0) and +(-5:3.0)..(0,-1.5) .. controls +(175:3.5) and +(-165:4) .. node[very near end, sloped, below=-2pt]{$\scriptstyle{\rho^{1/2}(\del)}$} (2.5,2.05); 
\draw[name path = rho2, orange] (2.9, 1.9) .. controls +(-90:3.2) and +(0:2) .. (0.1,-1.9) .. controls +(180:3.4) and +(180:3.4) ..(0.1,1.1) .. controls +(0:3.1) and +(-5:3.1)..(0,-1.6) .. controls +(175:3.8) and +(-165:4.1) .. node[very near end, sloped, above=-2pt]{$\scriptstyle{\rho^{1/2}(\rho')}$} (2.6,2.2); 
\fill[red, name intersections={of= rho1 and rho2, by={x1, x2}}] 
(x1) circle (1pt) node[below=1pt]{$\scriptstyle{\ x_1}$}
(x2) circle (1pt) node[above=1pt]{$\scriptstyle{x_2}$};
\draw[magenta] (0,-4) node{$\rho^{1/2}(\del)=\rho^{1/2}(\fra(\hbr'_{3,0}))=\fra_f(\tau_f(\hbr'_{3,0}))$};
\draw[orange] (0,-4.5) node{$\rho^{1/2}(\del')=\rho^{1/2}(\fra(\hbr'_{3,0}))=\fra_f(\tau_f(\hbr'_{3,0})) $};
\end{tikzpicture}
\caption{Calculating intersection numbers}
\label{Fig:ExpIntersect}
\end{figure}

The left-hand side of Figure~\ref{Fig:ExpIntersect} is closely related to Section~\ref{sssec:ExCombE-inv}~(1):
Recall, that 
\begin{align*}
\tau_f(\hbq_{2,2})&=\tau_f(\II^{-1}_{(2,+)}\bet^{-1}\eps_3^{-1}\gam^{-1}\eps_1^{-1}\gam\II_{(3,-)})=
\II^{-1}_{(4,-)}\alp'\bet^{-1}\eps_3^{-1}\gam^{-1}\eps_1^{-1}\gam\II_{(3,-)} \text{ and}\\ 
\tau_f(\hbp_{2,1})&=\tau_f(\II^{-1}_{(2,+)}\alp\eps_1\gam\II_{(3,-)})= \II^{-1}_{(4,-)}\alp'\alp\eps_1\gam\II_{(3,-)}.
\end{align*}
The curves $\rho^{1/2}(\lam)$ and $\rho^{1/2}(\mu)$ have exactly  one internal intersection point $x$, which corresponds to the unique element of the set
$\sfA_{\uQ^f}(\hbq_{2,2},\hbp_{2,1})=\{[(\uH^{(4)},\phi_q^{(4')},\phi_s^{(4')})]\}$.  
Similarly, $\rho^{1/2}(\lam)$ and $\rho^{1/2}(\mu)$ share 
(like $\lam$ and $\mu$) exactly the puncture  $p_1$ as their respective endpoints. In other words,
we have $\sfP_\bSig(\lam, \mu)=\{(1,1)\}$.
This corresponds to the fact that 
$\sfP_\uQ(\hbp_{2,1}, \hbq_{2,2})=\{(1,1)\}$. 
In view of  Definition~\ref{def:mIntersectN} we can calculate,
with $s^+:=(1,1)$ and $s^-:=(1,-1)$, as follows
\[
\Intn^*((\lam, s^-), (\mu, s^+))= \abs{\{x\}} + d_2(s^{-}_2, (s^+)^\chi_2)=1+1=2,
\]
in accordance with $e_\uQ((\hbp_{2,1}, s^-), (\hbq_{2,2},s^+))=2$,
as calculated in Section~\ref{sssec:ExCombE-inv}~(1).

The right-hand side of Figure~\ref{Fig:ExpIntersect} is closely related to Section~\ref{sssec:ExCombE-inv}~(2):
Recall, that 
\[
\tau_f(\hbr'_{3,0}) = \tau_f(\II^{-1}_{(3,+)}\eps_3^{-1}\gam^{-1}\eps_1^{-1}\II_{(1,+)})=
\II^{-1}_{(4,-)}(\bet')^{-1}\bet^{-1}\eps_1^{-1}\gam^{-1}\eps_1^{-1}\gam(\eps_3^{-1}\gam^{-1})\alp^{-1}\bet'\II_{(4,-)}.
\]
For the curves $\rho^{1/2}(\del)$ and 
$\rho^{1/2}(\del')\sim\rho^{1/2}(\del)$ intersect precisely in
the set $\{x_1, x_1\}$, which has minimal cardinality.
These two elements, $x_1$ and $x_2$ correspond naturally
to the two elements of 
$\sfA_{\uQ^f}(\tau_f(\hbr'_{3,0}),\tau_f(\hbr'_{3,0}))$. 
Since $\del$ and $\del'$ both start and end in $m_1$, we have
$\sfP_{\bSig}(\del, \del')=\emptyset$. Similarly, since
$\hbr'_{3,0}$ is of type $(u,u)$ we have $\sfP_\uQ(\hbr'_{3,0},\hbr'_{3,0})=\emptyset$. Accordingly, in both cases the only possible decoration is $s^+=(+1,+1)$, and thus we find
\[
\Intn^*((\rho, s^+), (\rho', s^+))= \abs{\{x_1, x_2\}} +0=2=
e_\uQ((\hbr'_{3,0}, s^+), (\hbr'_{3,0}, s^+)).
\]

\subsection{Arcs} \label{ssec:arcs}
Following~\cite[Def.~7.1]{FST08} an (oriented) arc, as used for tagged triangulation is, in our notation, an element 
$\gam\in\tcAC(\bSu)$, which has up to homotopy, no self-intersections, and which does not cut out a once punctured monogon.  
For $\type(\gam)=(p,p)$ it is not obvious if such an arc
belongs actually to $\cAC(\bSu)$, as pointed out in~\cite[Rem.~5.23~(3)]{AP17}.
In other words, it is not a priory clear if the notion of tagged arcs in the sense of~\cite{FST08} is the same as ours. 

It seems, that Qiu and Zhou assumed in~\cite[Rem.~3.5]{QZ17} that this is trivial.  We will address this issue here with the help of Lemma~\ref{lem:tadm}.

\begin{Prop}
  \label{Lem:arcs}
  Let $\gam\in\tcAC(\bSu)$ without interior self-intersections up to homotopy.
  If moreover $\gam$ does not cut out a once-punctured monogon, then in fact
  $\gam\in\cAC(\bSu)$.  
\end{Prop}

\begin{proof}
We may assume $\type(\gam)=(p,p)$, since otherwise the claim is trivial.  
Let $T$ be a signature zero tagged triangulation of $\bSu$, and
$\tAdm_{(p,p)}(\uQ(T)):=\{\bx\in\tAdm(\uQ(T))\mid\type(\bx)=(p,p)\}$. From the
discussion in Section~\ref{ssec:orb} it is clear 
that $\fra=\fra_T$ can be modified to a bijection
\[
  \tilde{\fra}_T\df\tAdm_{(p,p)}(\uQ(T))\ra\tcAC_b(\bSu),
\]  
where we used the notation from Definition~\ref{def:tcLC}.  Using the arguments
from the proof of Theorem~\ref{thm:intersec},  it is clear that
$\gam=\tilde{\fra}_T(\bx)$ for some $\bx\in\tAdm_{(p,p)}(\uQ(T))$, which fulfils
the hypothesis of Lemma~\ref{lem:tadm}. Thus, $\bx\in\Adm(\uQ(T))$ and
$\gam=\tilde{\fra}_T(\bx)=\fra_T(\bx)\in\cAC(\bSu)$.
\end{proof}

\subsection{Dual shear coordinates with respect to tagged triangulations of signature zero}
Let us fix in this section the following conventions and notations: Let
$T= (\tau_i)_{i=1,2,\ldots,n-2\abs{\Pu}}\cup (\tau_p)_{p\in\Pu}\cup(\tau'_p)_{p\in\Pu}$
be a signature zero tagged triangulation of a marked surface $\bSu$.
For $i=1,2,\ldots, n-2\abs{\Pu}$ we agree that the 
tagged arcs $\tau_i$ are of the form $\tau_i=(\gam_i, (1,1))$ for certain curves $\gam_i\in\cAC(\bSu)$ with 
$\type(\gam_i)=(u,u)$ for all $i$.
We agree moreover that $\tau_p=(\gam_p,c_p)$ and $\tau'_p=(\gam_p,c'_p)$ and
$\gam_p(1)=p$ for all $p\in\Pu$.  Thus, $c_p=(1,1)$ and $c'_p=(1,-1)$ for all 
$p\in\Pu$. In particular, $c_{p,1}=1=c'_{p,1}$ and $c_{p,2}=1=-c'_{p,2}$.
We have then  $\uQ(T)_0^{\ord}=\{1,2, \ldots, n-2\abs{\Pu}\}$ and
$\uQ(T)_0^{\spe}=\Pu$.
In particular,  we have a natural bijection from the arcs of $T$ to the
set of vertices  $\tQ_0$, which sends for $i=1,2,\ldots, n-2\abs{\Pu}$ the arc
$(\gam_i,(1,1))$ to the vertex $(i,o)$, and which sends for each $p\in\Pu$
and $\rho\in\{+,-\}$  the arc $(\gam_p, (1,\rho\cdot 1))$  to the vertex
$(p,\rho)$.  We also recall the notation
$\gam'_p:=\sig(\tau'_p)=\bar{\gam}_p$ for all $p\in\Pu$. Trivially,
$\widebar{\gam}_i=\gam_i$ for all $i=1,2,\ldots, n-2\abs{\Pu}$. 
We have then the ideal triangulation
$T^{\circ}=(\gam_i)_{i=1,2,\ldots, n-2\abs{\Pu}}\cup (\gam_p)_{p\in\Pu}\cup
  (\gam'_p)_{p\in\Pu}$,
see Definition~\ref{def-IdealT}.

In order to introduce the dual shear coordinates $\Sh_T(\lam, l)\in\ZZ^T$
for a marked curve $(\lam, l)\in\cLC^*(\bSu)$ with respect to $T$,
we slightly expand and adapt the discussion by  Fomin and Thurston
in~\cite[Ch.~13]{FT18}.
Note however, that in loc.~cit. shear coordinates of
$\cX$-laminations with respect to an \emph{arbitrary} tagged
triangulation were introduced, see Remark \ref{rem:these-coords-are-dual-to-FST} and Subsection \ref{subsec:dual-shear-coords-arb-triangs} below.

In $T^\circ$ each arc $\del$ of type $(u,u)$ is the diagonal of a unique quadrilateral $F(\del)$. See Figure~\ref{fig:idealT}. The four sides of $F(\del)$ are possibly not  pairwise
different, and some of them can be boundary segments.  We illustrate some
of the more degenerate cases in Figure~\ref{fig:4gon}.
\begin{figure}
\begin{tikzpicture}[scale=.8] 
\begin{scope}
    \clip (-2,2) circle (2.5cm);  
 \draw[lightgray, line width=1.0cm, line cap=round] (0,-0.1) +(-120:.5cm)
 arc[start angle=60, end angle=120, radius=3.5];
  \draw[name path = bottom1, thick] (0,0)
  arc[start angle=60, end angle=120, radius=4] node[near end, sloped]{$>$};
  \path [name path= bottom2]  (-2,0) -- (-2,3);
\fill [black, name intersections={of= bottom1 and bottom2, by=mb1}]
   (mb1) circle (2pt) node[below]{$m_1$}; 
\draw[lightgray, line width=1.0cm, line cap=round] (-4,4.1) +(60:.5cm)
    arc[start angle=240, end angle=300, radius=3.5];
  \draw[name path = top1, thick] (-4,4)
  arc[start angle=240, end angle=300, radius=4] node[near end, sloped]{$<$};
  \path [name path= top2]  (-2,4) -- (-2,2);
\fill [black, name intersections={of= top1 and top2, by=mb2}]
(mb2)  node[above]{$m_2$};
\draw[blue, thick]
(mb1) .. controls +(170:2.8) and +(190:2.8) .. node[sloped, near start]{$<$}
node[near end, name=a1]{$\cdot$} node[near end, left]{$\alp$}  (mb2);
\draw[blue, thick]
(mb1) .. controls +(10:2.8) and +(-10:2.8) .. node[sloped, near start]{$>$}
node[near end, name=a2]{$\cdot$} node[near end, right]{$\bet$} (mb2);
\draw[blue, thick, double distance=3pt]
(mb1) -- node[very near end, sloped, below=-5pt]{$\bowtie$} node[left=-2pt]{$\tau_p\phantom{'}$} node[right]{$\tau_p'$} +(90:1.4) node[name=pt2, black, above=-3pt]{$\times$} node[black, above=6pt]{$p$};
\fill[black] (mb1) circle (2pt);
\fill[black] (mb2) circle (2pt);
\fill[teal] (a1) circle (1.5pt);
\fill[teal] (a2) circle (1.5pt);
\draw[teal, ->] (a1) -- node[above=-2.5pt]{$\eta$} (a2);
\draw[teal, ->] (a2) -- node[below right=-3pt]{$\tht$} (pt2);
\draw[teal, ->] (pt2) -- node[below left=-3pt]{$\zet$}  (a1);
\draw[teal, thick, dotted, ->] (pt2) +(190:0.15)  arc[start angle=110, end angle=410, radius=.4];
\path (pt2) ++(35:-20pt) node[teal]{$\eps_p$};
\end{scope}
\node[scale=2] at (1,2){$\mapsto$};
\begin{scope}
\clip (4,2) circle (2.5cm);    
\node[scale=2] at (1,2) {$\mapsto$};
 \draw[lightgray, line width=1.0cm, line cap=round] (6,-0.1) +(-120:.5cm)
    arc[start angle=60, end angle=120, radius=3.5];
  \draw[name path = bottom3, thick] (6,0)
  arc[start angle=60, end angle=120, radius=4] node[near end, sloped]{$>$};
  \path [name path= bottom4]  (4,0) -- (4,3);
\fill [black, name intersections={of= bottom3 and bottom4, by=rb1}]
   (rb1) circle (2pt) node[below]{$m_1$}; 
\draw[lightgray, line width=1.0cm, line cap=round] (2,4.1) +(60:.5cm)
    arc[start angle=240, end angle=300, radius=3.5];
  \draw[name path = top3, thick] (2,4)
  arc[start angle=240, end angle=300, radius=4] node[near end, sloped]{$<$};
  \path [name path= top4]  (4,4) -- (4,2);
\fill [black, name intersections={of= top3 and top4, by=rb2}]
(rb2)  node[above]{$m_2$};
\draw[blue, thick]
(rb1) .. controls +(170:2.8) and +(190:2.8) .. node[sloped, near start]{$<$} node[near end, left]{$\alp$} (rb2);
\draw[blue, thick]
(rb1) .. controls +(10:2.8) and +(-10:2.8) .. node[sloped, near start]{$>$} node[near end, right]{$\bet$} (rb2);
\draw[blue, thick]
(rb1) -- node[near end, left=-2pt]{$\gam_p$} +(90:1.4) node[name=pt2, black, above=-3pt]{$\times$} node[black, above=6pt]{$p$};
\draw[blue, thick] (rb1) .. controls +(120:3.7) and +(60:3.7) .. node[very near end, right]{$\gam_p'$} (rb1);
\fill[black] (rb1) circle (2pt);
\fill[black] (rb2) circle (2pt);
\end{scope}
\coordinate[label=left: $m_1$] (M1l)  at (8,2);
\coordinate[label=right: $m_1$] (M1r) at  ($(M1l)+(2,0)$);
\coordinate[label=above: $m_2$] (M2)  at  ($(M1l)+(1, 1.5)$);
\coordinate[label=below: $p$] (PT)  at  ($(M1l)+(1, -1.5)$);
\foreach \point in {M1l, M1r, M2, PT} \fill [black] (\point) circle (2pt);
\draw[blue, thick] (M1l)-- node[fill=white]{$\scriptstyle\gam'_p$} node[minimum size=0pt, name=c1]{} (M1r) -- node[above right]{$\bet$} node[minimum size=0pt, name=b1]{} (M2) -- node[above left]{$\alp$} node[minimum size=0pt, name=a1]{} (M1l) -- node[below left]{$\gam_p$} node[minimum size=0pt, name=c2]{} (PT) -- node[below right]{$\gam_p$} node[minimum size=0pt, name=c3]{} (M1r);
\draw[teal, ->] (a1) -- node[above=-1pt]{$\scriptstyle\eta$} (b1);
\draw[teal, ->] (b1) -- node[below right=-3.5pt]{$\scriptstyle\tht$}(c1);
\draw[teal, ->] (c1) -- node[below left=-3.5pt]{$\scriptstyle\zet$}(a1);
\draw[teal, ->] (c1) -- node[right=-3pt]{$\scriptstyle\sqrt{\eps}_p$}(c3);
\draw[teal, ->] (c2) -- node[left=-2pt]{$\scriptstyle\sqrt{\eps}_p$}(c1);

\node[scale=1.5] at (-2,-1.3){$T$};
\node[scale=1.5] at (4,-1.3){$T^{\circ}$};
\node[scale=1.5] at (9.2,-1.3){$F(\gam'_p)$};
\end{tikzpicture}\\[.2cm]
\begin{tikzpicture}[scale=.8]
\begin{scope}
\clip (-2,2) circle (2.5cm);
  \draw[lightgray, line width=1.0cm, line cap=round] (0,-0.1) +(-120:.5cm)
    arc[start angle=60, end angle=120, radius=3.5];
  \draw[name path = bottom1, thick] (0,0)
  arc[start angle=60, end angle=120, radius=4] node[near end, sloped]{$>$};
  \path [name path= bottom2]  (-2,0) -- (-2,3);
\fill [black, name intersections={of= bottom1 and bottom2, by=mb1}]
   (mb1) circle (2pt) node[below]{$m$}; 
\draw[blue, thick]
(mb1) .. controls +(170:1.8) and +(180:4) .. node[sloped, near end]{$>$} (-2,4) node[blue, above=-2.5pt]{$\bet$} node[teal, name=b1]{$\cdot$}
   .. controls +(0:4) and +(10:1.8) .. node[sloped]{$<$} (mb1);
   \draw[blue, thick, double distance=3pt]
   (mb1) -- node[left=2pt]{$\tau'_p$}  node[above right=-3pt]{$\tau_p$} node[very near end, sloped, below=-4.2pt]{$\bowtie$}  +(120:2) node[name=pt1, black, above left=-3pt]{$\times$} node[black, above right=3pt]{$p$} node[teal, left=1pt]{$\eps_p$};
\draw[blue, thick, double distance=3pt]
(mb1) -- node[above left=-2.5pt]{$\tau_q$} node[right=2pt]{$\tau'_q$} node[very near end, sloped, below=-4.2pt]{$\bowtie$}  +(60:2) node[name=pt2, black, above right=-3pt]{$\times$} node[black, above left=3pt]{$q$} node[teal, above right=1pt]{$\eps_q$};
\fill (mb1) circle (2pt);
\draw[teal, ->] (pt2) -- node[below=-2pt]{$\eta$}  (pt1);
\draw[teal, ->] (pt1) -- node[above left=-2pt]{$\zet$} (b1);
\draw[teal, ->] (b1) -- node[above right=-1pt]{$\tht$}  (pt2);
\draw[teal, thick, dotted, ->]  (pt1) + (240:0.3)  arc[start angle=310, end angle=50, radius=0.3];
\draw[teal, thick, dotted, ->]  (pt2) + (60:0.3)  arc[start angle=-220, end angle=-490, radius=0.3];
\end{scope}
\node[scale=2] at (1,2){$\mapsto$};
\begin{scope}
\clip (4,2) circle (2.5cm);
  \draw[lightgray, line width=1.0cm, line cap=round] (6,-0.1) +(-120:.5cm)
    arc[start angle=60, end angle=120, radius=3.5];
  \draw[name path = bottom3, thick] (6,0)
  arc[start angle=60, end angle=120, radius=4] node[near end, sloped]{$>$};
  \path [name path= bottom4]  (4,0) -- (4,3);
\fill [black, name intersections={of= bottom3 and bottom4, by=rb1}]
   (rb1) circle (2pt) node[below]{$m$}; 
\draw[blue, thick]
(rb1) .. controls +(170:1.8) and +(180:4) .. node[sloped, near end]{$>$} (4,4) node[blue, below]{$\bet$}
   .. controls +(0:4) and +(10:1.8) .. node[sloped]{$<$} (rb1);
   \draw[blue, thick]
   (rb1) -- node[near end, above right=-4.5pt]{$\gam_p$}   +(120:2) node[name=pt1, black, above left=-3pt]{$\times$} node[black, above left=3pt]{$p$};
\draw[blue, thick]  (rb1) .. controls +(145:4.2) and +(95:4.2) .. node[near end, above right=-4pt]{$\gam'_p$}(rb1); 
\draw[blue, thick]
(rb1) -- node[near end, above left=-4.5pt]{$\gam_q$}    +(60:2) node[name=pt2, black, above right=-3pt]{$\times$} node[black, above right=3pt]{$q$};
\draw[blue, thick]  (rb1) .. controls +(85:4.2) and +(35:4.2) .. node[near start, above left=-4pt]{$\gam'_q$}(rb1); 
\fill (mb1) circle (2pt);
\end{scope}
\coordinate[label=left: $m$] (M1l)  at (7.5,2);
\coordinate[label=right: $m$] (M1r) at  ($(M1l)+(2,0)$);
\coordinate[label=above: $m$] (M2)  at  ($(M1l)+(1, 1.5)$);
\coordinate[label=below: $p$] (PT)  at  ($(M1l)+(1, -1.5)$);
\foreach \point in {M1l, M1r, M2, PT} \fill [black] (\point) circle (2pt);
\draw[blue, thick] (M1l)-- node[fill=white]{$\scriptstyle\gam'_p$} node[minimum size=0pt, name=cp1]{} (M1r) -- node[above right]{$\gam'_q$} node[minimum size=0pt, name=cq1]{} (M2) -- node[above left]{$\bet$} node[minimum size=0pt, name=b2]{} (M1l) -- node[below left]{$\gam_p$} node[minimum size=0pt, name=cp2]{} (PT) -- node[below right]{$\gam_p$} node[minimum size=0pt, name=cp3]{} (M1r);
\draw[teal, ->] (b2) -- node[above=-1pt]{$\scriptstyle\tht$} (cq1);
\draw[teal, ->] (cq1) -- node[below right=-3.5pt]{$\scriptstyle\eta$}(cp1);
\draw[teal, ->] (cp1) -- node[below left=-3.5pt]{$\scriptstyle\zet$}(b2);
\draw[teal, ->] (cp1) -- node[right=-3pt]{$\scriptstyle\sqrt{\eps}_p$}(cp3);
\draw[teal, ->] (cp2) -- node[left=-2pt]{$\scriptstyle\sqrt{\eps}_p$}(cp1);
\coordinate[label=left: $m$] (M1l)  at (11.5,2);
\coordinate[label=right: $m$] (M1r) at  ($(M1l)+(2,0)$);
\coordinate[label=above: $m$] (M2)  at  ($(M1l)+(1, 1.5)$);
\coordinate[label=below: $q$] (PT)  at  ($(M1l)+(1, -1.5)$);
\foreach \point in {M1l, M1r, M2, PT} \fill [black] (\point) circle (2pt);
\draw[blue, thick] (M1l)-- node[fill=white]{$\scriptstyle\gam'_q$} node[minimum size=0pt, name=cq2]{} (M1r) -- node[above right]{$\bet$} node[minimum size=0pt, name=b3]{} (M2) -- node[above left]{$\gam'_p$} node[minimum size=0pt, name=cp4]{} (M1l) -- node[below left]{$\gam_q$} node[minimum size=0pt, name=cq3]{} (PT) -- node[below right]{$\gam_q$} node[minimum size=0pt, name=cq4]{} (M1r);
\draw[teal, ->] (cp4) -- node[above=-1pt]{$\scriptstyle\zet$} (b3);
\draw[teal, ->] (b3) -- node[below right=-3.5pt]{$\scriptstyle\tht$}(cq2);
\draw[teal, ->] (cq2) -- node[below left=-3.5pt]{$\scriptstyle\eta$}(cp4);
\draw[teal, ->] (cq2) -- node[right=-3pt]{$\scriptstyle\sqrt{\eps}_q$}(cq4);
\draw[teal, ->] (cq3) -- node[left=-2pt]{$\scriptstyle\sqrt{\eps}_q$}(cq2);


\node[scale=1.5] at (-2,-1.3){$T$};
\node[scale=1.5] at (4,-1.3){$T^{\circ}$};
\node[scale=1.5] at (8.5,-1.3){$F(\gam'_p)$};
\node[scale=1.5] at (12.5,-1.3){$F(\gam'_q)$};
\end{tikzpicture}
\caption{Degenerate quadrilaterals}
\label{fig:4gon}
\end{figure}
\begin{Def} 
This allows us to introduce for each arc $\del =\bar{\gam}\in T^{\circ}$ 
the \emph{contribution of type $\sfA$} of $\lam$ to the component
$(\gam, c)\in T$  of $\Sh_T(\lam,l)$ as follows: We say that
an intersection of $\rho^{1/2}(\lam)$ with $\del$ belongs to the set
$A^\pm_{T,\del}(\lam)$,
if the intersection occurs within $F(\del)$ according to one of
the first two Diagrams  in Figure~\ref{fig:shearAD}.

Next, we define for each $r\in\Pu$ the (unweighted)
contribution of type $\sfD$ of $\lam$ to the component $\tau_p$ and $\tau'_p$ of the shear coordinates $\Sh_T(\lam, l)$ in terms of the
sets $D^\pm_{T,\gam'_r}(\lam)$ according to third diagram in 
Figure~\ref{fig:shearAD}.  We have here $0\in D^\pm_{T,\gam'_r}(\lam)$
iff $1\in D^\pm_{T,\gam'_r}(\lam^{-1})$.
In particular $D^\pm_{T,\gam'_p}(\lam)=\emptyset$ if $p\not\in\{\gam(0),\gam(1)\}$.
\begin{figure}[ht]
\begin{tikzpicture}
\coordinate (M1l)  at (0,2);
\coordinate (M1r) at  ($(M1l)+(2,0)$);
\coordinate (M2)  at  ($(M1l)+(1, 1.3)$);
\coordinate (PT)  at  ($(M1l)+(1, -1.3)$);
\foreach \point in {M1l, M1r, M2, PT} \fill [black] (\point) circle (2pt);
\draw[blue, very thick, name path=S] (M2) -- node[pos=0.65, minimum size=0pt, name=nw]{} (M1l) -- node[minimum size=0pt, name=cl]{} node[fill=white, near start]{$\scriptstyle\del$} (M1r) -- node[pos=0.4, minimum size=0pt, name=se]{}(PT);
\draw[blue] (M2) -- (M1r); \draw[blue] (M1l) -- (PT);
\draw[violet, thick, name path=RD] ($(M1l)+(-.5,1)$) .. controls +(350:1) and +(110:1) .. node[scale=0.8, very near start, above]{$\rho^{1/2}(\lam)$} ($(M1r)+(.5,-1.5)$);
\fill [red, name intersections={of=RD and S}]
(intersection-2) circle (2pt) node[above]{$x$};
\node at ($(PT)+(0,-1)$){$x\in A^+_{T,\del}(\lam)$};
\draw[teal, thick, ->] (cl) -- (nw);
\draw[teal, thick, ->] (cl) -- (se);
\coordinate (M1l)  at (4,2);
\coordinate (M1r) at  ($(M1l)+(2,0)$);
\coordinate (M2)  at  ($(M1l)+(1, 1.3)$);
\coordinate (PT)  at  ($(M1l)+(1, -1.3)$);
\foreach \point in {M1l, M1r, M2, PT} \fill [black] (\point) circle (2pt);
\draw[blue, very thick, name path=Z]
(PT) -- node[pos=0.6, minimum size=0pt, name=sw]{}
(M1l) -- node[fill=white, near end]{$\scriptstyle\del$} node[pos=0.55, minimum size=0pt, name=cr]{}
(M1r) -- node[pos=0.35, minimum size=0pt, name=ne]{}(M2);
\draw[blue] (M2) -- (M1l); \draw[blue] (M1r) -- (PT);
\draw[violet, thick, name path=LU] ($(M1l)+(-.5,-1.5)$) .. controls +(70:1) and +(200:1) ..  node[scale=0.8, very near end, above]{$\rho^{1/2}(\lam)$} ($(M1r)+(.5,1)$);
\fill [red, name intersections={of=LU and Z}]
(intersection-2) circle (2pt) node[above]{$x$};
\node at ($(PT)+(0,-1)$){$x\in A^-_{T,\del}(\lam)$};
\draw[teal, thick, ->] (sw) -- (cr);
\draw[teal, thick, ->] (ne) -- (cr);
\coordinate (M1l)  at (8,2);
\coordinate (M1r) at  ($(M1l)+(2,0)$);
\coordinate (M2)  at  ($(M1l)+(1, 1.3)$);
\coordinate (PT)  at  ($(M1l)+(1, -1.3)$);
\foreach \point in {M1l, M1r, M2, PT} \fill [black] (\point) circle (2pt);
\node[below] at (PT){$r$};
\draw[blue, thick] %
(M1l) -- node[pos=0.6, minimum size=0pt, name=nw3]{}%
(M2) --  node[pos=0.4, minimum size=0pt, name=ne3]{}%
(M1r) -- node[fill=white, name=ce3]{$\scriptstyle\gam'_r$}%
(M1l) -- node[below left=-3pt]{$\gam_r$}%
(PT) --  node[below right=-2pt]{$\gam_r$} (M1r);
\draw[teal, thick, ->] (ce3) -- (nw3);
\draw[teal, thick, ->] (ne3) -- (ce3);
\draw[violet, thick] (PT) ..controls +(100:1) and +(330:1).. node[near start, sloped]{$>$} node[scale=0.8, near end, above=8pt]{$\rho^{1/2}(\lam)$}
($(M1l)+(-.5,1.1)$);
\draw[magenta, thick] (PT) ..controls +(80:1) and +(210:1).. node[near start, sloped]{$>$} node[scale=0.8, near end, above=8pt]{$\rho^{1/2}(\mu)$} ($(M1l)+(2.5,1.1)$);
\node at ($(PT)+(0,-.8)$){$1\in D^+_{T,\gam'_r}(\lam)$};
\node at ($(PT)+(0,-1.4)$){$0\in D^-_{T,\gam'_r}(\mu)$};
\end{tikzpicture}
\caption{Definition: Contributions to $A^\pm_{T,\del}$ resp.~to $D^{\pm}_{T,\gam'_r}$}
\label{fig:shearAD}
\end{figure}
\end{Def}
For the case $\del=\gam'_p$ with $p\in\Pu$ we
illustrate the contribution of $\lam$ to $A^\pm_{T,\gam'_p}$ in
Figure~\ref{fig:shear-fold}.  In Figure~\ref{fig:shearD} we expand the
definition of the sets $D^\pm_{T,\gam'_r}$. 

\begin{figure}
\begin{tikzpicture}
\begin{scope} 
  \clip (-2,2) ellipse (2.8cm and 2.2cm);  
 \draw[lightgray, line width=1.0cm, line cap=round] (0,0) +(-120:.5cm)
 arc[start angle=60, end angle=120, radius=3.5];
  \draw[name path = bottom1, thick] (0,0)
  arc[start angle=60, end angle=120, radius=4] node[near end, sloped]{$>$};
  \path [name path= bottom2]  (-2,0) -- (-2,3);
\fill [black, name intersections={of= bottom1 and bottom2, by=mb1}]
   (mb1) circle (2pt) node[below]{$m_1$}; 
\draw[lightgray, line width=1.0cm, line cap=round] (-4,4) +(60:.5cm)
    arc[start angle=240, end angle=300, radius=3.5];
  \draw[name path = top1, thick] (-4,4)
  arc[start angle=240, end angle=300, radius=4] node[near end, sloped]{$<$};
  \path [name path= top2]  (-2,4) -- (-2,2);
\fill [black, name intersections={of= top1 and top2, by=mb2}]
(mb2)  node[above]{$m_2$};
\draw[blue, thick]
(mb1) .. controls +(170:2.8) and +(190:2.8) .. node[sloped, near start]{$<$} node[near end, left]{$\alp$} (mb2);
\draw[blue, thick]
(mb1) .. controls +(10:2.8) and +(-10:2.8) .. node[sloped, near start]{$>$} node[near end, right]{$\bet$} (mb2);
\draw[blue, thick, double distance=3pt]
(mb1) -- node[very near end, sloped, below=-5pt]{$\bowtie$} node[left=-3pt]{$\tau_p\phantom{'}$} node[right]{$\tau'_p$} +(90:1.4) node[name=pt2, black, above=-4pt]{$\times$} node[black, above=6pt]{$p$};
\draw[gray, name path=ALPP] (mb1) .. controls +(130:3.7) and +(50:3.7) .. node[very near end, right]{$\gam'_p$} (mb1);
\draw[violet, thick, name path=LAM] (-4.5,3.2) .. controls +(0:4.5) and +(0:4.5) .. node[near start, sloped]{$>$} node[scale=.7, very near end, above]{$\rho^{-1/2}(\lam)$} (-4.5,1.2);
\fill[red, name intersections={of= ALPP and LAM}]
(intersection-2) circle (2pt) node[right]{$x$};
\fill[black] (mb1) circle (2pt);
\fill[black] (mb2) circle (2pt);
\end{scope}
\node[scale=2] at (.8,2) {$\leadsto$};
\coordinate[label=left: $m_1$] (M1l)  at (2.2 ,2);
\coordinate[label=right: $m_1$] (M1r) at  ($(M1l)+(2,0)$);
\coordinate[label=above: $m_2$] (M2)  at  ($(M1l)+(1, 1.5)$);
\coordinate[label=below: $p$] (PT)  at  ($(M1l)+(1, -1.5)$);
\foreach \point in {M1l, M1r, M2, PT} \fill [black] (\point) circle (2pt);
\draw[blue, thick, name path=ALP2] (M2) -- node[above left=1pt]{$\alp$} (M1l)-- node[above left]{$\gam'_p$} (M1r) -- node[below right]{$\gam_p$} (PT);
\draw[blue] (M1r) -- node[above right]{$\bet$} (M2);
\draw[blue] (M1l) -- node[below left]{$\gam_p$} (PT);
\draw[violet, thick, name path=LAM2] ($(M1l)+(.6,1.5)$) .. controls +(295:4.7) and +(275:4.4) .. node[near start, sloped]{$>$} node[below]{$\rho^{1/2}(\lam)$} ($(M1l)+(0.2,0.8)$);
\fill[red, name intersections={of= ALP2 and LAM2}]
(intersection-3) circle (2pt) node[above right]{$x$};
\node at ($(PT)+(0,-1.8)$){$x\in A^+_{T,\gam'_p}(\lam)$};
\begin{scope} 
  \clip (7.5,2) ellipse (2.8cm and 2.2cm);  
 \draw[lightgray, line width=1.0cm, line cap=round] (9.5,0) +(-120:.5cm)
 arc[start angle=60, end angle=120, radius=3.5];
  \draw[name path = bottom1, thick] (9.5,0)
  arc[start angle=60, end angle=120, radius=4] node[near end, sloped]{$>$};
  \path [name path= bottom2]  (7.5,0) -- (7.5,3);
\fill [black, name intersections={of= bottom1 and bottom2, by=mb1}]
   (mb1) circle (2pt) node[below]{$m_1$}; 
\draw[lightgray, line width=1.0cm, line cap=round] (5.5,4) +(60:.5cm)
    arc[start angle=240, end angle=300, radius=3.5];
  \draw[name path = top1, thick] (5.5,4)
  arc[start angle=240, end angle=300, radius=4] node[near end, sloped]{$<$};
  \path [name path= top2]  (7.5,4) -- (7.5,2);
\fill [black, name intersections={of= top1 and top2, by=mb2}]
(mb2)  node[above]{$m_2$};
\draw[blue, thick]
(mb1) .. controls +(170:2.8) and +(190:2.8) .. node[sloped, near start]{$<$} node[near end, left]{$\alp$} (mb2);
\draw[blue, thick]
(mb1) .. controls +(10:2.8) and +(-10:2.8) .. node[sloped, near start]{$>$} node[near end, right]{$\bet$} (mb2);
\draw[blue, thick, double distance=3pt]
(mb1) -- node[very near end, sloped, below=-5pt]{$\bowtie$} node[left=-4.5pt]{$\tau_p\phantom{'}$} node[right=-2.5pt]{$\tau'_p$} +(90:1.4) node[name=pt2, black, above=-3pt]{$\times$} node[black, above=6pt]{$p$};
\draw[gray, name path=ALPP] (mb1) .. controls +(130:3.7) and +(50:3.7) .. node[very near end, right]{$\gam'_p$} (mb1);
\draw[violet, thick, name path=LAM] (10.5,3) .. controls +(200:2.5) and +(0:2.5) .. node[near start, sloped]{$<$} node[scale=.7, near end, sloped, above left]{$\rho^{-1/2}(\lam)$} (5.5,1.2);
\fill[black] (mb1) circle (2pt);
\fill[black] (mb2) circle (2pt);
\end{scope}
\node at ($(mb1)+(0,-1.5)$) {no contribution};  
\end{tikzpicture}

\begin{tikzpicture}
\begin{scope}
\clip (-2,2.2) ellipse (2.8cm and 2.3cm);
  \draw[lightgray, line width=1.0cm, line cap=round] (0,0) +(-120:.5cm)
    arc[start angle=60, end angle=120, radius=3.5];
  \draw[name path = bottom1, thick] (0,0)
  arc[start angle=60, end angle=120, radius=4] node[near end, sloped]{$>$};
  \path [name path= bottom2]  (-2,0) -- (-2,3);
\fill [black, name intersections={of= bottom1 and bottom2, by=mb1}]
   (mb1) circle (2pt) node[below]{$m$}; 
\draw[blue, thick]
(mb1) .. controls +(170:1.8) and +(180:4) .. node[sloped, near end]{$>$} (-2,4) node[blue, below]{$\bet$}
   .. controls +(0:4) and +(10:1.8) .. node[sloped]{$<$} (mb1);
   \draw[blue, thick, double distance=3pt]
   (mb1) -- node[left=3pt]{$\tau'_p$}  node[above right=-2.5pt]{$\tau_p$} node[very near end, sloped, below=-4.2pt]{$\bowtie$}  +(120:2) node[name=pt1, black, above left=-3pt]{$\times$} node[black, above left=3pt]{$p$};
\draw[gray, name path=ALP] (mb1) .. controls +(150:4.2) and +(90:4.2) ..node[above right=-2pt]{$\gam'_p$} (mb1);   
\draw[blue, thick, double distance=3pt]
(mb1) -- node[above left=-3pt]{$\tau_q$} node[right=2pt]{$\tau'_q$} node[very near end, sloped, below=-4.2pt]{$\bowtie$}  +(60:2) node[name=pt2, black, above right=-3pt]{$\times$} node[black, above left=1pt]{$q$};
\draw[gray, name path] (mb1) .. controls +(90:4.2) and +(30:4.2) ..node[near start, above left=-4pt]{$\gam'_q$} (mb1); 
\fill (mb1) circle (2pt);
\draw[violet, thick, name path=LAM] ($(mb1)+(.2,4.3)$) .. controls +(300:1.5) and +(338:3.5) .. node[sloped]{$>$} node[scale=0.8, very near start, left]{$\rho^{1/2}(\lam)$} ($(mb1)+(-0.3,1.8)$) .. controls +(158:2) and +(175:1.5) .. node[very near start,sloped]{$<$} ($(mb1)+(-0.2,1.3)$)  .. controls  +(355:3.5) and +(300:1.5)   .. node[near end, sloped]{$<$}   ($(mb1)+(.7,3.8)$) ;
\fill[red, name intersections={of= ALP and LAM}]
(intersection-1) circle (2pt) node[above right=-2pt]{$x$};
\end{scope}
\node[scale=2] at (1,2){$\leadsto$};
\coordinate[label=above: $p$] (LD) at (2.5,1.8);
\coordinate[label=above: $m$] (LT) at ($(LD)+(1.5,1.5)$);
\coordinate[label=below: $m$] (MD) at ($(LD)+(3,0)$);
\coordinate[label=above: $m$] (RT) at ($(LD)+(4.5,1.5)$);
\coordinate[label=above right: $q$] (RD) at ($(LD)+(6,0)$);
\foreach \point in {LD, LT, MD, RT, RD} \fill [black] (\point) circle (2pt);
\draw[blue,thick, name path=ZZ] (LD) -- node[right]{$\gam_p$} (LT) -- node[left]{$\gam'_p$} (MD) -- node[right]{$\gam'_q$} (RT);
\draw[blue] (LT) -- node[below]{$\beta$} (RT) -- node[below left=-2.5pt]{$\gam_q$} (RD) -- node[above=-2.5pt]{$\gam_q$} (MD) -- node[below]{$\gam_p$} (LD);
\draw[violet, thick, name path=LAM2]
($(LD)+(3.3, 2)$) .. controls +(270:1) and +(170:1.3) .. node[scale=.7, very near start, above left]{$\rho^{1/2}(\lam)$} node[near start, sloped]{$>$} ($(LD)+(6,.8)$) .. controls +(350:2) and +(320:2) .. ($(LD)+(4.3,-.3)$) .. controls +(140:2) and +(45:2) .. node[very near end, sloped]{$<$}($(LD)+(0,0.5)$) .. controls +(225:1.5) and +(200:1.5) .. ($(LD)+(1.5,0.3)$) .. controls +(20:1) and +(140:1) .. ($(LD)+(4,-0.4)$) .. controls +(320:2.7) and +(350:2.7) .. ($(LD)+(6,1)$) .. controls +(170:1.3) and +(270:.8) .. node[near start, sloped]{$<$} ($(LD)+(3.7,2)$);
\fill[red, name intersections={of= ZZ and LAM2}]
(intersection-2) circle (2pt) node[above right=-1pt]{$x$};
\node at ($(LD)+(1.5,-1)$){$x\in A^-_{T,\gam'_p}(\lam)$};
\node at ($(LD)+(4,-1.8)$){no contribution to $A^\pm_{T,\gam'_q}(\lam)$};
\end{tikzpicture}
\caption{Contribution to $A^\pm_{T,\del}$ for self-folded triangles}
\label{fig:shear-fold}
\end{figure}
\begin{figure}[ht]
\begin{tikzpicture}[scale=.9]
\begin{scope} 
  \clip (-2,2) ellipse (2.8cm and 2.2cm);  
 \draw[lightgray, line width=1.0cm, line cap=round] (0,-0.05) +(-120:.5cm)
 arc[start angle=60, end angle=120, radius=3.5];
  \draw[name path = bottom1, thick] (0,0)
  arc[start angle=60, end angle=120, radius=4] node[near end, sloped]{$>$};
  \path [name path= bottom2]  (-2,0) -- (-2,3);
\fill [black, name intersections={of= bottom1 and bottom2, by=mb1}]
   (mb1) circle (2pt) node[below]{$m_1$}; 
\draw[lightgray, line width=1.0cm, line cap=round] (-4,4.05) +(60:.5cm)
    arc[start angle=240, end angle=300, radius=3.5];
  \draw[name path = top1, thick] (-4,4)
  arc[start angle=240, end angle=300, radius=4] node[near end, sloped]{$<$};
  \path [name path= top2]  (-2,4) -- (-2,2);
\fill [black, name intersections={of= top1 and top2, by=mb2}]
(mb2)  node[above]{$m_2$};
\draw[blue, thick]
(mb1) .. controls +(170:2.8) and +(190:2.8) .. node[sloped, near start]{$<$} node[near end, left]{$\alp$} (mb2);
\draw[blue, thick]
(mb1) .. controls +(10:2.8) and +(-10:2.8) .. node[sloped, near start]{$>$} node[near end, right]{$\bet$} (mb2);
\draw[blue, thick, double distance=3pt]
(mb1) -- node[very near end, sloped, below=-5pt]{$\bowtie$} node[left=-3.5pt]{$\tau_p\phantom{'}$} node[right=-2pt]{$\tau'_p$} +(90:1.4) node[name=pt2, black, above=-3pt]{$\times$} node[black, above=6pt]{$p$};
\draw[gray, name path=ALPP] (mb1) .. controls +(130:3.7) and +(50:3.7) .. node[above]{$\gam'_p$} (mb1);
\draw[violet, thick, name path=LAM] (-4.5,1.2) .. controls +(0:1) and +(180:1) .. node[sloped]{$<$} node[scale=.8, sloped, above]{$\ \rho^{-1/2}(\lam)$} (pt2);
\draw[magenta, thick, name path=LAM] (pt2) .. controls +(0:1) and +(180:1) .. node[sloped]{$<$} node[scale=.8, sloped, below]{$\rho^{-1/2}(\mu)$} (.5, 2.5);
\fill[black] (mb1) circle (2pt);
\fill[black] (mb2) circle (2pt);
\end{scope}
\node at ($(mb1)+(3,-1.5)$) {$0\in D^+_{T,\gam'_p}(\lam)$};
\node at ($(mb1)+(3,-2.0)$) {$1\in D^-_{T,\gam'_p}(\mu)$};
\begin{scope}
\coordinate (MM) at (-.5,0);
\clip ($(3.9,2.2)+(MM)$) ellipse (2.6cm and 2.3cm);
  \draw[lightgray, line width=1.0cm, line cap=round] ($(6,-0.05)+(MM)$) +(-120:.5cm)
    arc[start angle=60, end angle=120, radius=3.5];
  \draw[name path = bottom1, thick] ($(6,0)+(MM)$)
  arc[start angle=60, end angle=120, radius=4] node[near end, sloped]{$>$};
  \path [name path= bottom2]  ($(4,0)+(MM)$) -- ($(4,3)+(MM)$);
\fill [black, name intersections={of= bottom1 and bottom2, by=mb1}]
   (mb1) circle (2pt) node[below]{$m$}; 
\draw[blue, thick]
(mb1) .. controls +(170:1.8) and +(180:4) .. node[sloped, near end]{$>$} ($(4,4)+(MM)$) node[blue, below]{$\bet$}
   .. controls +(0:4) and +(10:1.8) .. node[sloped]{$<$} (mb1);
   \draw[blue, thick, double distance=3pt]
   (mb1) -- node[left=2pt]{$\tau'_p$}  node[above right=-2pt]{$\tau_p$} node[very near end, sloped, below=-4.2pt]{$\bowtie$}  +(120:2) node[name=pt1, black, above left=-3pt]{$\,\,\times$} node[black, above left=3pt]{$p$};
\draw[gray, name path=ALP] (mb1) .. controls +(150:4.2) and +(90:4.2) ..node[above right=-1pt]{$\gam'_p$} (mb1);   
\draw[blue, thick, double distance=3pt]
(mb1) -- node[above left=-3pt]{$\tau_q$} node[right=1pt]{$\tau'_q$} node[very near end, sloped, below=-4.2pt]{$\bowtie$}  +(60:2) node[name=pt2, black, above right=-3pt]{$\!\times$} node[black, above right=2pt]{$q$};
\draw[gray, name path] (mb1) .. controls +(90:4.2) and +(30:4.2) ..node[above left=-4.5pt]{$\gam'_q$} (mb1); 
\fill (mb1) circle (2pt);
\draw[magenta, thick] (pt1) .. controls +(10:1) and +(180:1)..node[sloped, near end]{$<$} node[sloped, scale=.7, near start, above]{$\ \rho^{-1/2}(\mu)$} ($(mb1)+(2.5,0.5)$);
\draw[violet, thick] ($(pt1)+(-0.2,0)$) .. controls +(180:1) and  +(0:1) .. node[scale=.7,near end, below right=-6pt]{$\rho^{-1/2}(\lam)$} node[sloped]{$<$}  ($(mb1)+(-2.7,.5)$);
\end{scope}
\begin{scope}
\coordinate (RT) at (4.8,0);
\clip ($(4.1,2.2)+(RT)$) ellipse (2.6cm and 2.3cm);
  \draw[lightgray, line width=1.0cm, line cap=round] ($(6,-0.05)+(RT)$) +(-120:.5cm)
    arc[start angle=60, end angle=120, radius=3.5];
  \draw[name path = bottom1, thick] ($(6,0)+(RT)$)
  arc[start angle=60, end angle=120, radius=4] node[near end, sloped]{$>$};
  \path [name path= bottom2]  ($(4,0)+(RT)$) -- ($(4,3)+(RT)$);
\fill [black, name intersections={of= bottom1 and bottom2, by=mb1}]
   (mb1) circle (2pt) node[below]{$m$}; 
\draw[blue, thick]
(mb1) .. controls +(170:1.8) and +(180:4) .. node[sloped, near end]{$>$} ($(4,4)+(RT)$) node[blue, below]{$\bet$}
   .. controls +(0:4) and +(10:1.8) .. node[sloped]{$<$} (mb1);
   \draw[blue, thick, double distance=3pt]
   (mb1) -- node[left=1pt]{$\tau'_p$}  node[above right=-3.5pt]{$\tau_p$} node[very near end, sloped, below=-4.2pt]{$\bowtie$}  +(120:2) node[name=pt1, black, above left=-3pt]{$\,\times$} node[black, above left=3pt]{$p$};
\draw[gray, name path=ALP] (mb1) .. controls +(150:4.2) and +(90:4.2) ..node[above right=-1pt]{$\gam'_p$} (mb1);   
\draw[blue, thick, double distance=3pt]
(mb1) -- node[above left=-3pt]{$\tau_q$} node[right=1pt]{$\tau'_q$} node[very near end, sloped, below=-4.2pt]{$\bowtie$}  +(60:2) node[name=pt2, black, above right=-3pt]{$\!\times$} node[black, above right=2pt]{$q$};
\draw[gray, name path] (mb1) .. controls +(90:4.2) and +(30:4.2) ..node[above left=-3.5pt]{$\gam'_q$} (mb1); 
\fill (mb1) circle (2pt);
\draw[magenta, thick] (pt2) .. controls +(10:1) and +(180:1)..node[sloped, near end]{$<$} node[scale=.7, very near end, below left=-3pt]{$\ \rho^{-1/2}(\mu)$} ($(mb1)+(2.5,0.5)$);
\draw[violet, thick] (pt2) .. controls +(180:1) and  +(0:1) .. node[sloped, scale=.7,near start, above =2pt]{$\rho^{-1/2}(\lam)$} node[sloped, near end]{$<$}  ($(mb1)+(-2.7,.5)$);
;
\end{scope}
\node at ($(mb1)+(0,-1.5)$) {$0\in D^+_{T,\gam'_q}(\lam)$};
\node at ($(mb1)+(0,-2.0)$) {$1\in D^-_{T,\gam'_q}(\mu)$};
\end{tikzpicture}
\caption{Contributions to $D^\pm_{T,\gam'_r}$ for $r\in\Pu$}
\label{fig:shearD}
\end{figure}
\begin{Def} 
For a signature zero tagged triangulation $T$ of $\bSu$ and a tagged arc  $(\gam, c)\in T$ with $c=(c_1, c_2)\in S(\gam)$,  the corresponding component of the shear coordinate   $\Sh_T(\lam, l)$ with
respect to $T$   of a marked curve $(\lam, l)\in\cLC^*(\bSu)$ (with $l=(l_1, l_2)\in S(\lam)$) is
\begin{align*}  
    \Sh_T(\lam, l)_{(\gam,c)} &
    := (\abs{A^-_{T,\bar{\gam}}(\lam)}-\abs{A^+_{T,\bar{\gam}}(\lam)})\times \wt(l)\\
&  + \sum_{t\in D^-_{T,\bar{\gam}}(\lam)} d_2(c_2, -l_{t+1})
  - \sum_{t\in D^+_{T,\bar{\gam}}(\lam)} d_2(l_{t+1}, c_2)\quad\in\ZZ,
\end{align*}
where we used the above notation as well as the function $d_2$ from Definition~\ref{def:comb-inv}.
In particular, $\bar{\gam}\in T^{\circ}$ is the completion of the arc $\gam$. 
\end{Def}

\begin{Rem}\label{rem:these-coords-are-dual-to-FST}
The shear coordinates we work with are in some sense dual to the shear coordinates defined in~\cite[Sec.~13]{FT18}, see Subsection \ref{subsec:dual-shear-coords-arb-triangs}, and in particular, Example~\ref{ex:FT-shear-coords-VS-dual-shear-coords}, for more details.
\end{Rem}

\begin{Prop} \label{prop: g=sh}  Let $T$ be a signature zero tagged triangulation of $\bSu$ with
skewed-gentle polarized  quiver $\uQ(T)$. 
Then, for each $(\bx,s)\in\DAdm^*(\uQ(T))$, we have
\[
  \bg_{\uQ(T)}(\bx,s) = \Sh_T(\fra^*(\bx,s)),
\]
where we used the above-mentioned identification between the arcs of $T$ and the vertices $\tQ_0$.
In other words, the combinatorial  g-vector of $(\bx,s)$ from
Definition~\ref{def:comb-inv} can be identified with the shear coordinates
of the corresponding marked curve $\fra^*(\bx,s)$. 
\end{Prop}

\begin{proof}
Suppose first that $\bx=-\bs_v$  for some $v\in Q(T)_0$.
We have then $\fra(-\bs_v)=\gam_v$ by definition.
It is easy to see that in case 
$v\in Q(T)_0^{\ord}$ we obtain 
\[
  \abs{A^-_{T, \widebar{\gam}_{v'}}(\gam_v)}=\del_{v',v}
\]
for all $v'\in Q(T)_0$.
Moreover, we have in this case $A^+_{T,\widebar{\gam}_v'}(\gam_v)=\emptyset$ for all
$v'\in Q(T)_0$ and trivially  $D^\pm_{T,\widebar{\gam}_p}(\gam_v)=\emptyset$
for all $p\in\Pu=Q(T)_0^{\spe}$.
On the other hand, for $v\in Q(T)_0^{\spe}=\Pu$ we have
\[
  \abs{D^-_{T, \widebar{\gam}_p}(\gam_v)}=\del_{p,v}
\]
for all $p\in\Pu$. Moreover, it is easy to see that in this case
$D^+_{T,\widebar{\gam}_p}(\gam_v)=\emptyset$ for all $p\in\Pu$ and
$A^\pm_{T,\widebar{\gam}_{v'}}=\emptyset$ for all $v'\in Q(T)_0$.
It follows now from the definitions that
\[
  \Sh_T(\fra^*(-\bs_v,u))_{(\gam_q,c)}=\Sh_T(\gam_v,u)_{(\gam_q,c)}=\del_{(q,c),(v,u)}=\bg_{\uQ(T)}(-\bs_v,u)_{q,c_2}.
\]
Thus, from now on, we may assume $\bx\in\Adm(\uQ(T))$. 
After comparing the  definitions, we see that, with our convention in place,
we only have to show 
\begin{alignat*}{2}
  \abs{A^\pm_{T,\widebar{\gam}_v}(\fra(\bx))} &=\abs{A^\pm_v(\bx)}
  &&\quad\text{for each } v\in Q(T)_0, \text{ and}\\
  D^\pm_{T,\widebar{\gam}_p}(\fra(\bx))  &=D^\pm_p(\bx)
  &&\quad\text{for each } p\in Q(T)_0^{\spe}=\Pu. 
\end{alignat*}
for any $\bx\in\Adm(\uQ(T))$. In view of the definition of $\fra_f$ and the
orientation of the arrows of $\uQ(T)$, which are displayed in
Figure~\ref{fig:shearAD}, it is clear that we have indeed bijections between
the elements of $A^\pm_{T,\widebar{\gam}_v}(\fra(\bx))$ and the proper
sources resp.~sinks $A^\pm_v(\bx)$ of $H(\tau_f(\bx))$
(see Definition~\ref{def:ADpm}) for all $v\in\uQ(T)_0$.

On the other hand, suppose that for example 
$\tau_f(\bx)=\II^{-1}_{(p,-1)}\bx'$ for some $p\in\Pu\subset\uQ(T)_0$.  
Then, by definition, $0\in D^+_p(\bx)$ if and only if 
$1\in H(\tau_f(\bx))_0$ is a source in $\uH(\tau_f(\bx))$ after removing the loop $\eta_0$.
The other three possibilities are similar.
\end{proof}

\begin{Rem}
  The above argument is very similar to the considerations of
  the second author's Ph.D. Thesis~\cite[Thm.~10.0.5]{LF10}.
  See also~\cite[Sec.~10.10]{GLFS22}.  
\end{Rem}

It is now easy to prove a crucial special of our main result, Theorem~\ref{thm1}. This special case deals with tagged signature zero triangulations of $\bSu$  in the sense of~\cite[Def.~9.1]{FST08}.
Let $\uQ=\uQ(T)$ be the corresponding skewed-gentle polarized quiver, see Section~\ref{ssec:tagged3ang}.  In this situation, $A(T)=\Ka\uQ(T)$.

\begin{Thm} \label{thm:sign0}
Let $\bSu$ be a marked surface with $\dSu\neq\emptyset$. For each signature zero tagged triangulation $T$ of $\bSu$ with corresponding skewed-gentle Jacobian algebra
$A(T)=\cP_\Ka(Q(T), W(T))$, we have an isomorphism   
of tame, framed  partial KRS-monoids
\[
  \pi_T\df (\MSW(\bSu),\Intn^*,+, \Sh_T) \ra
  (\DecIrr^\tau(A(T)), e_{A(T)},\oplus, g_{A(T)}),
\]
which intertwines dual shear coordinates with generic $g$-vectors.
\end{Thm}

\begin{proof}
By Theorem~\ref{thm:SkewedGIso} we have an isomorphism of tame, framed and free partial KRS-monoids
\begin{equation} \label{eq:KRSIso1}
\widetilde{M}^o\df \KRS([\DAdm^*_\tau(\uQ)]/_{\simeq}, e_\uQ, \bg_\uQ)\ra
    (\DecIrr^\tau(\Ka\uQ), e_{\Ka\uQ}, \oplus, \bg_{\Ka\uQ}).
\end{equation}
Furthermore, by Theorem \ref{thm:intersec} and Proposition \ref{prop: g=sh}  we have the following bijection
\[  
\fra^*\df [\DAdm^*(\uQ)]/_{\simeq}\ra \cLC^*(\bSu)/_{\simeq},~(\bx,s)\mapsto (\fra(x),s)
\]
which is compatible with the respective equivalence relation
$\simeq$, and satisfies
  \[\label{e=int}
    e_\uQ((\bx,s), (\by,t)))=\Intn^*(\fra^*(\bx,s), \fra^*(\by,y))
    \]
    and
    \[
  \bg_{\uQ}(\bx,s) = \Sh_T(\fra^*(\bx,s))
\]
for all $(\bx,s), (\by,t)\in [\DAdm^*(\uQ)]/_{\simeq}$. \newline

Note that $\fra^*$ restricts to the following bijection
\[
    [\DAdm^*_\tau(\uQ)]/_{\simeq} \ra  \cLC^*_\tau(\bSu)/_{\simeq}.
\]
which in turn gives us trivially the following isomorphism of tame, framed and free partial KRS-monoids
 \begin{equation} \label{eq:KRSIso2}
    \MSW(\bSu):=\KRS(\cLC^*_\tau(\bSu)/_{\simeq}, \Intn^*, \Sh_T)\ra
        \KRS([\DAdm^*_\tau(\uQ)]/_{\simeq}, e_\uQ, \bg_\uQ),  f\mapsto f\circ\fra^*.
 \end{equation}
see Definition~\ref{def:MSWL} and Example~\ref{expl:KRS1}.
The composition of the isomorphisms in~\eqref{eq:KRSIso1} and~\eqref{eq:KRSIso2} is the
claimed isomorphism.
\end{proof}

\subsection{Dual shear coordinates with respect to arbitrary tagged triangulations}\label{subsec:dual-shear-coords-arb-triangs}

Let $T$ be a tagged triangulation of $\bSu$, and let $\lam$ be either a tagged arc or a simple closed curve on $\bSu$. Essentially following \cite{FT18}, we shall define the vector $\Sh_T(\lam)=(\Sh_T(\lam)_{\gam})_{\gam\in T}$ of \emph{dual shear coordinates} of $\lam$ with respect to $T$. We make a small notational compromise, so that the notation in this subsection lies between the notation used by Fomin--Shapiro--Thurston and Fomin--Thurston and the notation we have been using in the rest of the paper.

\begin{case} Suppose first that $T$ has non-negative signature $\delta_T$, i.e., that $T$ has non-negative signature at each puncture. Thus, at any given puncture $p$, either all tagged arcs in $T$ incident to $p$ are tagged plain, or exactly two tagged arcs in $T$ are incident to $p$, their underlying ordinary arcs being isotopic to each other, their tags at $p$ differing from one another, and the tags at their other endpoint being plain. 

As shown by Fomin--Shapiro--Thurston, there is an ideal triangulation $T^\circ$ canonically representing $T$. Namely, for each puncture $p\in\mathbb{P}$ at which $T$ has signature zero, i.e., $\delta_T(p)=0$, let $i_p,j_p,$ be the two tagged arcs in $T$ that are incident to $p$, with $i_p$ tagged \emph{plain} at $p$, and $j_p$ tagged \emph{notched} at $p$. Then $T^\circ:=\{\gam^\circ\ | \ \gam\in T\}$ is obtained from $T$ by setting $\gam^\circ:=\gam$ for every $\gam\in T$ both of whose ends are tagged \emph{plain}, and by replacing each $j_p$ with a loop $j_p^\circ$ closely surrounding $i_p^\circ$, for $i_p$ and $j_p$ as above. Thus, the corresponding ordinary arcs $i_p^\circ,j_p^\circ\in T^\circ$ form a self folded triangle, $i_p^\circ$ being the folded side, and $j_p^\circ$ being the loop closely surrounding $i_p^\circ$. See \cite[Sections 9.1 and 9.2]{FST08}.

If $\lam$ is a simple closed curve, set $L:=\lam$. Otherwise, let $L$ be the curve obtained from $\lam$ by modifying its two ending segments according to the following rules:
\begin{itemize}
\item in sync with the definition of $\rho^{1/2}(\lam)$, any endpoint incident to a marked point in the boundary is slightly slided along the boundary segment lying immediately to its right as in Figure \ref{Fig_rhohalf} (here, we stand upon the surface using its orientation, and look from the marked point towards the interior of surface, note that we use the orientation of $\Su$ to determine what is right and what is left);
\begin{figure}[ht]
\begin{tikzpicture}[scale=0.7]
\begin{scope} 
  \clip (-2,1.1) ellipse (2.8cm and 1.8cm);  
 \draw[lightgray, line width=1.0cm, line cap=round] (0,-0.2) +(-120:.5cm)
 arc[start angle=60, end angle=120, radius=3.5];
  \draw[name path = bottom1, thick] (0,0)
  arc[start angle=60, end angle=120, radius=4] node[very near end, sloped]{$>$};
  \path [name path= bottom2]  (-2.8,0) -- (-2.8,3);
\fill [black, name intersections={of= bottom1 and bottom2, by=mb1}]
   (mb1) circle (2pt) node[below]{$m_1$};   
  \path [name path= bottom3]  (-0.9,0) -- (-0.9,3);
\fill [black, name intersections={of= bottom1 and bottom3, by=mb2}]
   (mb2) circle (2pt) node[below]{$m_2$}; 
\draw[blue, thick]  (-2.0,2.5) .. controls +(280:1) and +(110:1) .. (mb1);   
\end{scope}
\draw[->, gray,line width=0.8mm, decorate, decoration={snake, segment length=4mm, amplitude = 2mm, post length=2mm}] (0.8,1) -- +(1.7,0);
\coordinate  (rs) at (7.5,0);
\begin{scope} 
  \clip (-2,1.1)+(rs) ellipse (2.8cm and 1.8cm);  
 \draw[lightgray, line width=1.0cm, line cap=round] (7.5,-0.2) +(-120:.5cm)
 arc[start angle=60, end angle=120, radius=3.5];
  \draw[name path = bottom4, thick] (0,0)+(rs)
  arc[start angle=60, end angle=120, radius=4] node[very near end, sloped]{$>$};
  \path [name path= bottom5]  (4.7,0) -- (4.7,3);
\fill [black, name intersections={of= bottom4 and bottom5, by=mb3}]
   (mb3) circle (2pt) node[below]{$m_1$};   
  \path [name path= bottom6]  (6.5,0) -- (6.5,3);
\fill [black, name intersections={of= bottom4 and bottom6, by=mb4}]
   (mb4) circle (2pt) node[below]{$m_2$}; 
\path [name path=bottom7] (5.3,0) -- (5.3,3);
\fill [blue, name intersections={of= bottom4 and bottom7, by=mb5}] (mb5) circle(1pt);
\draw[blue, thick]  (-2.0,2.5)+(rs) .. controls +(290:1) and +(120:1) .. (mb5);   
\end{scope}
\end{tikzpicture}
\caption{Slightly sliding endpoints lying on the boundary}
\label{Fig_rhohalf}
\end{figure}

\item any ending segment of $\lam$ tagged \emph{plain} at a puncture $q$ is replaced with a non-compact curve segment spiralling towards $q$ in the clockwise sense as  in Figure~\ref{Fig_untagged}.
\begin{figure}[ht]
\begin{tikzpicture}[scale=0.6]
\draw[blue, thick] (-5,0) .. controls +(20:2) and +(200:2) ..(-2,0) node[black]{$\times$} node[right, black ]{$q$};  
\draw[->, gray,line width=0.8mm, decorate, decoration={snake, segment length=4mm, amplitude = 2mm, post length=2mm}] (-1,0) -- (1,0);
\draw[blue, thick] (2,0) .. controls +(22:1) and +(210:1) ..(3.7,1.54); 
\draw[blue, thick, domain=4:30,smooth,variable=\t,samples=300]
plot ({3*exp(-0.1*\t)*cos(\t r)+5}, {3*exp(-0.1*\t)*sin(-\t r)});
\draw (5,0) node[black]{$\times$} node[right=1.2mm, black ]{$q$};  
\end{tikzpicture}
\caption{Tagged  plain to clockwise spiral}\label{Fig_untagged}
\end{figure}

\item any ending segment of $\lam$ tagged \emph{notched} at a puncture $q$ is replaced with a non-compact curve segment spiralling towards $q$ in the counter-clockwise sense, as in Figure~\ref{Fig_tagged}
\begin{figure}[ht]
\begin{tikzpicture}[scale=0.6]
\draw[blue, thick] (-5,0) .. controls +(20:2) and +(200:2) ..
node[very near end, sloped]{$\bowtie$}(-2,0) node[black]{$\times$} node[right, black ]{$q$};  
\draw[->, gray,line width=0.8mm, decorate, decoration={snake, segment length=4mm, amplitude = 2mm, post length=2mm}] (-1,0) -- (1,0);
\draw[blue, thick] (2,0) .. controls +(22:1) and +(150:1) ..(3.7,-1.54); 
\draw[blue, thick, domain=4:30,smooth,variable=\t,samples=300]
plot ({3*exp(-0.1*\t)*cos(\t r)+5}, {3*exp(-0.1*\t)*sin(\t r)});
\draw (5,0) node[black]{$\times$} node[right=1.2mm, black ]{$q$};  
\end{tikzpicture}
\caption{Tagged notched to anti-clockwise spiral}\label{Fig_tagged}
\end{figure}
\end{itemize}

Take an arc $\gam\in T$. In order to define the shear coordinate 
$\Sh_T(\lam)_{\gam}$, we need to consider two subcases.

\begin{subcase}
Suppose that the ordinary arc $\gam^{\circ}\in T^\circ$ is not the folded side of a self-folded triangle of $T^\circ$. Then $\gam^\circ$ is contained in exactly two ideal triangles of $T^\circ$, and the union $\overline{\lozenge}_{\gam^\circ}$ of these two triangles is either a quadrilateral (if $\gamma^\circ$ does not enclose a self-folded triangle) or a digon (if $\gamma^\circ$ encloses a self-folded triangle). In any of these two situations, the complement in $\overline{\lozenge}_{\gam^\circ}$ of the union of the arcs belonging to $T^\circ\setminus\{\gam^\circ\}$ can be thought to be an open quadrilateral $\lozenge_{\gam^\circ}$ in which $\gamma^\circ$ sits as a diagonal. The shear coordinate $\Sh_T(\lam)_{\gam}$ is defined to be the number of segments  of $\lozenge_{\gam^\circ}\cap L$ that form the shape of a letter $Z$ when crossing $\gam^\circ$ minus the  number of segments  of $\lozenge_{\gam^\circ}\cap L$ that form the shape of a letter $S$ when crossing $\gam^\circ$.
\end{subcase}

\begin{subcase}
Suppose that the ordinary arc $\gam^{\circ}\in T^\circ$ is the folded side of a self-folded triangle of $T^\circ$. Then there is a puncture $p\in\Pu$ such that $\gam=i_p$ and $\gam^\circ=i_p^\circ$. Define an auxiliary curve $L'$ by replacing any ending spiral of $L$ that goes towards the puncture $p$, with a non-compact curve segment spiralling towards $p$ in the opposite sense. Thus, the ending spirals not going towards $p$ remain unchanged, and $L'=L=\lam$ if $\gam$ is a simple closed curve. The shear coordinate $\Sh_T(\lam)_{\gam}$ is defined to be the number of segments  of $\lozenge_{j_p^\circ}\cap L'$ that form the shape of a letter $Z$ when crossing $j_p^\circ$ minus the  number of segments  of $\lozenge_{j_p^\circ}\cap L'$ that form the shape of a letter $S$ when crossing $j_p^\circ$. That is,
$$
\Sh_T(\lam)_{i_p}:=\Sh_T(\lam')_{j_p}.
$$
where $\lam'$ is obtained from $\lam$ by switching the tags of 
$\lam$ at the puncture $p$.
\end{subcase}
\end{case}

\begin{case} Suppose that $T$ is an arbitrary tagged triangulation of $\bSu$. Then $\delta_T^{-1}(-1)$ is the set of punctures at which $T$ has negative signature. Set $T'$ to be the tagged triangulation obtained from $T$ by changing from notched to plain all the tags incident to punctures 
in~$\delta_T^{-1}(-1)$. Thus, $T'$ is a tagged triangulation of signature zero, so dual shear coordinates with respect to $T'$ have already been defined. Set
$$
\Sh_T(\lam):=\Sh_{T'}(\lam'),
$$
where $\lam'$ is obtained from $\lam$ by switching all the tags of $\lam$ at the punctures belonging to the set $\delta_T^{-1}(-1)$.
\end{case}

\begin{Expl}\label{ex:FT-shear-coords-VS-dual-shear-coords} In Figure \ref{Fig_dualShearCoords_vs_FSTShearCoords},
        \begin{figure}[!h]
                \centering
                \includegraphics[scale=.1]{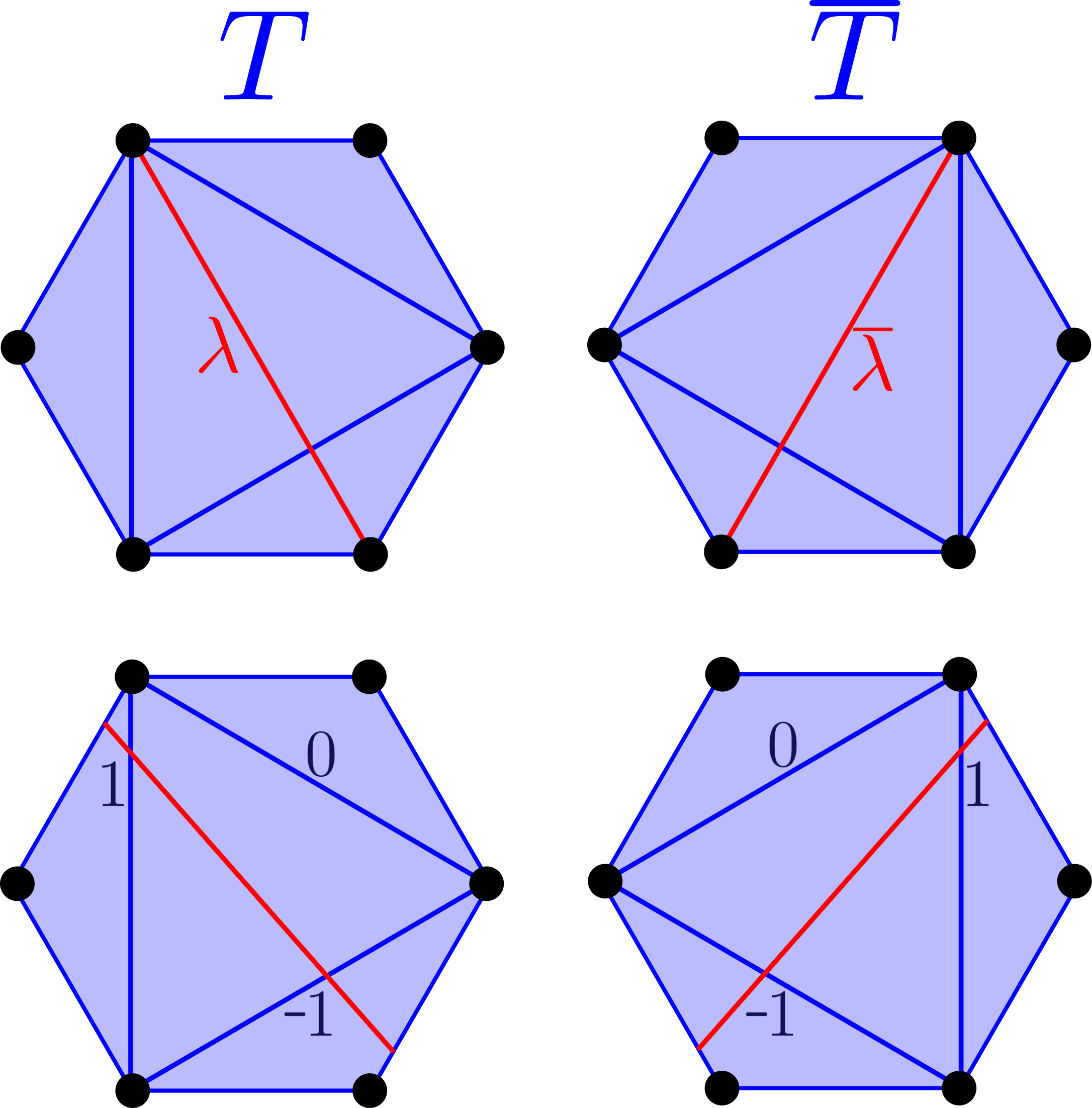}
                \caption{Left: computation of the vector of dual shear coordinates $\Sh_T(\lam)$. Right: computation of the vector of Fomin--Thurston's shear coordinates $\Sh_{\overline{T}}^{\operatorname{FST}}(\overline{\lam})$.}\label{Fig_dualShearCoords_vs_FSTShearCoords}
        \end{figure}
the reader can glimpse the relation between the dual shear coordinates we have defined above, and the shear coordinates used by Fomin--Thurston, namely,
$$
\Sh_T(\lam)=\Sh_{\overline{T}}^{\operatorname{FST}}(\overline{\lam}),
$$
where $\Sh_T(\lam)$ is the vector of dual shear coordinates we have defined above, $\overline{T}$ and $\overline{\lambda}$ are the images of $T$ and $\lambda$ in the surface obtained as the \emph{mirror image} of $\bSu$, and $\Sh_{\overline{T}}^{\operatorname{FST}}(\overline{\lam})$ is the vector of shear coordinates of Fomin--Thurston, \cite[Definition 13.1]{FT18}.
\end{Expl}

Denote by $B_{\operatorname{FST}}(T)$ the skew-symmetric matrix associated to $T$ by Fomin--Shapiro--Thurston \cite[Definitions 4.1 and 9.6]{FST08}, and denote by $B_{\operatorname{DWZ}}(Q)$ the skew-symmetric matrix associated by Derksen--Weyman--Zelevinsky to an arbitrary 2-acyclic quiver $Q$ \cite[Equation (1.4)]{DWZ2} and \cite[Section 2 and Equation (7.1)]{DWZ1}. Since we have chosen to define $Q(T)$ by drawing the arrows in the clockwise sense within each triangle of $T$ (i.e., following the convention of \cite{LF-QPsurfs1} and not the convention from \cite{ABCP}), we have
$$
B_{\operatorname{DWZ}}(Q(T))=-B_{\operatorname{FST}}(T).
$$
Furthermore, for dual shear coordinates, \cite[Theorem 13.5]{FT18} can be restated as follows.

\begin{Thm}\label{thm:behavior-of-dual-shear-coords-under-matrix-mutation}
    Let $T$ and $T'$ be tagged triangulations of $\bSu$ related by the flip of a tagged arc $k\in T$, and let $\lam$ be either a tagged arc or a simple closed curve on $\bSu$. Then
    $$
    \left[\begin{array}{c}-B_{\operatorname{FST}}(T')\\
    \Sh_{T'}(\lam)\end{array}\right]
    =\mu_k\left(
    \left[\begin{array}{c}-B_{\operatorname{FST}}(T)\\
    \Sh_T(\lam)\end{array}\right]\right).
    $$
\end{Thm}

Denote the entries of $B_{\operatorname{FST}}(T)$ as $b_{ij}$, and the entries of $-B_{\operatorname{FST}}(T)$ as $\beta_{ij}$. Since $B_{\operatorname{FST}}(T)$ and $-B_{\operatorname{FST}}(T)$ are transpose to each other, Theorem \ref{thm:behavior-of-dual-shear-coords-under-matrix-mutation} states in particular that
\begin{align}\label{eq:unraveled-recursion-satisfied-by-dual-shear-coords}
    \Sh_{T'}(\lam)_{\gam} &= 
\begin{cases}
-\Sh_T(\lam)_{\gam} & \text{if $\gamma=k$};\\
\Sh_T(\lam)_{\gam}+\sgn(\Sh_T(\lam)_k)[\Sh_T(\lam)_k\beta_{k\gamma}]_+ & \text{if $\gamma\neq k$};
\end{cases}\\
\nonumber &=
\begin{cases}
-\Sh_T(\lam)_{\gam} & \text{if $\gamma=k$};\\
\Sh_T(\lam)_{\gam}+\sgn(\Sh_T(\lam)_k)[b_{\gamma k}\Sh_T(\lam)_k ]_+ & \text{if $\gamma\neq k$}.
\end{cases}
\end{align}

\subsection{Proof of Theorem~\ref{thm1}}
Since $\dSu\neq\emptyset$, it is easy to see that $\bSu$ has at least one tagged triangulation of signature zero. For such a triangulation, we can apply Theorem~\ref{thm:sign0}. 

Next, suppose that for a tagged triangulation $T=(\tau_1, \tau_2, \ldots, \tau_n)$ of $\bSu$
we have already the requested isomorphism
\[
\pi_T\df\MSW(\bSu)\ra\DecIrr^\tau(A(T)),
\]
which intertwines the dual shear coordinates $\Sh_T$ with the generic $g$-vector
$\bg_{A(T)}$.  Consider the triangulation $T'$ which is obtained from $T$ by flipping an arc $\tau_k$. 
By~\cite[Theorem~8.1]{LF16} the QP-mutation $\mu_k(Q(T), W(T))$ is right equivalent
to $(Q(T'), W(T'))$.  This yields by Proposition~\ref{prp:JaMut}~(e) an isomorphism
\[
\tilde{\mu}_k\df \DecIrr^\tau(A(T))\ra\DecIrr^\tau(A(T'))
\]
of partial KRS-monoids. By Proposition~\ref{prp:JaMut}~(c) we have 
\[
\bg_{A(T')}(\tilde{\mu}_k(Z))=\gam^{B}_k(\bg_{A(T)}(Z))
\]
for all $Z\in\DecIrr^\tau(A(T))$, where $B=-B_{\operatorname{DWZ}}(Q(T))$ is the matrix defined in terms of $Q(T)$ at the beginning of 
Subsection~\ref{subsec:Jac-algs-and-muts}, and $\gam^{B}_k$ is the piecewise linear, bijective map from Equation~\eqref{eq:gproj-transf}. On the other hand, by \eqref{eq:unraveled-recursion-satisfied-by-dual-shear-coords} for every $\lam\in\MSW(\bSu)$ we have
\[
\Sh_{T'}(\lam)=\gam^{B}_k(\Sh_T(\lam)).
\]
Thus, $\pi_{T'}=\tilde{\mu}_k\circ\pi_T$ is the desired isomorphism
for the triangulation $T'$.

Now, since the boundary of $\Sigma$ is not empty, by~\cite[Prop.~7.10]{FST08}, 
the exchange graph for tagged triangulations is connected. Thus,
we can conclude by induction, that we have the desired isomorphism for all tagged triangulations.

\subsection*{Acknowledgments}
The first author acknowledges partial support from PAPIIT grant IN116723
(2023-2025). The work of this paper originated during the third author’s back-to-back visits at IMUNAM courtesy of the first author’s CONACyT-239255 grant and a DGAPA postdoctoral fellowship. He is grateful for the stimulating environment IMUNAM fostered, and for the additional generous support received from the second author’s grants: CONACyT-238754 and a C\'{a}tedra Marcos Moshinsky.

Finally, we wish to thank an anonymous referee for careful reading and numerous suggestions, which have improved the readability of our paper.
\section*{Index of Notations}
\printglossaries

\end{document}